\documentclass[reqno, colorBG]{amsart}
\usepackage{amssymb}
\usepackage{amsmath}
\usepackage[mathscr]{euscript}

\usepackage[small]{caption}
\usepackage{psfrag}
\usepackage{graphicx}
\graphicspath{ {images/} }

\usepackage{tikz}
\usepackage{color}

\makeatletter
\@addtoreset{equation}{section}
\makeatother

\newtheorem{theorem}{Theorem}[section]
\newtheorem{lemma}[theorem]{Lemma}
\newtheorem{proposition}[theorem]{Proposition}
\newtheorem{corollary}[theorem]{Corollary}

\newtheorem{remark}[theorem]{Remark}

\newcounter{as}[section]

\newtheorem{asser}[as]{Assertion}

\newcommand{\mf}[1]{{\mathfrak #1}}
\newcommand{\mb}[1]{{\mathbf #1}}
\newcommand{\bb}[1]{{\mathbb #1}}
\newcommand{\bs}[1]{{\boldsymbol #1}}
\newcommand{\ms}[1]{{\mathscr #1}}

\newcommand{\<}{\langle}
\renewcommand{\>}{\rangle}
\renewcommand{\Cap}{{\rm cap}}

\begin{document}

\title[Metastability of diffusion processes on the one-dimensional
torus] {Metastability of one-dimensional, non-reversible diffusions
  with periodic boundary conditions.}

\begin{abstract}
  We consider small perturbations of a dynamical system on the
  one-dimensional torus. We derive sharp estimates for the pre-factor
  of the stationary state, we examine the asymptotic behavior of the
  solutions of the Hamilton-Jacobi equation for the pre-factor, we
  compute the capacities between disjoint sets, and we prove the
  metastable behavior of the process among the deepest wells following
  the martingale approach. We also present a bound for the probability
  that a Markov process hits a set before some fixed time in terms of
  the capacity of an enlarged process.
\end{abstract}

\author{C. Landim, I. Seo}

\address{\noindent IMPA, Estrada Dona Castorina 110, CEP 22460 Rio de
  Janeiro, Brasil and CNRS UMR 6085, Universit\'e de Rouen, France.
  \newline e-mail: \rm \texttt{landim@impa.br} }

\address{\noindent Department of Mathematical Sciences, Seoul National University, Seoul, Korea. 
 \newline
  e-mail: \rm \texttt{insuk.seo@snu.ac.kr} }


\maketitle

\section{Introduction}
\label{sec-1}

Variational formulae for the capacity between two sets have been
derived recently in the context of continuous time Markov chains and
diffusions \cite{gl,slo,lms}. These formulae were used to prove the
metastable behavior of asymmetric condensing zero-range processes
\cite{l2}, random walks in a potential field \cite{ls1}, mean field
Potts model \cite{ls2}, and to derive the Eyring-Kramers formula for
the transition time of non-reversible diffusions \cite{lms}.

To estimate the capacity through the variational formulae alluded to
above, one needs to know explicitly the stationary state. This
property is shared by all the dynamics mentioned in the previous
paragraph, where the invariant measures are the equilibrium states of
the reversible version of the dynamics. Usually, however, the
stationary states of non-reversible Markovian dynamics are not known
explicitly.

It is possible, nonetheless, to derive through dynamical large
deviations methods a formula for the quasi-potential of non-reversible
dynamics and estimates for the stationary state with exponentially
small errors \cite{fw}.  A natural development of this approach
consists in using potential theory to get sharper bounds of the
stationary state, that is, to provide precise estimates for the
first-order term in the expansion of the quasi-potential, the
so-called pre-factor.

For one-dimensional diffusion processes with periodic boundary
conditions,
\begin{equation}
\label{03}
dX_\epsilon (t) \;=\; b(X_\epsilon (t))\, dt \;+\; \sqrt{2\epsilon} \,
dW_t\;,
\end{equation}
where $b: \bb T \to \bb R$ is a smooth drift, $\epsilon>0$ a small
parameter and $W_t$ the Brownian motion on the one-dimensional torus
$\bb T=[0,1)$, one may derive sharp estimates for the pre-factor due
to an explicit formula for the stationary state obtained by Faggionato
and Gabrielli \cite{fg1}. This estimate and the precise bounds for the
capacity between two wells constitute the first main result of the
article. 

The pre-factor of the stationary state solves a Hamilton-Jacobi
equation.  We take advantage of the explicit formulae to examine the
asymptotic behavior of the solution of the Hamilton-Jacobi equation in
the hope that these results might give some insight on the behavior of
the pre-factor in higher dimensions.

The second main result of this article provides an extension to
diffusion processes of the martingale approach proposed in
\cite{bl2,bl7, bl9} to derive the metastable behavior of Markov
chains. The main difficulty in applying this method to diffusions lies
in the fact that the martingale approach requires an analysis of the
trace of the process on the wells. While for Markov chains the trace
process is still a Markov chain [with long range jumps], for
diffusions the trace becomes a singular diffusion with jumps along the
boundary of the wells, a dynamics very different from the original one
and difficult to analyze. 

We present in this article an entirely new approach inspired by
results in partial differential equations obtained by Evans, Tabrizian
and Seo, Tabrizian \cite{et, st}. Here is the idea. Denote by
$\ms E_i$, $1\le i\le n$, the wells, and let $G$ be a function defined
on the entire space and which is constant [with possibly different
values] in each well. Denote by $L_\epsilon$ the generator and by
$F_\epsilon$ the solution of the Poisson equation
$L_\epsilon F_\epsilon = G$. Assume that for all such functions $G$
the solution $F_\epsilon$ is uniformly bounded and is asymptotically
constant in each well.  We prove in Section \ref{sec3} that the
convergence in law of the projection of the trace process on the wells
follows from the previous property of the solutions of the Poisson
equation.

This new way of deriving the metastable behavior of a Markov chain is
applied here to small perturbations of the dynamical system
\eqref{03}. It also provides the first example where the reduced
chain, which describes the asymptotic dynamics among the wells, is a
irreducible, non-reversible Markov chain.

The third main result of the article consists in a bound on the
probability that the hitting time of a set is less than or equal to a
constant in terms of capacities. In view of the variational formulae
for the capacity, this result provides a general method to obtain
upper bounds for the probability of an event which appears 
in many different contexts.

We conclude this introduction with some historical remarks and a
description of the article.  The convergence of the order parameter,
in the sense of finite-dimensional distributions, of small
perturbations by reversible Gaussian noises of dynamical systems has
been proved by Sugiura in \cite{su}. Imkeller and Pavlyukevich
obtained a similar result in one-dimension when the Brownian motion is
replaced by a L\'evy process.  More recently, Bouchet and Reygner
\cite{br} derived a formula for the transition time between two wells
for non-reversible diffusions. A rigorous proof of this result is
still an open problem.

The paper is organized as follows. In Section \ref{sec01}, we
introduce the model and the main assumptions on the drift $b$. In
Section \ref{sec02}, we present the main results of the article in the
case of two wells.  In Section \ref{sec03}, we introduce the notion of
valleys and landscapes used throughout the article. In Sections
\ref{sec0} and \ref{sec1}, we derive sharp asymptotic estimates for
the pre-factor and for the capacity between two wells. In Section
\ref{sec3}, we prove the metastable behavior of the process by showing
that the projection of the trace process on the wells converges in an
appropriate time scale to a finite-state Markov chain. In Section
\ref{sec7}, we prove a bound on the probability that a certain set is
attained before a fixed time in terms of the capacities of an enlarged
process. Finally, in Section \ref{sec5}, we prove that the solutions
of certain Poisson equations are asymptotically constant on the wells.

\section{The model} 
\label{sec01}

We introduce in this section the model, the main assumptions and 
known results.

\smallskip\noindent{\bf \ref{sec01}.1. The diffusion process.} Let
$\bb T = [0,1)$ be the one-dimensional torus of length $1$.  Consider
a continuous vector field $b: \bb T \to \bb R$. Throughout this
article, we assume that
$b$ fulfills the following conditions: \\
\smallskip\noindent ({\bf H1}) The closed set $\{\theta\in\bb T:
b(\theta) = 0\}$ has a finite number of connected components, denoted
by $I_j = [l_j,r_j]$, $1\le j\le p$, for $0<l_1\le r_1< l_2 \le \cdots
<l_p\le r_p <1$. Some of these intervals may be
degenerate, as we do not exclude the possibility that $l_j=r_j$. \\
\smallskip\noindent({\bf H2}) $b$ is of class $C^2$ in the set
$\bb T \setminus I$, where $I=\cup_{1\le j\le p} \, I_j$. \\
\smallskip\noindent({\bf H3}) If $(c,d)$ is a connected component of
$\bb T \setminus I$, then $b'(c+)\not = 0$ and $b'(d-)\not = 0$. When
the interval $[l_j,r_j]$ is degenerate, $l_j=r_j$, the left and right
derivatives of $b$ at $r_j$ may be different: it may happen that
$b'(r_j-)\not = b'(r_j+)$. However, both right and left derivatives do
not vanish.

The generator of the diffusion \eqref{03}, denoted by $L_\epsilon$, is
given by
\begin{equation*}
(L_\epsilon f)(\theta)\;=\; b(\theta) f'(\theta) \;+\; \epsilon \, f''(\theta)\;.
\end{equation*}
If the average drift vanishes,
\begin{equation*}
\int_{\bb T} b(\theta)\, d\theta \;=\; 0\;,
\end{equation*}
there exists a potential $U:\bb T\to \bb R$ such that $b(\theta) = -\,
U'(\theta)$.  In this case the stationary measure is given by
$Z^{-1}_\epsilon \exp\{- U(\theta)/\epsilon\} \, d\theta$ for a suitable
normalization factor $Z_\epsilon^{-1}$ and the process is reversible
with respect to this measure.

Assume, from now on, that
\begin{equation}
\label{04}
B\;:=\; \int_{\bb T} b(\theta)\, d\theta \;>\; 0\;,
\end{equation}
so that the process $X_\epsilon(t)$ is non-reversible. In \cite{fg1},
the stationary measure of this process has been explicitly computed.
Regard $b$ as an $1$-periodic function on $\bb R$.  Let $S:\bb R\to
\bb R$ be the function given by
\begin{equation}
\label{34}
S(x) \;=\; - \int_0^x b(z)\, dz\;,
\end{equation}
and let $\pi_\epsilon$, $m_\epsilon : \bb R \to \bb R_+$ be given by
\begin{equation}
\label{79}
\pi_\epsilon(x) \;=\; \int_x^{x+1} e^{[S(y)-S(x)]/\epsilon}\, dy
\;, \quad 
m_\epsilon(x) \;=\; \frac 1{c(\epsilon)} \, \pi_\epsilon(x)\;, 
\end{equation}
where $c(\epsilon)$ is the normalizing constant which turns
$m_\epsilon$ the density of a probability measure on $\bb T$. Indeed,
$\pi_\epsilon$, $m_\epsilon$ are $1$-periodic and can be considered as
defined on $\bb T$.  By \cite{fg1}, the measure $\mu_\epsilon
(d\theta) = m_\epsilon(\theta) \, d\theta$ on $\bb T$ is the
stationary state of the diffusion \eqref{03}.

\smallskip\noindent{\bf \ref{sec01}.2. The quasi-potential.} 
Let $z: \bb R \to \bb R$ be the function which indicates the position
of the farthest maxima of $S$ to the right: $z(x)=z_x$ is the largest
point in $[x,\infty)$ at which a maximum of the set $\{ S(y) : y\ge x
\}$ is attained. More precisely, \\
\smallskip 
\noindent (a) $\displaystyle S(z_x) \;=\; \max\{S(y) : y\ge x\}$, \\
\noindent (b) If $y\ge x$ and $S(y) = S(z_x)$, then $y\le z_x$.
\smallskip

Note that, for all $x\in\bb R$, $z(x)$ not only exists but also
satisfies $z(x) \in [x,x+1)$ because $S(y+1) = S(y) - B < S(y)$.
Moreover, $z_x$ is a local maximum of $S$ if $z_x \not = x$. In this
case, $b(z_x) = - S'(z_x)=0$.

Let $\widehat V:\bb R \to \bb R$ be given by
\begin{equation*}
\widehat V(x) \;=\; S(x) \;-\; S(z_x)\;.
\end{equation*}
Since $S(y+k) = S(y) - kB$ for $y\in\bb R$, $k\in \bb Z$, $z(x+k) =
z(x)+k$. In particular, $\widehat V(x+1) = \widehat V(x)$, so that
$\widehat V$ is a $1$-periodic function and can be considered as
defined on the torus $\bb T$. By \cite[Proposition 2.1]{fg1},
$\widehat V$ is the quasi-potential associated to the diffusion
\eqref{03}, and by \cite[Theorem 2.4]{fg1} it is a viscosity solution
of the Hamilton-Jacobi equation associated to the Hamiltonian $H(x,p)
= p[p-b(x)]$.

\section{Main result: Two stable points}
\label{sec02}

We present in this section the main results of the article in the
case, illustrated in Figure \ref{if01}, where the drift $b$ is smooth
and the dynamical system $dX(t)=b(X(t))\, dt$ exhibits two stable
equilibria and two unstable ones.  

In addition to the conditions ({\bf H1}-{\bf{H3})}, we shall assume
in this section that

\smallskip\noindent ({\bf H0}) The drift $b$ is smooth and the set
$\{\theta\in\bb T: b(\theta) = 0\}$ consists of four points.

\smallskip 

Condition ({\bf H0}) is not needed in the proofs of the results
presented in this article. It is assumed in this section because it
simplifies significantly the notation and the statement of the
results, helping the reader to access the content of the
article. Further assumptions will be formulated along the section.

\begin{figure}
  \protect
\includegraphics[scale=0.26]{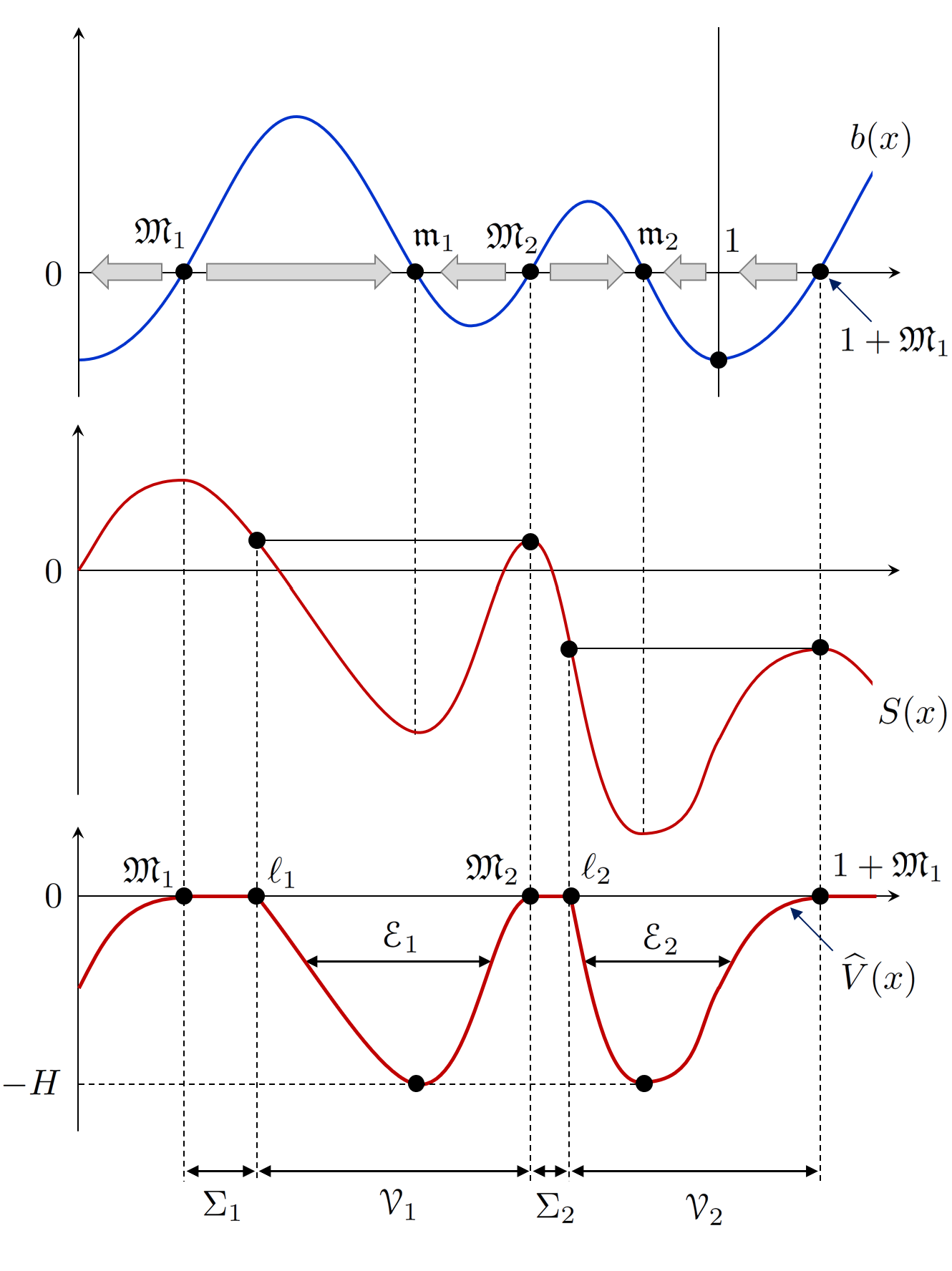}\protect
\caption{The graphs of $b$, $S$ and $\widehat{V}$. In the
  first graph, the gray arrows represent the direction of the drift in
  the dynamical system $dX(t)=b(X(t))\, dt$. Thus, $\mf{m}_1$ and
  $\mf{m}_2$ are stable equilibria, and $\mf{M}_1$ and $\mf{M}_2$ are
  unstable equilibria. The existence of two stable equilibria
  separated by unstable one implies a metastable behavior of the
  perturbed dynamics \eqref{03}.}
\label{if01}
\end{figure}

\smallskip\noindent{\bf \ref{sec02}.1. Notation.} By assumptions ({\bf
  H0}) and ({\bf H3}), $S(\cdot)$ has two local maxima
$\mf{M}_{1},\,\mf{M}_{2}$ and two local minima
$\mf{m}_{1},\,\mf{m}_{2}$. Without loss of generality, we assume that
$0<\mf{M}_{1}<\mf{m}_{1}<\mf{M}_{2}<\mf{m}_{2}<1$, and that that $S(\mf M_1) > S(\mf M_2)> S(\mf M_3)$, where we adopted the convention that $\mf{M}_{3}=1+\mf{M}_{1}$.  We refer to
Figure \ref{if01} for the graphs of $b(\cdot),\,S(\cdot)$ and
$\widehat{V}(\cdot)$.

For each $i=1,\,2$, let
$\ell_{i}=\inf\{x>\mf{M}_{i}: S(x)=S(\mf{M}_{i+1})\}$ and set
\begin{equation*}
\Sigma_{i}\;=\; (\mf{M}_{i},\,\ell_{i}) \;,\;\;
\ms{V}_{i}\;=\;(\ell_{i},\,\mf{M}_{i+1})\;, \;\;
i=1,\,2\;.
\end{equation*}
Note that $\ms{V}_{2}$ can be regarded as the subset of $\bb T$ given
by $(\ell_{i},\,1]\cup(0,\,\mf{M}_{1})$.  These notation will be
comprehensively extended to a general drift $b$ in Section
\ref{sec03}. Note that for $x\in\Sigma_{1}\cup\Sigma_{2}$, one has
that $z(x)=x$ and hence $\widehat{V}(x)=0$ (cf. Figure \ref{if01}).
Thus, the sets $\Sigma_{i}$, $i=1,\,2$, represent the \textit{saddle
  intervals} between two \textit{valleys} $\ms{V}_{1}$ and
$\ms{V}_{2}$. The notion of valley is extended to the one of
\textit{landscape} in Section \ref{sec03} to handle more general
situations. The depth of the valleys $\ms{V}_{1}$ and $\ms{V}_{2}$ are
$-\, \widehat{V}(\mf{m}_{1})$ and $-\, \widehat{V}(\mf{m}_{2})$,
respectively. Assume that
\begin{equation}
\label{ass01}
\widehat{V}(\mf{m}_{1})\;=\;\widehat{V}(\mf{m}_{2})\;=\;-H\;,
\end{equation}
so that the depth of the two valleys coincide. This assumption is not
necessary for the results below, but without it the results become
trivial and can easily be deduced from the argument.

For $i=1,\,2$, let
$\ms{E}_{i}=[e_{i}^{-},\,e_{i}^{+}]\subseteq\ms{V}_{i}$ be such that
$\widehat{V}(\mf{m}_{i})<\widehat{V}(e_{i}^{-})=\widehat{V}(e_{i}^{+})<0$
so that $\mf{m}_{i}\in\ms{E}_{i}$. The set $\ms{E}_{i}$ represents the
metastable well around the stable point $\mf{m}_{i}$.

\smallskip\noindent{\bf \ref{sec02}.2. Sharp asymptotics for the
  pre-factor. } The first main result of the article provides a sharp
estimate of the stationary state. Write $m_{\epsilon}$ as 
\begin{equation*}
m_{\epsilon}(\theta)\;=\;F_{\epsilon}(\theta)\, e^{-V(\theta)/\epsilon}\;,
\end{equation*}
where $V(\theta) = \widehat{V}(\theta) + H$.  The function
$F_{\epsilon}(\cdot)$ is called the pre-factor.  Its behavior as
$\epsilon \to 0$ plays a fundamental role in the estimation of the
capacity between two wells, which is one of the crucial steps in the
proof of the metastable behavior of a Markov chain. Such a result is
still open in the non reversible context except in the trivial case
where the pre-factor is constant. The first main result of this
article provides an expansion in $\epsilon$ of the pre-factor. For
$i=1$, $2$, let
\begin{equation*}
\sigma(\mf{m}_{i})\;=\;\sqrt{\frac{2\pi}{-b'(\mf{m}_{i})}}\;,
\quad
\omega (\mf{M}_{i})\;=\;\sqrt{\frac{2\pi}{b'(\mf{M}_{i})}}\;, 
\quad
Z \;=\; \sum_{j=1}^2 \sigma(\mf{m}_{j}) \, \omega (\mf{M}_{j+1}) \;.
\end{equation*} 

\begin{theorem}
\label{st01}
Under the assumptions ({\bf H0}-{\bf H3}),
\begin{enumerate}
\item (Pre-factor on valleys) for all $x\in\ms{V}_{i}$, $i=1,\,2$,
\begin{equation*}
F_{\epsilon}(x)\;=\; \big[1+o(1)\big]\,
\frac{\omega(\mf{M}_{i+1})}{Z\,\sqrt{\epsilon}}
\end{equation*}
where $o(1)\rightarrow0$ as $\epsilon\rightarrow0$ uniformly on
$\ms{V}_{1}\cup\ms{V}_{2}$.  \smallskip 
\item (Pre-factor on saddle intervals) for all
  $x\in\Sigma_{1}\cup\Sigma_{2}$,
\begin{equation*}
F_{\epsilon}(x)\;=\; \big[1+o(1)\big]\,
\frac{1}{Z\,b(x)}\;.
\end{equation*}
where $o(1)\rightarrow0$ as $\epsilon\rightarrow0$ for all
$x\in\Sigma_{1}\cup\Sigma_{2}$.
\end{enumerate}
\end{theorem}

The general case, without assumption({\bf H0}), is presented in
Propositions \ref{l01} and \ref{l02}.  Note the difference in the
scaling factor in parts (1) and (2). This difference is explained
along with a connection to the Hamilton-Jacobi equation for
$F_{\epsilon}$ in Subsection \ref{sec0}.5. The scaling difference of
the pre-factor indicates that its asymptotic analysis in higher
dimension may be a difficult problem.

\smallskip\noindent{\bf \ref{sec02}.3. Metastable behavior. } We turn
to the metastable behavior of the diffusion $X_{\epsilon}(t)$ between
the valleys $\ms{E}_{1}$ and $\ms{E}_{2}$. Let
$\widehat{X}_{\epsilon}(t):=X_{\epsilon}(e^{H/\epsilon}t)$ be the
speeded-up process. As in the approach developed in \cite{bl2, bl7},
we define the metastable behavior of the diffusion as the convergence
of the projection of the trace process.

To define the trace process of $\widehat{X}_{\epsilon}(t)$ on
$\ms{E}=\ms{E}_{1}\cup\ms{E}_{2}$, let
\begin{equation*}
T_{\ms{E}}(t)\;=\; \int_{0}^{t} \chi_{\ms{E}}
(\widehat{X}_{\epsilon}(s)) \, ds \;, \quad
S_{\ms{E}}(t) \;=\; \sup\{s\ge0:T_{\ms{E}}(s)\le t\}\;.
\end{equation*}
The process $Y_{\epsilon}(t)=\widehat{X}_{\epsilon}(S_{\ms{E}}(t))$ is
called the trace of the process $\widehat{X}_{\epsilon}$ on
$\ms{E}$. Informally, one obtains a trajectory of $Y_{\epsilon}(t)$
from $\widehat{X}_{\epsilon}(t)$ by deleting the excursion of
$\widehat{X}_{\epsilon}$ outside $\ms{E}$. In Subsection \ref{sec3}.4,
we show that $Y_{\epsilon}(\cdot)$ is a Markov process on $\ms{E}$
with respect to a suitable filtration. Let
$\Psi:\ms{E}\rightarrow\{1,\,2\}$ be the projection defined by
$\Psi(x) = \chi_{\ms{E}_{1}}(x) + 2\chi_{\ms{E}_{2}}(x)$.  Clearly,
$\bs{x}_{\epsilon}(t)=\Psi(Y_{\epsilon}(t))$ takes values in the set
$\{1,\,2\}$, and represents the valley visited by the process
$Y_{\epsilon}(t)$. Following \cite{bl2,bl7}, we shall say that the
process $X_\epsilon(t)$ is metastable in the time-scale
$e^{H/\epsilon}$ if $\bs{x}_{\epsilon}(\cdot)$ converges to a Markov
chain on $\{1,\,2\}$, and if the process $\widehat X_{\epsilon}(t)$
remains outside $\ms E$ for a negligible amount of time.

The method developed in \cite{bl2, bl7, bl9} provides a robust way to
establish these results. Moreover, it has been shown in \cite{llm}
that under some mild extra assumptions the metastability as stated
above entails the convergence of the finite-dimensional distributions
of the projections of $\widehat X_{\epsilon}(t)$.

This approach to metastability was successfully enforced in the
context of Markov chains. Its extension to diffusions, such as the one
considered in the current paper, faced a major difficulty due to the
singular behavior of the trace process $Y_{\epsilon}(t)$ at the
boundary of $\ms{E}$. In this paper, we propose a new way of
establishing the convergence of the projection of the trace process
based on results from the theory of partial differential equations
(cf. \cite{et, st}). This is the content of Sections \ref{sec3} and
\ref{sec5}. \smallskip

Let 
\begin{equation*}
R(1,\,2)\;=\;\frac{1}{\sigma(\mf{m}_{1})\, \omega(\mf{M}_{2})}
\;,\quad
R(2,\,1)\;=\;\frac{1}{\sigma(\mf{m}_{2})\, \omega(\mf{M}_{1})}\;,
\end{equation*}
and consider the Markov chain $\bs{X}(t)$ on $\{1,\,2\}$ with
generator given by
\begin{equation*}
(\bs{L}f)(i)\;=\;R(i,\,j)\left[F(j)-F(i)\right]\;,
\quad \{i,\,j\} \,=\,\{1,\,2\}\;.
\end{equation*}
Let $\bs{Q}_{j}$, $j=1,\,2$, be the law of the Markov chain
$\bs{X}(t)$ starting from $j$.

\begin{theorem}
\label{st02}
Fix $j\in\{1,\,2\}$, $\theta_{0}\in\ms{E}_{j}$ and suppose that
$X_{\epsilon}(0)=\theta_{0}$ for all $\epsilon>0$. Then, the law of
the process $Y_{\epsilon}(\cdot)$ converges to $\bs{Q}_{j}$ as
$\epsilon\to 0$.
\end{theorem}

The general version of this result is presented in Theorem \ref{th01}.
Although, under {(\bf{H0})}, the process $\bs{X}(t)$ is reversible
with respect to its invariant distribution, this is no longer true in
the general setting. Actually, Theorem \ref{th01} provides the first
example of a dynamics whose asymptotic evolution is described by a
non-reversible and irreducible Markov chain.

The proof of Theorem \ref{st02} is divided in two parts. We have first
to establish the tightness of the process $\bs{x}_{\epsilon}(t)$
(cf. Section \ref{sec3}.5). The core step in the proof of this result
is an estimation of the escape time from a metastable well. For this
purpose we establish a general inequality, presented in Proposition
\ref{p01}, which bounds the hitting time of a set in terms of a
capacity which can be easily estimated through the variational
formulae for the capacity.  We believe that this inequality, the third
main result of the article, can be useful in numerous different
contexts.

The second part of the proof consists in the characterization of the
limit point.  This part is based on the analysis of the solution of a
certain Poisson equation (cf. Proposition \ref{as14}). This sort of
analysis has been carried out in \cite{et, st} for reversible
diffusions based on ideas from PDEs. 

\section{Valleys and landscapes}
\label{sec03}

\begin{figure}
  \centering
\begin{tikzpicture}
\draw (-5.4,3.4) [rounded corners = 5pt] -- (-5,5) 
[rounded corners = 5pt] -- (-4, 5) 
[rounded corners = 5pt] -- (-3.7,4)  -- (-3.3,4) 
[rounded corners = 5pt] -- (-3,3) -- (-2,3) 
[rounded corners = 5pt] -- (-1,.5) 
[rounded corners = 5pt] -- (-0.5,2.5) 
[rounded corners = 5pt] -- (0,1) 
[rounded corners = 5pt] -- (.7,3) 
[rounded corners = 5pt] -- (1.6,3)
[rounded corners = 5pt] -- (2.1,.8)
[rounded corners = 5pt] -- (2.6,3.1)
[rounded corners = 5pt] -- (3.5,.4)
[rounded corners = 5pt] -- (4.4,.4)
[rounded corners = 5pt] -- (4.85,2) 
[rounded corners = 5pt] -- (5.5,2); 
\draw (5,0) -- (5, 2);
\draw (2.96,0) -- (2.96, 2);
\draw (2.61,0) -- (2.61, 3);
\draw (1.45,0) -- (1.45, 3);
\draw (.85,0) -- (.85, 3);
\draw (-2.15,0) -- (-2.15, 3);
\draw (-2.85,0) -- (-2.85, 3);
\draw (-4.85,0) -- (-4.85, 5);
\draw[blue, very thick] (-2.15,0) -- (.85,0);
\draw[blue, very thick] (1.45,0) -- (2.61,0);
\draw[red, very thick] (2.61,0) -- (2.96,0);
\draw[blue, very thick] (2.96,0) -- (5,0);
\draw (4.5,4) node[anchor=east] {$S(x)$};
\draw (-3.9,5) node[anchor=west] {$\mf M^+_1 = \mf L^+_1$};
\fill[cyan] (-4.15,5) circle [radius = .05cm];
\draw (-.7,2.8) node[anchor=west] {$\mf M^+_2$};
\fill[cyan] (-.5,2.38) circle [radius = .05cm];
\draw (1.45,3.2) node[anchor=south] {$\mf M^+_3$};
\fill[cyan] (1.45,3) circle [radius = .05cm];
\fill[cyan] (2.61,3) circle [radius = .05cm];
\draw (2.61,3.2) node[anchor=west] {$\mf M^+_4=\mf L^+_2$};
\draw (5,2) node[anchor=south] {$\mf M^-_5$};
\fill[cyan] (5,2) circle [radius = .05cm];
\draw[->] (-5.5,-.2) -- (5.5,-.2);
\draw (-4.85,-.2) node[anchor=south east] {$0$};
\draw (-4.85,-.3) -- (-4.85,-.1); 
\draw (5,-.3) -- (5,-.1); 
\draw (5,-.2) node[anchor=south west] {$1$};
\draw (-2.15,3) node[anchor= south] {$\ell^+_1$};
\fill[cyan] (-2.15,3) circle [radius = .05cm];
\draw (-2.85,3.1) node[anchor=north east] {$\mf u_1$};
\fill[cyan] (-2.85,3) circle [radius = .05cm];
\draw (2.96,2) node[anchor=west] {$\mf u_4 = \ell_2$};
\fill[cyan] (2.96,2) circle [radius = .05cm];
\draw (-4.15,5) -- (-4.15, 0);
\draw[red, thick] (-4.85,3) -- (-4.15,3);
\draw (-4.5,3) node[anchor = north]{$\ms A_1$};
\draw[red, very thick] (-4.15,0) -- (-2.85,0);
\draw (-3.5,0) node[anchor=south] {$\Sigma_1$};
\draw[brown, very thick] (-5.4,-.4) -- (-4.15, -.4);
\draw (-4.65,-.4) node[anchor=north] {$\Lambda_0$};
\draw[brown, very thick] (-2.85,-.4) -- (2.61, -.4);
\draw (0.5,-.4) node[anchor=north] {$\Lambda_1$};
\draw[brown, very thick] (2.96,-.4) -- (5.5,-.4);
\draw (4,-.4) node[anchor=north] {$\Lambda_2$};
\end{tikzpicture}
\caption{This figure represents the graph on the interval
  $[-\delta,1+\delta]$ of a functions $S:\bb R\to \bb R$ associated by
  \eqref{34} to a vector field $b$ defined on the torus $\bb T$. In
  this example, the set $\{\theta\in \bb T : b(\theta)=0\}$ has $p=10$
  connected components. There are $q=4$ local maxima, and $\bs m =2$
  landscapes indicated in brown. The landscape $\Lambda_1$ has two
  valleys represented in blue, while the landscape $\Lambda_2$ has
  only $1$ valley. The two saddle intervals $\Sigma_1$, $\Sigma_2$ are
  displayed in red.}
\label{fig1}
\end{figure}
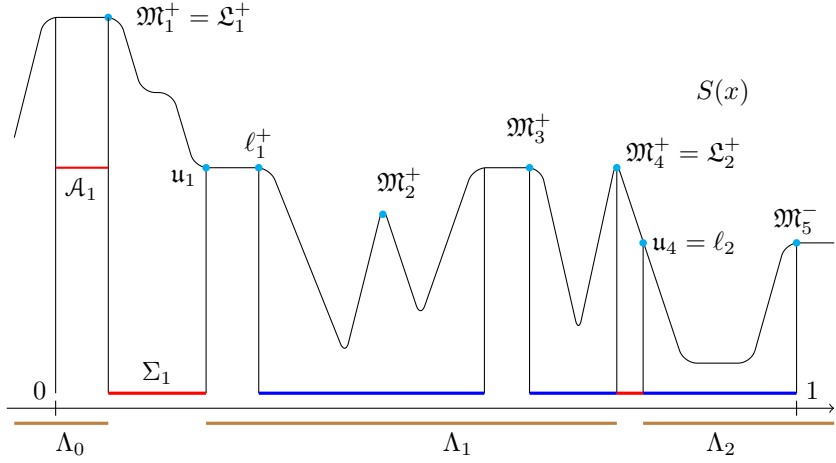

We introduce in this section the notion of valleys and
landscapes which play an important role in the description
of the quasi-potential $\widehat V$.

Let $\ms A_k = [\mf M^-_k, \mf M^+_k]$, $1\le k\le q$, $q\ge 0$,
(resp. $\ms U_j = [\mf m^-_j, \mf m^+_j]$, $1\le j\le q'$, $q'\ge 0$)
be the sub-intervals of $[0,1)$ where $S$ assumes a local maximum
(resp. minimum). For the local maxima, this means that $b$ vanishes on
each interval $[\mf M^-_k, \mf M^+_k]$ and that $b'(\mf M^+_k)>0$,
$b'(\mf M^-_k)>0$. Similar relations hold for the intervals where $S$
attains a local minimum. Note that $q$, $q'$ might be equal to $0$.
The set $\ms A_1$ is represented in Figure \ref{fig1}.

Since each local maxima is succeeded by a local minima, $q'$ must be
equal to $q$.  The intervals $\ms A_k$, $\ms U_k$ might be reduced to
points, and they are supposed to be ordered in the sense that $0\le
\mf M^-_1 \le \mf M^+_1 < \mf m^-_1 \le \mf m^+_1 < \mf M^-_2 < \cdots
< \mf M^+_q < \mf m^-_q \le \mf m^+_q< 1$. For $m\in \bb Z$, $1\le
j\le q$, let $\mf M^\pm_{j+mq} = m+ \mf M^\pm_j$, $\mf m^\pm_{j+mq} =
m+ \mf m^\pm_j$. 

Observe that
\begin{equation}
\label{40}
z(x) \;=\; x \;\;\text{or}\;\;
z(x) \;\in\; \big\{\mf M^+_k : \mf M^+_k \,>\, x\big\}\;.
\end{equation}
Indeed, if $z(x)\not = x$, $z(x)$ must be greater than $x$ and $z(x)$
must be the right endpoint of an interval where $S$ attains a local
maximum. \smallskip

If $q=0$, the diffusion $X_\epsilon (t)$ has a nonnegative drift. In
this case, $S$ is a non-increasing function and $S(x)=S(z_x)$ for all
$x\in \bb R$ so that the quasi-potential vanishes, $\widehat V\equiv
0$. Conversely, if $\widehat V\equiv 0$, $S(z_x) = S(x)$ for all $x$,
which implies that $S$ has no local minima. Hence,
\begin{equation*}
\widehat V\equiv 0 \quad\text{if and only if}\quad b\;\ge\; 0
\quad\text{if and only if}\quad q\;=\;0 \;.
\end{equation*}
Assume from now on
that $q\ge 1$.  Consider a local maximum $\mf M^+_k$ such that
\begin{equation}
\label{53}
z(\mf M^+_k) \;=\; \mf M^+_k  \;.
\end{equation}
There is at least one maximum which comply with this condition: if
$\mf M^+_1$ does not fulfill it, then $z(\mf M^+_1)$ does. In Figure
\ref{fig1}, $\mf M^+_1$, $\mf M^+_4$ are the maxima which satisfy this
condition.

Denote by $\bs m$ the number of local maxima which satisfy condition
\eqref{53}, and represent them by $\mf L^\pm_n = \mf M^\pm_{j(n)}$,
$1\le n\le \bs m$, for some sequence $1\le j(1)< \cdots < j(\bs m)\le
q$. As $\bs m \ge 1$ and $\bs m \le q$, we have that $0\le \mf L^+_1 <
\dots < \mf L^+_{\bs m} <1$. Assume, without loss of generality, that
$\mf L^+_1 = \mf M^+_1$, and extend the definition of the maxima $\mf
L^\pm_k$ by setting $\mf L^\pm_{n+r\bs m} = r + \mf L^\pm_{n}$, $r\in
\bb Z$, $1\le n\le \bs m$.

\begin{asser}
\label{as12}
Fix a point $\mf L^+_n = \mf M^+_{j(n)}$. We claim that
$\mf L^+_{n+1} = z(\mf M^+_{j(n)+1})$. 
\end{asser}

\begin{proof}
Figure \ref{fig1} illustrates
this assertion, as $\mf L^+_2 = z(\mf M^+_2)$. To prove the claim, we
have to show that $z(\mf M^+_{j(n)+1})$ fulfills \eqref{53} and that
no maxima $\mf M^+_k$ in the interval
$(\mf M^+_{j(n)} , z(\mf M^+_{j(n)+1}))$ satisfies \eqref{53}.  It is
clear that $z(\mf M^+_{j(n)+1})$ fulfills \eqref{53}. We turn to the
second property. By \eqref{40}, 
$z(\mf M^+_{j(n)+1}) = \mf M^+_\ell$ for some $\ell \ge j(n)+1$. 
By definition of $z$, $S(\mf M^+_k) \le S(z(\mf M^+_{j(n)+1}))$ or all 
$j(n)+1 \le k< \ell$. Hence $\mf M^+_k$ does not satisfy \eqref{53}
for $k$ in this range, which proves the claim. 
\end{proof}

Let $\mf M^+_j$ be a maximum for which \eqref{53} holds. Consider the
next maximum, $\mf M^+_{j+1}$. As $z(\mf M^+_{j}) = \mf M^+_{j}$,
$S(\mf M^+_{j+1}) < S(\mf M^+_{j})$ and $S(z(\mf M^+_{j+1})) < S(\mf
M^+_{j})$. Let $\mf u_{j}$ be the first point $x$ larger than $\mf
M^+_{j}$ such that $S(x) = S(z(\mf M^+_{j+1}))$:
\begin{equation*}
\mf u_{j} \;=\; \inf \big\{x \ge \mf M^+_{j} : 
S(x) = S(z(\mf M^+_{j+1})) \big\} \;.
\end{equation*}
We refer to Figure \ref{fig1} for a representation of $\mf u_{1}$ and
$\mf u_{4}$. It is clear that $S(\mf u_{j}) = S(z(\mf M^+_{j+1}))$.

\begin{asser}
\label{as13}
Suppose that $\mf M_j^+$ satisfies \eqref{53}. Then, we have that 
\begin{equation*}
\widehat V(x) \;=\;
\begin{cases}
0   & \mf M^-_{j} \;\le\; x \;\le \; \mf u_{j}\;, \\
S(x) \,-\, S(z(\mf M^+_{j+1})) & 
\mf u_{j} \;\le\; x \;\le \;z(\mf M^+_{j+1})\;.
\end{cases}
\end{equation*}
In particular, $\widehat V(\mf u_{j}) = \widehat V(z(\mf M^+_{j+1})) =
0$, and on the set $[\mf u_{j} , z(\mf M^+_{j+1})]$, the functions
$\widehat V$ and $S$ differ by an additive constant. 
\end{asser}

\begin{proof}
We leave to the reader to check that
\begin{equation*}
S(z(x))\;=\;
\begin{cases}
S(x)   & \text{for $\mf M^-_{j} \;\le\; x \;\le \; \mf u_{j}$,} \\
S(z(\mf M^+_{j+1})) & \text{for $\mf u_{j} \;\le\; x \;\le \;z(\mf M^+_{j+1})$.}
\end{cases}
\end{equation*}
The assertion follows from this identity and from the definition of
$\widehat V$.
\end{proof}

Note that it is not true that $z(x)= x$ for $x\in [\mf M^+_{j} , \mf
u_{j}]$ because there might be subintervals of $[\mf M^+_{j} , \mf
u_{j}]$ where $S$ is constant. We refer to Figure \ref{fig1}. In
contrast, $z(x) = z(\mf M^+_{j+1})$ for $x\in [\mf u_{j} , z(\mf
M^+_{j+1})]$.

Consider a maximum $\mf L^+_n = \mf M^+_{j(n)}$. The previous result
characterizes the function $\widehat V$ in the interval $[\mf L^-_n ,
z(\mf M^+_{j(n) +1})]$.  Since, by Assertion \ref{as12}, $z(\mf
M^+_{j(n) +1}) = \mf L^+_{n+1}$ Assertion \ref{as13} provides in fact
a representation of the quasi-potential in the interval $[\mf L^-_n ,
\mf L^+_{n+1}]$ and, therefore, in $\bb T$.

Let $\ell_n = \mf u_{j(n)}$, $\ell_{r\bs m +n} = r + \ell_{n}$,
$r\in \bb Z$, $1\le n\le \bs m$, and let
\begin{equation}
\label{26}
\Sigma_n \;=\; (\mf L^+_n, \ell_n)\;, \quad
\Lambda_n \;=\; [\ell_n, \mf L^+_{n+1}]  \;, \quad
\Sigma \;=\; \bigcup_{n=1}^{\bs m} \Sigma_n \;, \quad
\Lambda \;=\; \bigcup_{n=1}^{\bs m} \Lambda_n\;.
\end{equation}
The sets $\Lambda_n$ are named landscapes and the sets $\Sigma_n$
saddle intervals. Of course, $\{\Sigma, \Lambda\}$ forms a partition
of $\bb T$.  

\begin{remark}
\label{rm6}
By Assertion \ref{as13} and by definition of $\Lambda_n$, $\Sigma_n$,
$\widehat V$ vanishes on $\Sigma$ and $\widehat V$ and $S$ differ by
an additive constant on each landscape $\Lambda_n$. This additive
constant may be different at each set $\Lambda_n$.
\end{remark}

Note that $\mf L^-_n < \ell_n$. Hence, even if the connected components
of the set $\{\theta\in\bb T : b(\theta)=0\}$ are points (that is, if
$r_i=l_i$ for all $1\le i\le p$), the intervals $\Sigma_n$ at which the
quasi-potential $\widehat V$ vanishes are non-degenerate. (See Figure
\ref{if01}). 
 
By Assertion \ref{as13}, $S(\ell_n) = S(\mf L^+_{n+1})$. If $\ell_n =
l_k$ for some $1\le k\le p$, let $\ell^+_n = r_k$, otherwise let
$\ell^+_n = \ell_n$:
\begin{equation*}
\ell^+_n \;=\; \sup \big\{ x\ge \ell_n : S(y) = S(\ell_n) \text{ for all }
  \ell_n \le y\le x\, \big\}\;.
\end{equation*}

Each landscape $\Lambda_n = [\ell_n, \mf L^+_{n+1}]$ may contain in
$(\ell^+_n , \mf L^-_{n+1})$ local maxima $\mf M^+_k$ of $S$ such that
$S(\mf M^+_k) = S(\ell_n)$ (and thus $\widehat V(\mf M^+_k) =0$). Let
$\ell^+_n < \mf M^+_{a(n,1)} < \dots < \mf M^+_{a(n,r_n-1)} < \mf
L^-_{n+1}$ be an enumeration of these local maxima. Set $a(n,r_n) =
j(n+1)$ so that $\mf M^+_{a(n,r_n)} = \mf M^+_{j(n+1)} = \mf
L^+_{n+1}$. The sets
\begin{equation}
\label{54}
\begin{aligned}
& \ms V_{n,1} \;=\; (\ell^+_n, \mf M^-_{a(n,1)}) \;, \quad
\ms V_{n,r_n} \;=\;  (\mf M^+_{a(n,r_n-1)}, \mf L^-_{n+1})    \;, \\
&\quad \ms V_{n,j+1} \;=\;  (\mf M^+_{a(n,j)}, \mf M^-_{a(n, j+1)})\;,
\quad 1\le j\le r_n-2\;.
\end{aligned}
\end{equation}
are called the valleys of the landscape $\Lambda_n$. To simplify some
equations below let $\mf M^+_{a(n,0)} = \ell^+_{n}$, $\mf M^-_{a(n,0)}
= \ell^-_{n}= \ell_{n}$. Note that there is an abuse of notation since
$\mf M^\pm_{a(n,0)} = \ell^\pm_{n}$ may not be a maxima.

In Figure \ref{fig1}, the landscape $\Lambda_1 = [\mf u_1, \mf M^+_4]$
has $2$ valleys, $(\ell^+_1, \mf M^-_3)$ and $(\mf M^+_3, \mf M^-_4)$,
while the landscape $\Lambda_2=[\mf u_4, \mf M^+_5]$ has only one
valley. Each landscape has at least one valley. If there are no local
maxima $\mf M^+_k$ of $S$ in $(\ell^+_n , \mf L^-_{n+1})$ such that
$S(\mf M^+_k) = S(\ell_n)$, then $r_n=1$ and the set $(\ell^+_n, \mf
L^-_{n+1})$ forms a valley.  

\begin{remark}
\label{rm1}
The quasi-potential $\widehat V$ may have plateaux $\{\theta \in \bb T
: \widehat V(\theta)=0\}$ which are not saddle intervals, but which
belong to a landscape. This happens if one of the local maxima $\mf
M^+_{a(n,j)}$, $0\le j\le r_n$, introduced above is such that $\mf
M^+_{a(n,j)} \not = \mf M^-_{a(n,j)}$ [with the convention adopted
concerning $a(n,0)$ and $a(n,r_n)$]. This possibility is illustrated
in Figure \ref{fig1} by the intervals $[\mf M^-_3, \mf M^+_3]$,
$[\ell_1, \ell^+_1]$. However, if the connected components of the set
$\{\theta \in \bb T : b(\theta)=0\}$ are points, all plateaux of
$\widehat V$ are saddle intervals because in the landscapes the
quasi-potential differ from $S$ by an additive constant.
\end{remark}

\begin{remark}
\label{rm3}
In a landscape $\Lambda_n = [\ell_n, \mf L^+_{n+1}]$, the process
$X_\epsilon(t)$ evolves among the valleys $\ms V_{n,j}$ as a
reversible process until it leaves $\Lambda_n$. Since $S(\mf L^-_{n})
> S(\ell_n)= S(\mf L^+_{n+1})$, with a probability exponentially close
to $1$, $X_\epsilon(t)$ leaves the landscape $\Lambda_n$ through the
saddle interval $\Sigma_{n+1}$. In $\Sigma_{n+1}$, as the drift is
nonnegative, $X_\epsilon(t)$ slides to the next landscape
$\Lambda_{n+1}$. Once in $\Lambda_{n+1}$, with a probability
exponentially close to $1$, the process $X_\epsilon(t)$ does not
return to $\Sigma_{n+1}$. In particular, the saddle interval
$\Sigma_{n+1}$ is only visited during the excursion from $\Lambda_{n}$
to $\Lambda_{n+1}$. This explains why the quasi-potential $\widehat V$
vanishes on the saddle intervals.
\end{remark}

\begin{remark}
\label{rm2}
It is not possible to recover $S$ from $\widehat V$.  Given a maximal
interval $[\theta_1, \theta_2]$ at which $\widehat V$ is constant
equal to $0$, it not possible to determine whether this interval is a
saddle interval or whether it belongs to a landscape. However, if the
connected components of $\{\theta\in\bb T : b(\theta)=0\}$ are points,
it is possible to recover $S$ from $\widehat V$ and the pre-factor
introduced in the next subsection.
\end{remark}

\section{The stationary state}
\label{sec0}

One important question in the theory of non-reversible Markovian
dynamics is to access the stationary state. Bounds for the
quasi-potential with small exponential errors can be deduced from the
theory of large deviations \cite{fw}. We present in Propositions
\ref{l01}, \ref{l02} and \ref{p02} below sharp asymptotics for the
first-order term of the expansion in $\epsilon$ of the
quasi-potential, the so-called pre-factor, defined in \eqref{i01}
below.

Precise estimates of the pre-factor play a central role in the
derivation of the metastable behavior of a random process based on the
potential theory, as one needs to evaluate the measure of a valley and
the capacity between valleys (cf. \cite{begk, bl2, bl7}). An
asymptotic analysis of the pre-factor for non-reversible dynamics
similar to the one presented in this section has never been carried
out before.

One available tool to obtain estimates for the pre-factor is the
Hamilton-Jacobi equation (cf. \eqref{16} below). Write this equation
as $H_\epsilon(F_\epsilon)=0$. One is tempted to argue that
$F_\epsilon$ should converge, as $\epsilon\to 0$, to the solution of
$H_0(F_0)=0$. We show in Subsection \ref{sec0}.5 the limits of this
analysis, proving that $F_\epsilon$ converges to a function which is
discontinuous at the saddle points.

The main results of this section are based on the explicit expression
\eqref{79} for the stationary state obtained by Faggionato and
Gabrielli in \cite{fg1}. Some of the claims below appear in
\cite{fg1}. They are stated here in sake of completeness as they will
be used in the next sections.

\smallskip\noindent{\bf \ref{sec0}.1. Definition.}
Recall the definition of the quasi-potential $\widehat V$ and let
$V:\bb T \rightarrow \bb R_+$ be the non-negative function given by
\begin{equation}
\label{19}
V(\theta) \;=\; \widehat V(\theta) \;+\; H \;.
\end{equation}
Since the quasi-potential is defined up to constants, $V$ can be
regarded as another version of the quasi-potential.  Write the density
$m_\epsilon (\theta)$ of the stationary distribution as
\begin{equation}
\label{i01}
m_\epsilon (\theta) \;=\; F_\epsilon (\theta) \, e^{- V
  (\theta)/\epsilon}\;.
\end{equation}
The function $F_\epsilon$ is called the pre-factor, and corresponds to
the first order correction of the quasi-potential.  

\smallskip\noindent{\bf \ref{sec0}.2. Sharp asymptotics.}  We
introduce three functions $G_k: \bb T \to \bb R$, $0\le k\le 2$, which
appear in the pre-factor. These functions are defined separately on
each interval $\Lambda_n$, $\Sigma_n$, $1\le n\le \bs m$.

We first consider the landscape. Fix $1\le n\le \bs m$ and consider
the set $\Lambda_n = [\ell_n, \mf L^+_{n+1}] = [\mf u_{j(n)}, \mf
M^+_{j(n+1)}]$.  Denote by $G_0: \Lambda_n \to \bb R_+$ the function
given by
\begin{equation*}
G_0(x) \; =\; \int_x^{\mf L^+_{n+1}} \bs 1\{ S(y) = S(\ell_n)\}\, dy \;.
\end{equation*}
In this formula, $\mb 1\{A\}$ takes the value $1$ if $A$ holds and $0$
otherwise. The value of $G_0$ at $x$ provides the Lebesgue measure of
the set $[x, \mf L^+_{n+1}] \cap \{ y\in \bb R: S(y) =
S(\ell_n)\}$. Note that $G_0$ is non-increasing, that it is constant
on each valley of the landscape $\Lambda_n$, and that it vanishes if
the connected components of $\{\theta\in \bb T : b(\theta)=0\}$ are
points.

We turn to the definition of $G_1$. Denote by $\mf n^\pm_i$, $1\le
i\le q$, a local maximum $\mf M^\pm_i$ or a local minimum $\mf
m^\pm_i$ of $S$, and let
\begin{equation}
\label{07b}
\omega_+(\mf n^+_i) \;=\; \sqrt{\frac{\pi}{2\, |\, b'(\mf n^+_i+)| }} \;, \quad
\omega_-(\mf n^-_i) \;=\; \sqrt{\frac{\pi}{2\, |b'(\mf n^-_i-)|}} \;.
\end{equation}
Recall the definition of the valleys $\ms V_{n,j}$, $1\le j\le r_n$,
introduced in \eqref{54}, and that $\mf M^+_{a(n,0)} = \ell^+_{n}$.
Denote by $G_1 : \Lambda_n \to \bb R_+$ the function given by
\begin{equation}
\label{27b}
\begin{aligned}
& G_1(x) \; =\; \omega_+(\mf M^+_{a(n,0)})\, \mb 1\{b(\mf M^+_{a(n,0)})=0 \,,\,
x \le  \mf M^+_{a(n,0)} \} \\
&\quad + \; \sum_{j=1}^{r_n} \Big[ \omega_-(\mf M^-_{a(n,j)}) \mb
1\{x < \mf M^-_{a(n,j)} \} \;+\; \omega_+(\mf M^+_{a(n,j)}) \mb
1\{x \le  \mf M^+_{a(n,j)} \} \Big] \;.
\end{aligned}
\end{equation}

\begin{remark}
\label{rm5}
The function $G_1$ is non-increasing. It may be discontinuous at $\mf
M^+_{a(n,0)}$, it is discontinuous at the points $\mf M^-_{a(n,j)}$,
$\mf M^+_{a(n,j)}$, $1\le j\le r_n$, and it is constant on the valleys
$\ms V_{n,k} = (\mf M^+_{a(n,k)} , \mf M^-_{a(n,k+1)})$, $0\le
k< r_n$.  Actually, we defined the valleys as open intervals instead
of closed ones for the last property to hold.
\end{remark}

We turn to the definition of the pre-factor on the saddle intervals.
Fix $1\le n\le \bs m$ and consider the set $\Sigma_n = (\mf L^+_{n} ,
\ell_n) = (\mf M^+_{j(n)}, \mf u_{j(n)})$.  This set may contain
connected components of the set $\{\theta\in \bb T: b(\theta) =0\}$.
Denote by $s_n\ge 0$ the number of such components and by $[c^-_{n,1},
c^+_{n,1}], \dots , [c^-_{n,s_n}, c^+_{n,s_n}]$ the components. Note
that some of these intervals might be points: $c^-_{n,i}$ may be equal
to $c^+_{n,i}$.  Assume that these intervals are ordered in the sense
that $c^-_{n,i} < c^-_{n,i+1}$.

Let
\begin{equation}
\label{32}
\ms F_n \;=\; \bigcup_{j=1}^{s_n} [c^-_{n,j} , c^+_{n,j}] \;,\quad
\ms G_n \;=\; \Sigma_n \setminus  \ms F_n\;,\quad
\ms F \;=\; \bigcup_{n=1}^{\bs m} \ms F_n\;, \quad
\ms G \;=\; \bigcup_{n=1}^{\bs m} \ms G_n \;.
\end{equation}
Define $G_0 \colon \Sigma_n \to \bb R_+$ as
\begin{equation*}
G_0(x) \; =\; z(x) \,-\, x\;.
\end{equation*}
As $S$ is non-increasing on $\Sigma_n$, $G_0$ vanishes on $\ms G_n$
and $G_0(x) = c^+_{n,j} - x$ on the interval
$[c^-_{n,j} , c^+_{n,j}]$:
\begin{equation*}
G_0(x) \; =\; \sum_{j=1}^{s_n} (c^+_{n,j} -x)\, 
\chi_{[c^-_{n,j} , c^+_{n,j}]}(x) \;.
\end{equation*}
In this formula and below, $\chi_A$, $A\subset \bb R$, represents the
indicator function of the set $A$:
\begin{equation*}
\chi_A(x) \;=\; 1 \;\;\text{for}\;\; x\in A\;,\quad
\chi_A(x) \;=\; 0 \;\;\text{otherwise}\;.
\end{equation*}

Define $G_1 \colon \Sigma_n \to \bb R_+$ as
\begin{equation*}
G_1(x) \; =\; \sum_{j=1}^{s_n} \omega_+(c^+_{n,j})\, 
\chi_{[c^-_{n,j} , c^+_{n,j}]}(x) \;.
\end{equation*}
As $G_0$, the function $G_1$ vanishes on on $\ms G_n$.  Finally,
define $G_2 \colon \Sigma_n \to \bb R_+$ as
\begin{equation*}
G_2(x) \; =\; \frac 1{b(x)} \, \chi_{\ms G_n} (x)\;.
\end{equation*}
We are now in a position to present a sharp asymptotics for
$\pi_\epsilon(\cdot)$. 

\begin{proposition}
\label{l01}
Assume that $q\ge 1$. Then, 
\begin{enumerate}
\item (Sharp estimate on the landscapes)
\begin{equation*}
\lim_{\epsilon\to 0} \sup_{x\in\Lambda} \frac 1{\sqrt{\epsilon}}  \, 
e^{\widehat V(x)/\epsilon} \Big| \, \pi_\epsilon (x) \,-\,
e^{-\widehat V(x)/\epsilon} \big\{\, G_0(x) \,+\,
\sqrt{\epsilon}\, G_1(x)  \big\} \Big| \;=\; 0\;.
\end{equation*}
\item (Sharp estimate on the saddle intervals) On the set $\ms F$,
\begin{equation*}
\pi_\epsilon (x) \;=\; 
\big\{ G_0 (x) \;+\; \sqrt{\epsilon}\, G_1 (x) 
\;+\; o(\sqrt{\epsilon}) \big\} \; e^{ - \widehat V(x)/ \epsilon}\;,
\end{equation*}
and on the set $\ms G$,
\begin{equation*}
\pi_\epsilon (x) \;=\;  
[1+o(1)] \, \epsilon \, G_2 (x) \; e^{ - \widehat V(x)/ \epsilon}\;.
\end{equation*}
\end{enumerate}
In these formulas and below, $o(\sqrt{\epsilon})$, resp. $o(1)$ represent
quantities [which may depend on $x$] with the property that
$\lim_{\epsilon \to 0} o(\sqrt{\epsilon})/ \sqrt{\epsilon}=0$, resp.
$\lim_{\epsilon \to 0} o(1)=0$.
\end{proposition}

We turn to the normalizing constant $c(\epsilon)$. For $1\le k\le q$,
let $\sigma (\mf m^+_k)$ be the weights given by
\begin{equation}
\label{39}
\sigma (\mf m^+_k) 
\;=\; \omega_+(\mf m^+_k) \,+\, \omega_-(\mf m^-_k)\;,
\end{equation}
where the weights $\omega_\pm$ have been introduced in \eqref{07b}. 
Denote by $H$ the depth of the deepest well,
\begin{equation*}
H \;=\; -\, \min_{\theta\in\bb T}  \widehat V (\theta) \;=\;
\max_{0\le x<1} \big\{ S(z_x) - S(x) \big\} \;.
\end{equation*}
Clearly, $H>0$ because
$H\ge S(z(\mf m^+_1)) - S(\mf m^+_1) = S(z(\mf M^+_2)) - S(\mf m^+_1)
\ge S(\mf M^+_2) - S(\mf m^+_1) > 0$.
Let $\bb I$ be the set given by
\begin{equation}
\label{37}
\bb I \;=\; \big\{ j\in \{1, \dots, q\} : \widehat V (\mf m^+_j) = -\,
H\big\} \;,
\end{equation}
and let $Z_\epsilon$ be the normalizing constant given by
\begin{equation*}
Z_\epsilon\;=\; 
\sum_{j\in\bb I} \Big\{ G_0(\mf m^+_j)  \;+\;
\sqrt{\epsilon}\,  G_1(\mf m^+_j) \, \Big\}\,
\Big\{ \, [\mf m^+_j - \mf m^-_j ] \,+\, 
\sqrt{\epsilon} \, \sigma (\mf m^+_j) \, \Big\}\;.
\end{equation*}

\begin{proposition}
\label{l02}
Assume that $q\ge 1$. Then, 
\begin{equation*}
c(\epsilon) \;=\; \big[ 1 + o(\sqrt{\epsilon}) \big]\, Z_\epsilon
\, e^{H /\epsilon} \; .
\end{equation*}
\end{proposition}

Of course, a sharp asymptotic for the pre-factor $F_\epsilon$ can be
derived from Propositions \ref{l01} and \ref{l02}.

\smallskip\noindent{\bf \ref{sec0}.3. Proofs.}  We present in this
subsection the proofs of Propositions \ref{l01} and \ref{l02}.  We
start with an elementary observation.

\begin{lemma}
\label{as06}
The quasi-potential $V(\cdot)$ is continuous.
\end{lemma}

\begin{proof}
This result follows from Assertion \ref{as13}. On each landscape
$\Lambda_n$ the quasi-potential $\widehat V$ differs from $S$ by an
additive constant. At the boundary of the landscape $\Lambda_n$,
$\widehat V(\ell_n) = \widehat V(\mf L^+_{n+1}) =0$. On each saddle
interval $\Sigma_n$, $\widehat V$ vanishes, which proves the
continuity of $\widehat V$, and therefore the one of $V$.
\end{proof}

We continue with a uniform bound for the density $\pi_\epsilon$ on the
landscapes.

\begin{lemma}
\label{l09}
There exists a continuous function $\Xi:\bb R_+ \to \bb R_+$ vanishing
at the origin such that for each $1\le n\le \bs m$,
\begin{equation*}
\sup_{x\in\Lambda_n} \frac 1{\sqrt{\epsilon}}  \, 
e^{\widehat V(x)/\epsilon} \Big| \, \pi_\epsilon (x) \,-\,
e^{-\widehat V(x)/\epsilon} \big\{\, G_0(x) \,+\,
\sqrt{\epsilon}\, G_1(x)  \big\} \Big| \;\le\;
\Xi\,(\epsilon) \;.
\end{equation*}
\end{lemma}

\begin{proof}
Fix a landscape $\Lambda_n = [\ell_n, \mf L^+_n]$ and $x\in
\Lambda_n$. Since $S(z(x)) = S(\ell_n)$ on this landscape, rewrite
$\pi_\epsilon (x)$ as
\begin{equation*}
e^{[S(\ell_n) - S(x)]/\epsilon}
\int_{x}^{x+1} e^{ [S(y)-S(\ell_n)]/\epsilon}\, dy \;=\;
 e^{-\widehat V(x)/\epsilon}
\int_{x}^{x+1} e^{ [S(y)-S(\ell_n)]/\epsilon}\, dy\;.
\end{equation*}
It remains to estimate the integral. Note that $S(y)-S(\ell_n) \le 0$
for $y\ge x$. 

The integral is estimated in three steps. Recall from \eqref{54} the
definition of the local maxima $\ell^+_n < \mf M^+_{a(n,1)} < \dots <
\mf M^+_{a(n,r_n-1)} < \mf L^-_{n+1}$ of $S$ such that $S(\mf
M^+_{a(n,k)}) = S(\ell_n)$.  We first consider the integral over the
intervals $[\mf M^-_{a(n,k)}, \mf M^+_{a(n,k)}]$. Then, over the
intervals $[\mf M^+_{a(n,k)} + \eta, \mf M^+_{a(n,k+1)} -\eta]$ for
some $\eta>0$. Finally on the sets $[\mf M^+_{a(n,k)}, \mf
M^+_{a(n,k)} +\eta]$ and $[\mf M^-_{a(n,k)} - \eta, \mf
M^-_{a(n,k)}]$.

Let $\ms N_n$ be the set of points $x'$ in the landscape $\Lambda_n$
such that $S(x')= S(\ell_n)$. With the notation just introduced,
\begin{equation}
\label{68}
\ms N_n \;=\; \{x'\in \Lambda_n : S(x') = S(\ell_n)\} \;=\;
\bigcup_{k=0}^{r_n} \, [\, \mf M^-_{a(n,k)} \,,\, \mf M^+_{a(n,k)}]\;, 
\end{equation}
provided $\mf M^\pm_{a(n,r_n)} = \mf L^\pm_{n+1}$, $\mf M^\pm_{a(n,0)}
= \ell^\pm_{n}$, $\ell^-_{n} = \ell_{n}$. Of course, some of these
intervals may be reduced to points. Since $S(y)= S(\ell_n)$ on $\ms
N_n$,
\begin{equation*}
\int_{x}^{x+1} e^{ [S(y)-S(\ell_n)]/\epsilon}\, dy \;=\;
\int_{x}^{x+1} \chi_{_{\ms N_n}} (y) \, dy \;+\; 
\int_{[x,x+1]\setminus \ms N_n}   e^{ [S(y)-S(\ell_n)]/\epsilon}\,
dy\;. 
\end{equation*}
The first term on the right hand side is equal to $G_0(x)$. 

We turn to the second integral. We first estimate the integral over
open intervals between the maxima. Consider each local maximum $\mf
M^+_{a(n,k)}$, $1\le k\le r_n$. Note that the first one, $\mf
M^+_{a(n,0)} = \ell^+_n$, has not been included and will be treated
separately. At each of these points $b(\mf M^+_{a(n,k)})=0$ and, by
assumption (H3), $b'(\mf M^+_{a(n,k)}+)>0$. Choose $\eta>0$ small
enough such that $b'(y)\ge (1/2) \, b'(\mf M^+_{a(n,k)}+)$ for all $y
\in [\mf M^+_{a(n,k)},\mf M^+_{a(n,k)}+\eta]$ and all $k$.

Repeat the same procedure for the left endpoints $\mf M^-_{a(n,k)}$,
$1\le k\le r_n$. For the point $\mf M^+_{a(n,0)} = \ell^+_n$, either
$b(\ell^+_n)=0$ or $b(\ell^+_n)>0$. In the former case, by assumption
(H3), $b'(\ell^+_n+)>0$, and we may choose $\eta>0$ small enough such
that $b'(y)\ge (1/2) \, b'(\ell^+_n+)$ for all $y \in [\ell^+_n,
\ell^+_n + \eta]$. In the latter case, choose $\eta>0$ such that $b(y)
\ge b(\ell^+_n)/2$ for all $y\in [\ell^+_n, \ell^+_n + \eta]$.

Recall that the landscape $\Lambda_n = [\ell_n, \mf L^+_{n+1}] = [\mf
M^-_{a(n,0)}, \mf M^+_{a(n,r_n)}]$. Hence, for any $x\in \Lambda_n$,
the interval $[x,x+1]$ is contained in $[\ell_n, \mf L^+_{n+1} +1]$.
Let $\ms C_n\subset \bb R$ be the closed set given by
\begin{equation*}
\ms C_n \;=\; 
[\ell^+_n + \eta , \mf L^+_{n+1} +1] \;\setminus\; \Big\{ 
\bigcup_{k=1}^{r_n} \,  (\, \mf M^-_{a(n,k)}-\eta \,,\, \mf
M^+_{a(n,k)}+\eta \, ) \Big\}
\end{equation*}
For any $y\in \ms C_n$, $S(y) < S(\ell_n)$. There exists, therefore, a
constant $c(\eta)>0$ such that $S(y) \le S(\ell_n) - c(\eta)$ for all
$y\in \ms C_n$. Hence,
\begin{equation*}
\int_x^{x+1} \chi_{\ms C_n} (y)\,   
e^{ [S(y)-S(\ell_n)]/\epsilon}\, dy 
\;\le\; e^{-c(\eta)/\epsilon}\;.  
\end{equation*}

In view of formula \eqref{68} for the set $\ms N_n$, it remains to
estimate the integral on the intervals $[\ell^+_n, \ell^+_n + \eta ]$,
$[\mf M^-_{a(n,k)}-\eta , \mf M^-_{a(n,k)}]$, $[\mf M^+_{a(n,k)}, \mf
M^+_{a(n,k)} + \eta ]$. Let
\begin{equation*}
\ms D_n \;=\; 
[\ell^+_n , \ell^+_n + \eta ] \bigcup_{k=1}^{r_n} \,  [\, \mf M^-_{a(n,k)}-\eta \,,\, \mf
M^-_{a(n,k)}\,] \bigcup_{k=1}^{r_n} \,  [\, \mf M^+_{a(n,k)} \,,\, \mf
M^+_{a(n,k)} + \eta \,] \;.
\end{equation*}
Assume that $b(\ell^+_n)=0$.  By Assertions \ref{as09a} and \ref{as09b} below,
\begin{equation*}
\Big|\, \int_{[x,x+1] \cap \ms D_n}   e^{ [S(y)-S(\ell_n)]/\epsilon}\,
dy \,-\, \sqrt{\epsilon}\, G_1(x) \, \Big| \;\le\; \sqrt{\epsilon}
\;\Xi\,(\epsilon)\;,
\end{equation*}
where 
\begin{equation*}
\Xi\,(\epsilon) \;=\; C_0 \, \Xi\,( b'(\ell^+_n+) , \epsilon) \;+\; 
C_0 \sum_{k=1}^{r_n} \Xi\,\big( \, b'(\mf M^\pm_{a(n,k)} \pm) \,,\, \epsilon\big) \;.
\end{equation*}
In the second sum, it has to be understood that there are two sums,
one for the terms $\mf M^-_{a(n,k)}$ and one for $\mf M^+_{a(n,k)}$.

If $b(\ell^+_n)>0$, $\ell_n = \ell^+_n$, and by the choice of $\eta$
and Assertion \ref{as10b} below,
\begin{equation*}
\Big|\, \int_{[x,x+1] \cap [\ell^+_n, \ell^+_n + \eta]}   e^{ [S(y)-S(\ell_n)]/\epsilon}\,
dy \, \Big| \;\le\; \frac {2\,\epsilon}{b(\ell_n)}\;\cdot
\end{equation*}
This completes the proof of the lemma.
\end{proof}

\smallskip\noindent{\bf Proof of Proposition \ref{l01}}.  Fix $x\in
\bb T$. The case where $x$ belongs to some landscape has been
considered in the previous lemma.  Consider a saddle interval
$\Sigma_n$. Recall the definition of the intervals $[c^-_{n,j},
c^+_{n,j}]$, $1\le j\le s_n$ introduced in \eqref{32}. If $x\in
[c^-_{n,j}, c^+_{n,j}]$, since $z(x) = c^+_{n,j}$
\begin{equation*}
\pi_\epsilon (x) \;=\; e^{-\widehat V(x)/\epsilon}
\int_{x}^{x+1} e^{ [S(y)-S(c^+_{n,j})]/\epsilon}\, dy\;.
\end{equation*}
On the interval $[x, c^+_{n,j}]$, $S(y)=S(c^+_{n,j})$. Hence the
integral on this interval is equal to $c^+_{n,j} - x = G_0(x)$. 

By assumption (H3), $b'(c^+_{n,j}+)>0$. Let $\eta>0$ such that
$b'(y)>b'(c^+_{n,j}+)/2$ for all $y\in [c^+_{n,j},c^+_{n,j}+\eta]$.
Since $S(y)<S(c^+_{n,j})$ for all $y>c^+_{n,j}$, there exists
$c(\eta)>0$ such that $S(y) \le S(c^+_{n,j}) - c(\eta)$ for all $y\ge
c^+_{n,j} + \eta$. Hence,
\begin{equation*}
\int_{c^+_{n,j}}^{x+1} e^{ [S(y)-S(c^+_{n,j})]/\epsilon}\, dy \;=\;
\int_{c^+_{n,j}}^{c^+_{n,j} + \eta} e^{
  [S(y)-S(c^+_{n,j})]/\epsilon}\, dy \;+\; R_\epsilon\;,
\end{equation*}
where $R_\epsilon \le e^{- c(\eta)/\epsilon}$. By Assertion
\ref{as09a}, the last integral is equal to $[1+o(1)]\,
\sqrt{\epsilon}\, G_1(x)$. This completes the proof in the case where
$x\in [c^-_{n,j}, c^+_{n,j}]$.

In the case where $x\in \ms G_n$, the statement of the proposition
follows from Assertion \ref{as10}. \qed

\begin{proof}[{\bf Proof of Proposition \ref{l02}.}]
Recall the definition of the set $\bb I$ introduced in \eqref{37}. Fix
$\eta>0$, to be chosen later, and let $\ms B_{\eta}$ be an
$\eta$-neighborhood of the global minima of $V$:
\begin{equation*}
\ms B_{\eta} \;=\; \bigcup_{j\in \bb I} 
(\mf m^-_j -\eta , \mf m^+_j +\eta )\;.
\end{equation*}
Since, by Lemma \ref{as06}, $\widehat V$ is continuous and since $\widehat
V>-H$ on the closed set $\ms B^c_{\eta}$, there exists $c(\eta)>0$ such that
\begin{equation}
\label{38}
\inf \big\{ \widehat V(\theta) : \, \theta \not\in \ms B_{\eta} \, \big\}
\;\ge\; - \, H \,+\, c(\eta)\;.
\end{equation}
Hence, as
\begin{equation*}
\sup \big\{ S(y) - S(x) : x\le y\le x+1 \big\} \;=\; 
S(z(x)) - S(x) \;= - \, \widehat V(x) \;,
\end{equation*}
for every $\eta>0$,
\begin{equation*}
\int_{\ms B_{\eta}^c} \pi_\epsilon (x)\, dx  \;=\; 
\int_{\ms B_{\eta}^c} dx \int_x^{x+1} e^{ [S(y)-S(x)]/\epsilon}\, dy \;\le\;
e^{[H \,-\, c(\eta)]/\epsilon}\; .
\end{equation*}

We examine the integral of $\pi_\epsilon$ on the set $\ms
B_{\eta}$. Each set $[\mf m^-_j , \mf m^+_j ]$ is contained in the
interior of a valley. Choose $\eta$ small enough for each $[\mf
m^-_j -\eta, \mf m^+_j +\eta]$ to be contained in the same valley.
In this case, by Lemma \ref{l09}, and since $G_0$ and $G_1$ are
constant in the valleys
\begin{equation*}
\int_{\mf m^-_j -\eta}^{\mf m^+_j +\eta} \pi_\epsilon (x)\, dx 
\;=\; \Big\{ G_0(\mf m^-_j) \,+\, \sqrt{\epsilon}\, G_1(\mf m^-_j)  \;\pm\;
\sqrt{\epsilon}  \, \Xi\,(\epsilon) \Big\}
\int_{\mf m^-_j -\eta}^{\mf m^+_j +\eta} e^{-\widehat V(x)/\epsilon}
\, dx \;.
\end{equation*}

Since $[\mf m^-_j -\eta, \mf m^+_j+\eta]$ is contained in a landscape,
since in each landscape $\widehat V$ and $S$ differ only by an
additive constant, and since $\widehat V(\mf m^+_j) = -\, H$, on $[\mf
m^-_j -\eta, \mf m^+_j+\eta]$, $\widehat V(x) = \widehat V(x) -
\widehat V(\mf m^+_j) - H = S(x) -S (\mf m^+_j) - H$. Hence,
\begin{equation*}
\int_{\mf m^-_j - \eta}^{\mf m^+_j + \eta}   e^{- \widehat
  V(x)/\epsilon} \, dx \;=\; e^{H /\epsilon}
\int_{\mf m^-_j - \eta}^{\mf m^+_j + \eta}   e^{- [S(x)
  -S (\mf m^+_j)]/\epsilon} \, dx\; .
\end{equation*}
Choose $\eta$ small enough to fulfill the assumptions of Assertions
\ref{as09a}, \ref{as09b} (with the obvious modifications since $b'(\mf
m^+_j)<0$). By these results,
\begin{equation*}
\int_{\mf m^-_j - \eta}^{\mf m^+_j + \eta} e^{- [S(x)
  -S (\mf m^+_j)]/\epsilon} dx 
\; =\;  \big\{ \, [\mf m^+_j - \mf m^-_j ] \,+\, [1+o(1)]\,
\sqrt{\epsilon} \, \sigma(\mf m_j^+)\, \big\} \;,
\end{equation*} 
where $\sigma(\mf m_j^+)$ has been introduced in \eqref{39}.

Putting together the previous estimates yields that
\begin{equation*}
c(\epsilon) \;=\; \sum_{j\in\bb I} \Big\{ G_0(\mf m^+_j)  \;+\;
\big[ 1 + o(1) \big]\, \sqrt{\epsilon}\,  G_1(\mf m^+_j) \, \Big\}\,
\big\{ \, [\mf m^+_j - \mf m^-_j ] \,+\, [1+o(1)]\,
\sqrt{\epsilon} \, \sigma(\mf m_j^+)\, \big\}\, e^{H /\epsilon} \; ,
\end{equation*}
which completes the proof of the proposition.
\end{proof}

\smallskip We conclude this section with some estimates used in the
proofs above. 

\begin{remark}
\label{rm7}
The proof of these estimates relies on a Taylor expansion of the
function $S$ around the local maxima of this function. We need in this
argument $S''$ [that is $b'$] to be Lipschitz continuous. It is for
this reason that we assumed $b$ to be in $C^2$ in the intervals
$[r_j,l_{j+1}]$. We could have assumed the weaker assumption that $b'$
is Lipschitz continuous on these intervals.
\end{remark}

Denote by $K_0$ the Lipschitz continuity constant of $b'$.

\begin{asser}
\label{as09a}
Let $x\in \bb R$ be a point such that $b(x)=0$, $b'(x+)>0$. Let
$\eta>0$ be such that $b'(y)\ge (1/2) b'(x+)$ for all $y \in
[x,x+\eta]$. Then, there exists a finite constant $C_0$, which depends
only on $K_0$, and a function $\Xi:\bb R^2_+ \to\bb R_+$ such that
$\lim_{\epsilon\to 0} \Xi\, (a,\epsilon)=0$ for all $a>0$, and for
which
\begin{equation*}
\Big|\, \int_x^{x+\eta} e^{ [S(y)-S(x)]/\epsilon}\, dy 
\,-\,  \sqrt{\frac{\pi\, \epsilon }{2\, b'(x+)}} \,\Big|
\;\le \; C_0 \, \sqrt{\epsilon}\;  \Xi\, (b'(x+) , \epsilon)\;.
\end{equation*}
\end{asser}

\begin{proof}
We derive an upper bound for the integral. The lower bound is obtained
by changing $+$ signs into $-$ signs.

Let $\delta= \delta(\epsilon)>0$ be a sequence such that $\delta^3 \ll
\epsilon \ll \delta^2$.  We first estimate the integral in the
interval $[x,x+\delta]$.  Since $S''(x +) = - b'(x+)< 0$ and $b'$ is
uniformly Lipschitz continuous, in view of the properties of $\delta$,
a Taylor expansion and a change of variables yield that
\begin{align*}
\int_{x}^{x+\delta} e^{  [S(y)-S(x)]/\epsilon}\, dy
\; &\le \; \big[ 1+ C_0 \, (\delta^3/\epsilon) \, \big]\, 
\int_{0}^{\delta} e^{-(1/2) b'(x+) z^2 /\epsilon}\, dz  \\
\; &\le \;  \big[ 1+ C_0 \, (\delta^3/\epsilon) \, \big]\, 
\sqrt{\frac{\pi\, \epsilon}{2\, b'(x+)}} \;.
\end{align*}

It remains to estimate the integral on the interval
$[x+\delta,x+\eta]$. By assumption, $S''(y) \le (1/2) S''(x)$ for all
$y\in [x,x+\eta]$. Hence,
\begin{gather*}
\int_{x+\delta}^{x+\eta} e^{ [S(y)-S(x)]/\epsilon}\, dy 
\;\le \; \int_{\delta}^{\infty} e^{S''(x+)y^2/4\epsilon}\, dy 
\;\le\; \sqrt{\epsilon} \; \Xi\, (b'(x+), \epsilon) \;.
\end{gather*}
This proves the assertion.
\end{proof}

The same argument yields the next assertion.
\begin{asser}
\label{as09b}
Let $x\in \bb R$ be a point such that $b(x)=0$, $b'(x-)>0$.  Let
$\eta>0$ be such that $b'(y)\ge (1/2) b'(x-)$ for all $y \in
[x-\eta,x]$. Then, there exists a finite constant $C_0$, which depends
only on $K_0$, and a function $\Xi:\bb R^2_+ \to\bb R_+$ such that
$\lim_{\epsilon\to 0} \Xi\, (a,\epsilon)=0$  for all $a>0$, and for which
\begin{equation*}
\Big|\, \int_{x-\eta}^{x} e^{  [S(y)-S(x)]/\epsilon}\, dy 
\,-\,  \sqrt{\frac{\pi\, \epsilon }{2\, b'(x-)}} \,\Big|
\;\le \; C_0 \, \sqrt{\epsilon}\;  \Xi\, (b'(x-) , \epsilon)\;.
\end{equation*}
\end{asser}

It remains to consider the case where $b(x)>0$.

\begin{asser}
\label{as10b}
Let $x\in \bb R$ be a point such that $b(x)>0$. Let $\eta>0$ such that
$b(y)\ge b(x)/2$ for all $y\in [x,x+\eta]$. Then,
\begin{equation*}
\int_{x}^{x+\eta} e^{[S(y)-S(x)]/\epsilon}\, dy 
\;\le \; \frac{2\, \epsilon}{b(x)}\;\cdot
\end{equation*}
\end{asser}

\begin{proof}
By a Taylor expansion and by hypothesis,
\begin{equation*}
\int_{x}^{x+\eta} e^{[S(y)-S(x)]/\epsilon}\, dy \;\le\;
\int_{x}^{x+\eta} e^{-b(x)(y-x) /2\epsilon}\, dy \;\le\;
\int_{0}^{\infty} e^{-b(x) z /2\epsilon}\, dz \;=\;
\frac{2\, \epsilon}{b(x)}\;\cdot
\end{equation*}
\end{proof}

If we assume that $S(y)< S(x)$ for all $x< y$, we may estimate the
integral over the interval $[x, x+1]$.

\begin{asser}
\label{as10}
Let $x\in \bb R$ be a point such that $b(x)>0$.  Assume,
furthermore, that $S(y)< S(x)$ for all $y\in(x,x+1]$. Then,
\begin{equation*}
\int_{x}^{x+1} e^{[S(y)-S(x)]/\epsilon}\, dy 
\;=\; [1+o(1)]\, \frac{\epsilon}{b(x)}\;\cdot
\end{equation*}
\end{asser}

\begin{proof}
Let $\delta= \delta(\epsilon)>0$ be a sequence such that $\delta^2 \ll
\epsilon \ll \delta$. By the Taylor expansion and an elementary
computation, as $S'(x) = - b(x)$,
\begin{equation*}
\int_x^{x+\delta} e^{[S(y)-S(x)]/\epsilon}\, dy 
\;=\; \int_0^{\delta} e^{[-\, b(x)\, y \,+\,  O(\delta^2)]/\epsilon}\, dy 
\;=\; [1+o(1)] \, \frac{\epsilon}{ b (x)} \;\cdot
\end{equation*}

Let $\alpha = b(x)/2>0$. There exists $\eta>0$ such that
$S(y) - S(x) \le - \alpha (y - x) $ for all $y\in [x,x+\eta]$.  On the
other hand, since $S(y)<S(x)$ for all $y\in [x+\eta, x+1]$ and since
$S$ is continuous, there exists $\kappa>0$ such that
$S(y)\le S(x) - \kappa$ for all $y\in [x+\eta, x+1]$. Therefore,
\begin{gather*}
\int_{x+\delta}^{x+\eta} e^{[S(y)-S(x)]/\epsilon}\, dy 
\;\le \; \int_{\delta}^{\eta} e^{- (\alpha/\epsilon)\, y}\, dy 
\;=\; o(1) \, \epsilon \;, \\
\int_{x+\eta}^{x+1} e^{[S(y)-S(x)]/\epsilon}\, dy
\;\le \;  e^{-\kappa/\epsilon} \;=\; o(1) \,\epsilon\;.
\end{gather*}
The assertion follows from the three previous estimates.
\end{proof}

\smallskip\noindent{\bf \ref{sec0}.4. When the set $\{\theta \in \bb T
  : b(\theta)=0\}$ is finite.} We present in this subsection a formula
for the pre-factor in the case where the connected components of the
set $\{\theta \in \bb T : b(\theta)=0\}$ are points.

\smallskip\noindent ({\bf H4}) Assume that the connected components of
the set $\{\theta \in \bb T : b(\theta)=0\}$ are points, that $b$ is
of class $C^2(\bb T)$ and that $b'(\theta) \not = 0$ for all
$\theta\in \bb T$ such that $b(\theta)=0$ [that is $S''(x)\not =
0$ at the critical points of $S$]. \smallskip

Note that these assumptions imply that $\ell^+_n = \ell_n$, $b(\ell_n)
>0$ for all left endpoints of a landscape and that $\mf M_k^+=\mf
M_k^-$, $\mf m_k^+=\mf m_k^-$ for all $k$. Moreover, the sets $\ms
F_n$ introduced in \eqref{32} are empty, so that $\Sigma_n =\ms G_n$.

Set $\mf M_k := \mf M_k^+$, $\mf m_k :=\mf m_k^+$, $\mf L_k:=\mf
L_k^+$ for all indices $k$.  Fix a landscape $\Lambda_n$. Under the
previous hypotheses, $G_0 \equiv 0$ and $G_1$ is given by
\eqref{27b}. In a saddle interval $\Sigma_n$, $G_0\equiv 0$ and
$G_1\equiv 0$, while the function $G_2$ is unchanged.  The weights
$\omega(\mf M_k)$, $\sigma(\mf m_k)$, $1\le k\le q$, become
\begin{equation*}
\omega(\mf M_k) \;=\; \sqrt{\frac{2\,\pi}{ b'(\mf M_k) }} \;,\quad
\sigma(\mf m_k) \;=\; \sqrt{\frac{2\, \pi}{-\, b'(\mf m_k) }} \;\cdot
\end{equation*}
Set
\begin{equation*}
Z\;=\; \sum_{j\in\bb I} G_1(\mf m_j) \, \sigma (\mf m_j)\;.
\end{equation*}
By the definition of $Z_\epsilon$ and by Propositions \ref{l02},
$Z_\epsilon=\epsilon\, Z$ and 
\begin{equation}
\label{75}
c(\epsilon) \;=\; \big[ 1 + o(\sqrt{\epsilon}) \big]\, Z\, \epsilon
\, e^{H /\epsilon} \; . 
\end{equation}
Thus, Proposition \ref{l01} can be restated in this context as follows. 

\begin{proposition}
\label{p02}
Assume that hypotheses (H4) are in force. Then,  
\begin{enumerate}
\item(Pre-factor on the landscapes)
\begin{equation*}
\lim_{\epsilon\to 0} \sup_{x\in\Lambda} 
\sqrt{\epsilon} \, e^{V(x)/ \epsilon}  \,
\Big| m_\epsilon (x) \,-\, \frac 1{Z} \; \frac 1{\sqrt{\epsilon}}
\; G_1(x) \; e^{ -V(x)/ \epsilon}\,\Big| \;=\; 0\;.
\end{equation*}
\item(Pre-factor on the saddle intervals)
\begin{equation*}
m_\epsilon (x) \; =\; [1+o(1)]\; \frac 1{Z} \;
G_2(x) \; e^{ -V(x)/ \epsilon} \;, \quad x\in \Sigma \;.
\end{equation*}
\end{enumerate}
\end{proposition}

\begin{remark}
\label{rm9}
The results of this article remain in force if we add a
$(d-1)$-transversal drift. More precisely, consider the diffusion on
$\bb T^d$ given by
\begin{equation*}
dX_\epsilon (t) \;=\; \bs b(X_\epsilon (t))\, dt \;+\; \sqrt{2\epsilon} \,
dW_t\;,
\end{equation*}
where $W_t$ is a Brownian motion on $\bb T^d$, and $\bs b = (b_1,
\dots, b_d): \bb T^d \to \bb R$ a drift.  The same results hold
provided that
\begin{equation*}
b_1(x_1, \dots, x_d) \,=\, b_1(x_1)\quad\text{and}\quad
\sum_{j=2}^d (\partial_{x_j} b_j)(x) \,=\, 0\;.
\end{equation*}
\end{remark}

\smallskip\noindent{\bf \ref{sec0}.5. The Hamilton-Jacobi equation.}
We examine in this subsection the asymptotic behavior, as
$\epsilon\to 0$, of the solution of the Hamilton-Jacobi equation
satisfied by the pre-factor of the stationary measure. We consider
this problem under the assumptions ({\bf H4}).

Since $m_\epsilon$ is the density of the stationary state,
\begin{equation}
\label{76}
\epsilon \, m_\epsilon'' -  (b \, m_\epsilon)'\;=\; 0\;.
\end{equation}
Since the quasi-potential $V$ is not continuously differentiable, but
only smooth by parts, we consider the previous equation separately on
the landscapes $\Lambda_n$ and on the saddle intervals $\Sigma_n$,
$1\le n\le \bs m$.

Inserting expression \eqref{i01} for the stationary state
$m_\epsilon$ in \eqref{76} yields the following equation:
\begin{equation}
\label{16}
\epsilon \, F_\epsilon'' \,+\, F_\epsilon'\, b \;=\; 0
\;\; \text{on $\Lambda$ }\quad \text{and}\quad
\epsilon \, F_\epsilon'' \,-\, F_\epsilon'\, b \,-\, F_\epsilon\, b'
\;=\; 0 \;\; \text{on $\Sigma$}\;,
\end{equation}
which is the Hamilton-Jacobi equation for the pre-factor.

Denote by $F$ the solution of the Hamilton-Jacobi equation \eqref{16}
with $\epsilon=0$.  Clearly, there exist constants $c_0$ and $c_1$
such that
\begin{equation}
\label{25}
\begin{aligned}
& F(\theta) \,=\, c_1 \;\; 
\text{at each connected component 
of $\{\theta : b(\theta) \not = 0\}  \cap \Lambda$}\;, \\
&\quad \quad
F(\theta) \;=\; \frac {c_0}{b(\theta)} \;\; 
\text{at  each connected component  of $\Sigma$} \;.
\end{aligned}
\end{equation}
Note that the constants may different on distinct connected
components. 

We now compare \eqref{25} with the asymptotic behavior, as
$\epsilon \to 0$, of the solution of the Hamilton-Jacobi equation on
the set $\Lambda$.  The solution is given by
\begin{equation*}
F_\epsilon (\theta) \;=\; c_0 \;+\; c_1 \, \int_{\theta_0}^\theta e^{S(y)/\epsilon} \, dy \;,
\end{equation*}
for $c_0$, $c_1\in \bb R$, $\theta_0\in \bb T$. 

Recall from \eqref{26} that the connected component $\Lambda_n$ of
$\Lambda$ are intervals of the form $(\ell_n, \mf L_{n+1})$,
$1\le n \le \bs m$.  Keep in mind that $\mf L_{n+1}$ is a local
maximum of $S$ and $\ell_n$ a point such that $S'(\ell_n)<0$.
Moreover, $S(\ell_n)=S(\mf L_{n+1})$ and $S(\theta)\le S(\ell_n)$ for all
$\theta\in (\ell_n, \mf L_{n+1})$.

For $F_\epsilon (\theta)$ to converge at $\theta=\mf L_{n+1}$ to a non
trivial value, we have to choose $c_1$ as
$c'_1 \epsilon^{-1/2} \exp\{-S(\ell_n)/\epsilon\}$ for some
$c'_1\in\bb R$.  In contrast, the choice of $\theta_0$ is not
important. With this choice,
\begin{equation}
\label{17}
F_\epsilon (\theta) \;=\; c_0 \;+\; c'_1 \, \frac 1{\sqrt{\epsilon}}\,
\int_{\theta_0}^\theta e^{[S(y)-S(\ell_n)]/\epsilon} \, dy \;.
\end{equation}

The next result follows from the calculations presented in
Assertions \ref{as09a} -- \ref{as10}.

\begin{asser}
\label{as2} 
Fix $\theta_0 \in (\ell_n, \mf L_{n+1})$ and consider $F_\epsilon$ given by
\eqref{17}.  Then, for all $\theta\in [\ell_n, \mf L_{n+1}]$,
\begin{equation*}
\bs F(\theta) := \lim_{\epsilon\to 0} F_\epsilon (\theta) \;=\; c_0 \,+\,  c'_1
\, \big[ G_1(\theta) - G_1(\theta_0)\big]\;. 
\end{equation*}
\end{asser}

The function $\bs F$ inherits the properties of $G_1$, it is constant
in the valleys $\ms V_{n,j}$, $1\le j\le r_n$, and discontinuous at
the local maxima $\mf M^+_{a(n,j)}$, unless $c_1'=0$. In particular,
it fulfills the conditions in the first line of \eqref{25}. 

We set the value of $c_1$ for $F_\epsilon(\mf L_{n+1})$ to
converge. Choosing $c_1$ for $F_\epsilon(\theta_1)$ to converge, for some
$\theta_1\in \Lambda_n$ such that $S(\theta_1) < S(\ell_n)$, would produce a
limit equal to $\pm\infty$ at every point $y$ such that $S(y)>S(\theta_1)$.
\medskip

We turn to the set $\Sigma$. Fix a connected component $\Sigma_n =
(\mf L_n, \ell_n)$. An elementary computation yields that the solution
of equation \eqref{16} on $\Sigma_n$ is given by
\begin{equation}
\label{15}
F_\epsilon (\theta) \;=\; \frac 1{\epsilon}\, \Big\{ c_0 \int_{\theta_0}^\theta
e^{[S(y)-S(\theta)]/\epsilon} \, dy \;+\; c_1 e^{-S(\theta)/\epsilon} \Big\}
\end{equation}
for constants $c_0$, $c_1\in \bb R$, which may depend on $\epsilon$,
and some $\theta_0\in \Sigma_n$ which may also depend on $\epsilon$.

\begin{asser}
\label{as1}
There are no choices of the constants $c_0(\epsilon)$,
$c_1(\epsilon)$, $\theta_0(\epsilon)$ for which $F_\epsilon$ has a
non-trivial limit as $\epsilon\to 0$.
\end{asser}
 
\begin{proof}
If we set $\theta_0=\mf L_n$, a Taylor expansion yields that \eqref{15}
is equal to
\begin{equation*}
F_\epsilon (\theta) \;=\; \frac 1{\epsilon}\, \Big\{ [1+o(1)]\, c_0 \, 
\sqrt{\frac{\epsilon \pi}{2|S''(\mf L_n)|}} \, e^{S(\mf L_n)/\epsilon}  \,+\,
c_1 \Big\} \, e^{-S(\theta)/\epsilon}\;.
\end{equation*}
The expression inside braces is a function of $\epsilon$ which can
compensate the factor $\epsilon^{-1}$ or which can be of a smaller
order. In any case, this constant is multiplied by
$\exp\{-S(\theta)/\epsilon\}$ which may converges for one specific
$\theta\in \Sigma_n$ but which will diverge for all other
$\theta$. Hence, if $\theta_0=\mf L_n$ there is no choice of
$c_0(\epsilon)$, $c_1(\epsilon)$ which provide a non-trivial limit for
\eqref{15}. A similar analysis can be carried through if $\theta_0$ is
chosen in $(\mf L_n,\ell_n]$, which proves the assertion.
\end{proof}

The previous assertion shows that on the set $\Sigma$ the solution
$F_\epsilon(\theta)$ of \eqref{16} does not converge, as $\epsilon\to 0$,
to the solution $F(\theta)$ of \eqref{25} unless we consider the trivial
solutions $F_\epsilon(\theta) = F(\theta) =0$.

\section{Equilibrium potential and capacities}
\label{sec1}

We estimate in this section capacities between wells. We start with an
explicit formula for the adjoint of $L _\epsilon$ in
$L^2(\mu_\epsilon)$, the Hilbert space of measurable functions $f:\bb
T \to\bb R$ endowed with the scalar product given by
\begin{equation*}
\< f, g\>_{\mu_\epsilon} \;=\; \int_{\bb T} f(\theta)\, g(\theta)\,
\mu_\epsilon(d\theta)\;. 
\end{equation*}

Integrating the equation \eqref{76} once provides that
\begin{equation}
\label{12}
 m_\epsilon'(\theta) \;=\; -\, R_\epsilon \,+\,
\frac 1{\epsilon} \, b(\theta)\,  m_\epsilon (\theta) \;, \quad\text{where}
\;\; R_\epsilon\;=\;  \frac 1{c(\epsilon)} \, \Big(
1 - e^{-B/\epsilon} \Big)\;.
\end{equation}
Note that $R_\epsilon$ is positive and that it vanishes if $B=0$.  

Denote by $L^*_\epsilon$ the adjoint operator of $L_\epsilon$ in
$L^2(\mu_\epsilon)$. It follows from \eqref{12} that for every twice
continuously differentiable function $f:\bb T\to\bb R$,
\begin{equation*}
(L^*_\epsilon f)(\theta) \;=\; \Big( b(\theta) \,-\, 
\frac{2\, \epsilon \, R_\epsilon}{m_\epsilon(\theta)} \Big)\,
f'(\theta) \;+\; \epsilon \, f''(\theta) \;.
\end{equation*}
In particular, $L^*_\epsilon = L_\epsilon$ if $B=0$, and the symmetric
part of the generator, denoted by $S_\epsilon = (1/2)(L_\epsilon +
L^*_\epsilon)$,  is given by
\begin{equation*}
(S_\epsilon f)(\theta) \;=\; \Big( b(\theta) \,-\, 
\frac{\epsilon \, R_\epsilon}{m_\epsilon(\theta)} \Big)\,
f'(\theta) \;+\; \epsilon \, f''(\theta) \;.
\end{equation*}

The Dirichlet form, denoted by $D_\epsilon(\cdot)$, associated to the
generator $L_\epsilon$ is given by
\begin{equation}
\label{14}
D_\epsilon (f)\;:=\; -\, \int_{\bb T} f(\theta) \,  (S_\epsilon f)(\theta) \,
m_\epsilon(\theta) \, d\theta
\;=\; \epsilon\, \int_{\bb T} [f'(\theta)]^2 \, m_\epsilon(\theta)\,
d\theta\;.  
\end{equation}

\smallskip\noindent{\bf Equilibrium potential and capacity.} Fix two
disjoint closed intervals $\ms A_1=[\theta^-_1,\theta^+_1]$,
$\ms A_2=[\theta^-_2,\theta^+_2]$ of $\bb T$. Without loss of
generality, we suppose that
\begin{equation}
\label{50}
0 \;\le\; \theta^-_1 \;\le\; \theta^+_1 \;<\; \theta^-_2 
\;\le\; \theta^+_2 <1\;.
\end{equation}
Note that we allow the intervals to be reduced to a point.  The unique
solution to the elliptic problem
\begin{equation*}
\begin{cases}
(L_\epsilon f)(\theta) \;=\; 0\;,\;\; 
\theta\in \bb T\setminus (\ms A_1\cup \ms A_2)\;, \\
f(\theta) \;=\; \chi_{\ms A_1} (\theta) \;,\;\; 
\theta\in \ms A_1\cup \ms A_2\;.    
\end{cases}
\end{equation*}
is called the equilibrium potential between the sets $\ms A_1$ and
$\ms A_2$, and is denoted by
$h_{\ms A_1,\ms A_2} = h_{\ms A_1,\ms A_2}^\epsilon$.

In dimension $1$, an explicit formula for the equilibrium potential is
available, a straightforward computation shows that
\begin{equation}
\label{11}
h_{\ms A_1,\ms A_2}(\theta) \;=\; \frac{\int_\theta^{\theta^-_2} e^{S(y)/\epsilon} dy}
{\int_{\theta^+_1}^{\theta^-_2} e^{S(y)/\epsilon} dy} \,
\chi_{[\theta^+_1 , \theta^-_2]} (\theta) \;+\;
\frac{\int_{\theta^+_2}^\theta e^{S(y)/\epsilon} dy}
{\int_{\theta^+_2}^{1+\theta^-_1} e^{S(y)/\epsilon} dy} \, 
\chi_{[\theta^+_2 , 1+\theta^-_1]} (\theta) \;.
\end{equation}

Define the capacity between $\ms A_1$ and $\ms A_2$ as the Dirichlet
form of the equilibrium potential:
\begin{equation*}
\Cap_\epsilon (\ms A_1,\ms A_2) \;:=\; D_\epsilon (h_{\ms A_1,\ms A_2}) 
\;=\; \epsilon\, \int_{\bb T} [h_{\ms A_1,\ms A_2}'(\theta)]^2 \, m_\epsilon(\theta)\,
d\theta\;.
\end{equation*}
We show in Assertion \ref{as03} below that
\begin{equation}
\label{30}
\Cap_\epsilon (\ms A_1,\ms A_2) \;=\; \epsilon\, 
\big\{ h_{\ms A_1,\ms A_2}'(\theta^-_1) \, m_\epsilon(\theta^-_1)
\,-\, h_{\ms A_1,\ms A_2}'(\theta^+_1) \, m_\epsilon(\theta^+_1)\big\}\;.
\end{equation}
Moreover, as $h_{\ms A_2,\ms A_1} = 1 - h_{\ms A_1,\ms A_2}$,
\begin{equation*}
\Cap_\epsilon (\ms A_2,\ms A_1) \;=\; \Cap_\epsilon (\ms A_1,\ms A_2)
\;.
\end{equation*}

\smallskip\noindent{\bf Estimation of Capacity.} We present in
Propositions \ref{p04}--\ref{p03} below sharp estimates of the
capacity between two sets which satisfy the conditions below.

Assume that the intervals $\ms A_1=[\theta^-_1,\theta^+_1]$,
$\ms A_2=[\theta^-_2,\theta^+_2]$ represent wells (cf. Section
\ref{sec3}.1) in the sense that
\begin{equation}
\label{44}
V(\theta) \;<\; V(\theta^-_i) \;=\; V(\theta^+_i)
\;<\; H \quad \text{for all} \quad 
\theta\in(\theta^-_i ,\, \theta^+_i)\;, \;\; i=1,\,2\;.
\end{equation}
We refer to Figure \ref{fig3}. In particular, each interval $\ms A_i$
is contained in some valley, denoted by $\ms W_i = \ms V_{n(i),k(i)}$,
of some landscape $\Lambda'_i = \Lambda_{n(i)}$. Of course, the
valleys and the landscapes may coincide or not.

As the sets $\ms A_i$ are contained in valleys and the pre-factors
$G_a$, $a=0$, $1$, are constant in valleys, 
\begin{equation}
\label{51}
G_a(\theta^-_i) \;=\; G_a(\theta^+_i) \;,\quad a=0\,, 1\;, \;\;
i=1\,,\, 2\;.
\end{equation}
This identity will be used repeatedly below to replace $\theta^-_i$ by
$\theta^+_i$.

By Assertion \ref{as13}, $V$ and $S$ differ only by an additive
constant on the valley $\ms W_i$. In particular,
\begin{equation*}
S(\theta) \;=\; V(\theta) \;+\; C_i \;,\;\; \theta\in \ms A_i\;, 
\end{equation*}
and $V$ is differentiable in $\ms W_i$. It follows from
\eqref{44} that $V'(\theta^-_i) \le 0 \le V'(\theta^+_i)$. We assume
a strict inequality:
\begin{equation}
\label{45}
V'(\theta^-_i) \;<\; 0  \;<\;  V'(\theta^+_i)\;.
\end{equation}

Two points $\sigma$, resp. $\sigma^*$ in $(\theta^+_1, \theta^-_2)$,
resp. $(\theta^+_2, 1 + \theta^-_1)$, are called saddle points between
$\ms A_1$ and $\ms A_2$ if
\begin{align*}
& S(\sigma) \;=\; \max \big\{S(y) :  
\theta^+_1 \le y \le \theta^-_2  \big\}\;, \\
& \quad
S(\sigma^*) \;=\; \max \big\{S(y) : 
\theta^+_2 \le x \le 1+\theta^-_1 \big\}\;.
\end{align*}
Of course, there may be more than one, but let us fix two saddle
points between $\ms A_1$ and $\ms A_2$, $\sigma_{1,2}\in (\theta^+_1,
\theta^-_2)$, $\sigma_{2,1} \in (\theta^+_2, 1 + \theta^-_1)$.

Observe that $V(\sigma_{2,1}) = H$ if $V(\sigma_{1,2}) < H$.  Indeed,
if $V(\sigma_{1,2}) < H$, $\ms A_1$, $\ms A_2$ belong to the same
valley. This implies that $V(\sigma) = H$ for all saddle points in
$(\theta^+_2 , 1+\theta^-_1)$. \smallskip

In the computation of the capacity between $\ms A_1$, $\ms A_2$, three
cases emerge. The sets $\ms A_1$, $\ms A_2$ may belong to the same
valley, to different valleys but to the same landscape, or to
different landscapes. Consider first the case, illustrated in Figure
\ref{fig3}, in which both sets belong to the same valley.

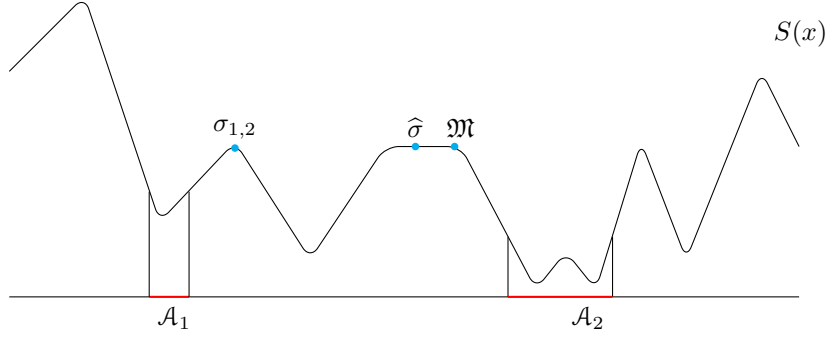
\begin{figure}
  \centering
\begin{tikzpicture}
\draw (-5,3) 
[rounded corners = 5pt] -- (-4,4) 
[rounded corners = 5pt] -- (-3,1) 
[rounded corners = 5pt] -- (-2,2.06) 
[rounded corners = 5pt] -- (-1,.5) 
[rounded corners = 5pt] -- (0,2) 
[rounded corners = 5pt] -- (1,2)
[rounded corners = 5pt] -- (2,.1)
[rounded corners = 5pt] -- (2.4,.6)
[rounded corners = 5pt] -- (2.8,.1)
[rounded corners = 5pt] -- (3.4,2.06)
[rounded corners = 5pt] -- (4,.5)
[rounded corners = 5pt] -- (5,3)
[rounded corners = 5pt] -- (5.5,2);
\draw (6,3.5) node[anchor=east] {$S(x)$};
\draw[red, thick] (-3.14,0) -- (-2.61, 0);
\draw (-3.14,0) -- (-3.14, 1.4);
\draw (-2.61,0) -- (-2.61, 1.4);
\draw (-2.8,0) node[anchor=north] {$\ms A_1$};
\draw (1.63,0) -- (1.63, .8);
\draw (3.02,.8) -- (3.02, 0);
\draw[red, thick] (1.63,0) -- (3.02, 0);
\draw (2.7,0) node[anchor=north] {$\ms A_2$};
\draw (-5,0) -- (-3.14, 0);
\draw (-2.61,0)--(1.63,0);
\draw (3.02,0) -- (5.5,0);
\draw (-2,1.98) node[anchor=south] {$\sigma_{1,2}$};
\draw (.4,2) node[anchor=south] {$\widehat\sigma$};
\draw (1,2) node[anchor=south] {$\mf M$};
\fill[cyan] (-2,1.98) circle [radius = .05cm];
\fill[cyan] (.4,2) circle [radius = .05cm];
\fill[cyan] (.92,2) circle [radius = .05cm];
\end{tikzpicture}
\caption{This figure represents two disjoint intervals $\ms A_1$,
  $\ms A_2$ of $\bb T$ which belong to a valley $\ms W$. In this case,
  the energy barrier between $\ms A_1$ and $\ms A_2$ is much smaller
  inside the valley (that is, in the interval
  $[\theta^+_1, \theta^-_2]$) than outside it. The calculation of the
  capacity is thus reduced to a computation in the latter interval.}
\label{fig3}
\end{figure}

Assume that the sets $\ms A_i$ are contained in a valley
$\ms W = (\mf w^-, \mf w^+)$ . If
$\mf w^- < \theta^-_1 < \theta^+_2 < \mf w^+$, let $E_{1,2}$
be the set of local maxima $\mf M^+_k$ of $S$ in
$(\theta^+_1, \theta^-_2)$ such that $S(\mf M^+_k) = S(\sigma_{1,2})$:
\begin{equation}
\label{48b}
E_{1,2} \;=\; \big\{ k\in \{1, \dots, q\}: 
\mf M^+_k \in [\theta^+_1,\theta^-_2] \,,\, 
S(\mf M^+_k) = S(\sigma_{1,2}) \big\}\;. 
\end{equation}
If $\mf w^- < \theta^-_2 < 1 + \theta^+_1 < \mf w^+$, $E_{1,2}$
represents the set of local maxima $\mf M^+_k$ of $S$ in $(\theta^+_2,
1+\theta^-_1)$ such that $S(\mf M^+_k) = S(\sigma_{2,1})$. In Figure
\ref{fig3}, $E_{1,2} = \{a,b\}$ if $\sigma_{1,2} = \mf M^+_a$, $\mf M
= \mf M^+_b$.

For $1\le i\le q$, let
\begin{equation}
\label{07c}
\omega(\mf M^+_i) \;=\;  \omega_+(\mf M^+_i) \,+\, \omega_-(\mf M^-_i) \;. 
\end{equation}

\begin{proposition}
\label{p04}
Let $\ms A_1 =[\theta^-_1,\theta^+_1]$,
$\ms A_2 =[\theta^-_2,\theta^+_2]$ be two intervals satisfying
conditions \eqref{50}, \eqref{44}, \eqref{45}.  Suppose that the sets
$\ms A_1$, $\ms A_2$ belong to a valley $\ms W = (\mf w^-, \mf w^+)$,
$1\le j\le n$.  Then,
\begin{align*}
\Cap_\epsilon (\ms A_1,\ms A_2) \;=\; \big[ 1 + o(\sqrt{\epsilon}) \big]\, 
\frac{\epsilon}{Z_\epsilon}\, 
\frac{G_0(\theta^+_1)  \,+\, \sqrt{\epsilon} \, G_1(\theta^+_1) + o(\sqrt{\epsilon})} 
{\sum_{k\in E_{1,2}} [\mf M^+_k - \mf M^-_k
\,+\, \sqrt{\epsilon}\, \omega(\mf M^+_k) ]+ o(\sqrt{\epsilon})}
\, e^{-V(\sigma_{1,2})/\epsilon}\;. 
\end{align*}
\end{proposition}

We may replace on the right hand side $\theta^+_1$ by $\theta^+_2$
because $G_0$, $G_1$ are constant in the valleys. \smallskip

We turn to the case in which the sets $\ms A_1$, $\ms A_2$ belong to
different landscapes, so that  $z(\theta^+_1) < z(\theta^+_2) <
z(1+\theta^+_1)$. Figure \ref{fig2} illustrates this situation.

\begin{proposition}
\label{p05}
Let $\ms A_1 =[\theta^-_1,\theta^+_1]$, $\ms A_2
=[\theta^-_2,\theta^+_2]$ be two intervals satisfying conditions
\eqref{50}, \eqref{44}, \eqref{45}. Suppose that they belong to
different landscapes. Then,
\begin{equation*}
\Cap_\epsilon (\ms A_1,\ms A_2) \; =\; 
\big[ 1 + o(1) \big]\, \frac{\epsilon}{Z_\epsilon}\, e^{-H/\epsilon} \;.
\end{equation*}
\end{proposition}

For $a=0$, $1$, let
\begin{equation*}
(\Delta_{1,2} G_a) \;=\; G_a (\theta^+_1) \,-\, G_a (\theta^+_2)\;, \quad
(\Delta_{2,1} G_a) \;=\; G_a (\theta^+_2) \,-\, G_a (\theta^+_1) \;.
\end{equation*}
Suppose that $\theta^+_1$, $\theta^+_2$ belong to the same landscape
but to different valleys. Then, either $z(\theta^+_2) = z
(\theta^+_1)$ or $z(\theta^+_2) = z (1+\theta^+_1)$.  If
$z(\theta^+_2) = z (\theta^+_1)$, $G_1 (\theta^+_2) < G_1
(\theta^+_1)$ so that $(\Delta_{1,2} G_1)>0$. While, if $z(\theta^+_2)
= z (1+\theta^+_1)$, $ G_1 (\theta^+_1) = G_1 (1+\theta^+_1) < G_1
(\theta^+_2)$ so that $(\Delta_{2,1} G_1)>0$. Similar conclusions hold
if we replace $G_1$ by $G_0$ and a strict inequality by an inequality.

\begin{proposition}
\label{p03}
Let $\ms A_1 =[\theta^-_1,\theta^+_1]$, $\ms A_2
=[\theta^-_2,\theta^+_2]$ be two intervals satisfying conditions
\eqref{50}, \eqref{44}, \eqref{45}. Suppose that they belong to
different valleys, but to the same landscape. Then, if
$z(\theta^+_1) = z(\theta^+_2)$,
\begin{align*}
\Cap_\epsilon (\ms A_1,\ms A_2) \; =\; 
\big[ 1 + o(\sqrt{\epsilon}) \big]\, \frac{\epsilon}{Z_\epsilon}\, 
\, e^{-H/\epsilon} \, 
\frac{G_0(\theta^+_1)  \,+\, \sqrt{\epsilon} \, G_1(\theta^+_1) 
\,+\, o(\sqrt{\epsilon})}
{(\Delta_{1,2} G_0) + \sqrt{\epsilon} \, (\Delta_{1,2} G_1) 
\,+\, o(\sqrt{\epsilon}) } \;\cdot
\end{align*}
If $z(\theta^+_2) = z(1+\theta^+_1)$,
$\Cap_\epsilon (\ms A_1,\ms A_2)$ is equal to
\begin{align*}
\big[ 1 + o(\sqrt{\epsilon})) \big]\, \frac{\epsilon}{Z_\epsilon}\, 
\, e^{-H/\epsilon} \, \Big( 1 \;+\; o(1) \;+\;
\frac{G_0(\theta^+_1)  \,+\, \sqrt{\epsilon} \, G_1(\theta^+_1)
\,+\, o(\sqrt{\epsilon})}
{(\Delta_{2,1} G_0) + \sqrt{\epsilon} \, (\Delta_{2,1} G_1)
\,+\, o(\sqrt{\epsilon})} \Big) \;.
\end{align*}
\end{proposition}

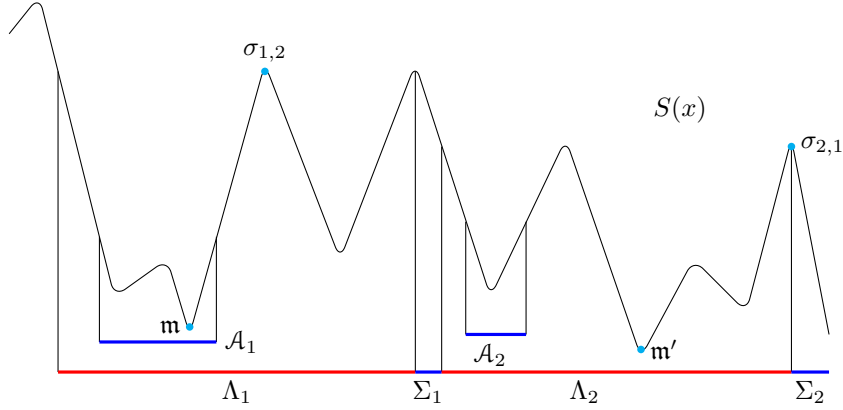
\begin{figure}
  \centering
\begin{tikzpicture}
\draw (-5.4,4.5) [rounded corners = 5pt] -- (-5,5) 
[rounded corners = 5pt] -- (-4, 1) 
[rounded corners = 5pt] -- (-3.3,1.5) 
[rounded corners = 5pt] -- (-3,.5) 
[rounded corners = 5pt] -- (-2,4.1) 
[rounded corners = 5pt] -- (-1,1.5) 
[rounded corners = 5pt] -- (0,4.1)
[rounded corners = 5pt] -- (1,1)
[rounded corners = 5pt] -- (2,3.1)
[rounded corners = 5pt] -- (3,.2)
[rounded corners = 5pt] -- (3.7,1.5)
[rounded corners = 5pt] -- (4.4,.8)
[rounded corners = 5pt] -- (5,3.1)
[rounded corners = 5pt] -- (5.5,.5);
\draw (4,3.5) node[anchor=east] {$S(x)$};
\draw (-4.75,0) -- (-4.75,4);
\draw (0,0) -- (0,4);
\draw[red, very thick] (-4.75,0) -- (0,0);
\draw (-2.37,0) node[anchor=north] {$\Lambda_1$};
\draw (-4.2, .4) -- (-4.2,1.8);
\draw (-2.645, .4) -- (-2.645,1.8);
\draw[blue, very thick] (-4.2, .4) -- (-2.645,.4);
\draw (-2.645,.4) node[anchor=west] {$\ms A_1$};
\draw (.35, 0) -- (.35,3);
\draw (5, 0) -- (5,3);
\fill[cyan] (-2,4) circle [radius = .05cm];
\draw (-2,4) node[anchor=south] {$\sigma_{1,2}$};
\draw (5,3) node[anchor=west] {$\sigma_{2,1}$};
\fill[cyan] (5,3) circle [radius = .05cm];
\fill[cyan] (-3,.6) circle [radius = .05cm];
\draw (-3,.6) node[anchor=east] {$\mf m$};
\fill[cyan] (3,.3) circle [radius = .05cm];
\draw (3,.3) node[anchor=west] {$\mf m'$};
\draw[blue, very thick] (0,0) -- (.35,0);
\draw (.17,0) node[anchor=north] {$\Sigma_1$};
\draw[red, very thick] (.35,0) -- (5,0);
\draw (2.25,0) node[anchor=north] {$\Lambda_2$};
\draw[blue, very thick] (5,0) -- (5.5,0);
\draw (5.25,0) node[anchor=north] {$\Sigma_2$};
\draw (.67,0.5) -- (.67,2);
\draw (1.472,.5) -- (1.472,2);
\draw[blue, very thick] (.67, .5) -- (1.472,.5);
\draw (1,.5) node[anchor=north] {$\ms A_2$};
\end{tikzpicture}
\caption{This figure illustrates the case in which the intervals
  $\ms A_i$ belong to different landscapes. Starting from $\ms A_1$
  the process reaches $\ms A_2$ by surmounting the energetic barrier
  $[S(\sigma_{1,2}) - S(\mf m)]/\epsilon$, while it reaches $\ms A_1$
  when starting from $\ms A_2$, by surpassing the energetic barrier
  $[S(\sigma_{2,1}) - S(\mf m')]/\epsilon$ because at the large
  deviations level, it never visits a landscape $\Lambda_n$ once in
  $\Lambda_{n+1}$. The capacity represents the height of the saddle
  point which in this case is equal to $H$ in both cases since
  $V(\sigma_{1,2}) = V(\sigma_{2,1}) = H$.}
\label{fig2}
\end{figure}

The proofs of the previous results rely on the next claim.

\begin{asser}
\label{as11}
Let $\ms A_1 =[\theta^-_1,\theta^+_1]$,
$\ms A_2 =[\theta^-_2,\theta^+_2]$ be two intervals satisfying
conditions \eqref{50}, \eqref{44}, \eqref{45}. Then,
\begin{equation*}
\begin{aligned}
\Cap_\epsilon (\ms A_1,\ms A_2) \; & =\; 
\big[1 + o(\sqrt{\epsilon}) \big]\, \frac{\epsilon}{Z_\epsilon}\,
\Big\{G_0(\theta^+_1)  \,+\, \sqrt{\epsilon} \, G_1(\theta^+_1) 
\,+\, o(\sqrt{\epsilon}) \Big\} 
\, e^{-V(\theta^+_1)/\epsilon} \;\times \\
&\qquad \times \; \Big\{ \frac{1}
{\int_{\theta^+_2}^{1+\theta^-_1} e^{[S(y)-S(1+\theta^-_1)]/\epsilon} \, dy}\, 
\,+\, \frac{1}
{\int_{\theta^+_1}^{\theta^-_2} e^{[S(y) -S(\theta^+_1)]/\epsilon} \, dy} \Big\}\;.
\end{aligned} 
\end{equation*}
\end{asser}

\begin{proof}
Fix two intervals $\ms A_1 =[\theta^-_1,\theta^+_1]$, $\ms A_2
=[\theta^-_2,\theta^+_2]$ satisfying the hypotheses of the assertion.

In view of equations \eqref{11}, \eqref{30}, 
\begin{equation*}
\Cap_\epsilon (\ms A_1,\ms A_2) \;=\; \epsilon\, \Big\{ 
\frac{e^{S(1+\theta^-_1)/\epsilon}}
{\int_{\theta^+_2}^{1+\theta^-_1} e^{S(y)/\epsilon} \, dy}
\, m_\epsilon(\theta^-_1)
\,+\, \frac{e^{S(\theta^+_1)/\epsilon}}
{\int_{\theta^+_1}^{\theta^-_2} e^{S(y)/\epsilon} \, dy}
\, m_\epsilon(\theta^+_1) \Big\}\;.
\end{equation*}
Note that in the first ratio on the right hand side we have
$1+\theta^-_1$ instead of $\theta^-_1$ because we are now working on
the line so that the harmonic function is defined on the interval
$[\theta^+_2 , 1+\theta^-_1]$. As $m_\epsilon$ is periodic,
$m_\epsilon(1+\theta^-_1) = m_\epsilon(\theta^-_1)$.

Propositions \ref{l01}, \ref{l02} provide a formula for $m_\epsilon =
\pi_\epsilon/c(\epsilon)$. By \eqref{51}, we may replace
$G_a(\theta^-_1)$ by $G_a(\theta^+_1)$, $a=1$, $2$. To complete the
proof, it remains to recall that $V(\theta^-_1) = V(\theta^+_1)$ by
hypothesis.
\end{proof}

\begin{proof}[Proof of Proposition \ref{p04}]
Fix two intervals $\ms A_1 =[\theta^-_1,\theta^+_1]$,
$\ms A_2 =[\theta^-_2,\theta^+_2]$ satisfying the hypotheses of the
proposition. 

We estimate the second term inside braces in the formula for the
capacity appearing in Assertion \ref{as11}.  Since $\ms A_1$ and $\ms
A_2$ belong to the same valley, by Assertion \ref{as13}, in the
interval $[\theta^+_1,\theta^-_2]$, the functions $V$ and $S$ differ
only by an additive constant. Recall the definition \eqref{48b} of the
set $E_{1,2}$. By the computation performed in Assertion \ref{as09a},
\begin{align*}
& \int_{\theta^+_1}^{\theta^-_2} e^{[S(y) -S(\theta^+_1)]/\epsilon} \,
dy \\
&\quad \;=\;
\Big( \sum_{k\in E_{1,2}} [\mf M^+_k - \mf M^-_k] 
\,+\, \sqrt{\epsilon}\,  \sum_{k \in E_{1,2}} \omega(\mf M^+_k)
\,+\, o(\sqrt{\epsilon}) \Big)
\, e^{[S(\sigma_{1,2}) -S(\theta^+_1)]/\epsilon}\;.
\end{align*}
Since $V$ and $S$ differ by a constant, we may replace in the previous
formula $S(\sigma_{1,2}) -S(\theta^+_1)$ by
$V(\sigma_{1,2}) - V(\theta^+_1)$.

We turn to the first term inside braces. As $\ms A_1$, $\ms A_2$
belong to the same valley, $z(\theta^+_2) = z(\theta^-_1)$, so that 
$z(\theta^+_2) <  1 + \theta^-_1$ because $z(\theta^-_1) <  1 +
\theta^-_1$. Therefore, since $S(1+\theta^-_1) = - B + S(\theta^-_1)$, 
\begin{equation*}
\int_{\theta^+_2}^{1+\theta^-_1} e^{[S(y)-S(1+\theta^-_1)]/\epsilon} \, dy
\;=\; e^{[S(z(\theta^+_2)) - S(\theta^-_1) + B]/\epsilon}
\int_{\theta^+_2}^{1+\theta^-_1} e^{[S(y)-S(z(\theta^+_2))]/\epsilon} \, dy
\;.
\end{equation*}
As $z(\theta^+_2) = z(\theta^-_1)$,
$S(z(\theta^+_2)) - S(\theta^-_1) = -\widehat V(\theta^-_1)$. On the
other hand, since $z(\theta^+_2) < 1 + \theta^-_1$, by the computation
performed in Assertion \ref{as09a}
\begin{equation*}
\int_{\theta^+_2}^{1+\theta^-_1} e^{[S(y)-S(z(\theta^+_2))]/\epsilon}
\, dy \;\ge\; C_0 \, \sqrt{\epsilon}  
\end{equation*}
for some positive constant $C_0$ independent of $\epsilon$. Thus, the
right hand side of the penultimate displayed equation is bounded below
by $C_0 \, \sqrt{\epsilon} \, \exp\{-[\widehat
V(\theta^-_1)-B]/\epsilon\}$. 

To complete the proof of the proposition, it remains to show that
$\widehat V(\sigma_{1,2}) < B$, but this is clear because $\widehat
V(\sigma_{1,2}) \le 0 < B$.
\end{proof}

\begin{proof}[Proof of Proposition \ref{p05}]
In the formula for the capacity of Assertion \ref{as11}, write the
second term in the expression inside braces as
\begin{equation*}
\int_{\theta^+_1}^{\theta^-_2} e^{[S(y) -S(\theta^+_1)]/\epsilon} \,
dy \;=\; e^{[S(z(\theta^+_1)) - S(\theta^+_1)]/\epsilon}
\int_{\theta^+_1}^{\theta^-_2} e^{[S(y) -S(z(\theta^+_1))]/\epsilon}
\, dy\;.
\end{equation*}
Since $z(\theta^+_1) < \theta^-_2$, the integral on the right hand
side is equal to
$G_0(\theta^+_1) + \sqrt{\epsilon} \, G_1(\theta^+_1) +
o(\sqrt{\epsilon})$,
while the term appearing in the exponential is equal to
$-\, \widehat V(\theta^+_1)$.

We turn to the first term in the expression inside braces in Assertion
\ref{as11}. It can be written as 
\begin{equation*}
e^{[S(z(\theta^+_2))-S(1+\theta^-_1)]/\epsilon}\, 
\int_{\theta^+_2}^{1+\theta^-_1} e^{[S(y)-S(z(\theta^+_2))]/\epsilon} \, dy\;.
\end{equation*}
Since $z(\theta^+_2) < 1+\theta^-_1$, the integral on the right hand
side is equal to
$G_0(\theta^+_2) + \sqrt{\epsilon} \, G_1(\theta^+_2) +
o(\sqrt{\epsilon})$.
As $S(1+\theta^-_1) = S(\theta^-_1)-B$ and
$-\, \widehat V(\theta^+_1) = -\, \widehat V(\theta^-_1) =
S(z(\theta^-_1)) - S(\theta^-_1)$,
to complete the proof of the proposition, it remains to show that
$S(z(\theta^-_1)) < S(z(\theta^+_2)) + B$. This is clear because
$S(z(\theta^-_1)) - B = S(1+z(\theta^-_1)) = S(z(1+\theta^-_1)) <
S(z(\theta^+_2))$.
The last inequality follows from the fact that
$z(\theta^+_2) < 1+\theta^-_1$.
\end{proof}

\begin{proof}[Proof of Proposition \ref{p03}]
Assume that $z(\theta^+_1) = z(\theta^+_2)$.  We estimate the
integrals appearing in Assertion \ref{as11}. Clearly,
\begin{equation*}
\int_{\theta^+_1}^{\theta^-_2} e^{[S(y) -S(\theta^+_1)]/\epsilon} \,
dy \;=\; \int_{\theta^+_1}^{1+\theta^+_1} e^{[S(y) -S(\theta^+_1)]/\epsilon} \,
dy \;- \; \int_{\theta^-_2}^{1+\theta^+_1} e^{[S(y) -S(\theta^+_1)]/\epsilon} \,
dy \;.
\end{equation*}
Since $z(\theta^-_2) = z(\theta^+_1) < 1+\theta^+_1$, in the second
integral we may replace $1+\theta^+_1$ by $1+\theta^-_2$. By
Proposition \ref{l01}, the right hand side is equal to
\begin{align*}
& \big \{G_0(\theta^+_1) 
\,+\, \sqrt{\epsilon} \, G_1(\theta^+_1) 
\,+\, o(\sqrt{\epsilon}) \big\}  
\, e^{[S(z(\theta^+_1)) - S(\theta^+_1)]/\epsilon} \\
&\quad \;- \; 
\big \{G_0(\theta^-_2) \,+\, \sqrt{\epsilon} \, G_1(\theta^-_2) 
\,+\, o(\sqrt{\epsilon}) \big\} 
\, e^{[S(z(\theta^-_2)) - S(\theta^+_1)]/\epsilon}\;.
\end{align*}
As $z(\theta^-_2) = z(\theta^+_1)$ and $G_a(\theta^-_2) =
G_a(\theta^+_2)$, in view of the definition of $\Delta_{1,2}G_a$,
this difference is equal to
\begin{equation*}
\big \{ (\Delta_{1,2}G_0) \,+\, \sqrt{\epsilon} \, (\Delta_{1,2}G_1) 
\,+\, o(\sqrt{\epsilon})\,  \big\}  \, e^{- \widehat V(\theta^+_1)/\epsilon} \;.
\end{equation*}

On the other hand, since $z(\theta^+_2) = z(\theta^-_1) <
1+\theta^-_1$, there exists a constant $C_0>0$, independent of
$\epsilon$, such that
\begin{equation*}
\int_{\theta^+_2}^{1+\theta^-_1} e^{[S(y) -S(1+\theta^-_1)]/\epsilon} \,
dy \;\ge\; C_0\, \sqrt{\epsilon} \, 
e^{[S(z(\theta^+_2)) - S(1+\theta^-_1)]/\epsilon}\;.
\end{equation*}
As $z(\theta^+_2) = z(\theta^-_1)$ and $S(1+\theta^-_1) =
S(\theta^-_1) - B$, the expression in the exponential in the previous
formula is equal to $- \,\widehat V(\theta^-_1) + B = - \, \widehat
V(\theta^+_1) + B$. This proves the first assertion of the
proposition.

We turn to the case $z(\theta^+_2) = z(1+\theta^+_1)$. We estimate the
integrals appearing in Assertion \ref{as11}. Since $z(\theta^+_1) <
z(1+\theta^+_1)$, we have that $z(\theta^+_1) < z(\theta^+_2)$. This
implies that $z(\theta^+_1) < \theta^-_2$, so that
\begin{equation*}
\int_{\theta^+_1}^{\theta^-_2} e^{[S(y) -S(\theta^+_1)]/\epsilon} \,
dy \;=\; 
\big \{G_0(\theta^+_1) \,+\, G_1(\theta^+_1) \, \sqrt{\epsilon} 
\,+\, o(\sqrt{\epsilon})\, \big\}  
\, e^{- \widehat V(\theta^+_1)/\epsilon} \;.
\end{equation*}
On the other hand, 
\begin{align*}
& \int_{\theta^+_2}^{1+\theta^-_1} e^{[S(y)
  -S(1+\theta^+_1)]/\epsilon} \, dy \\
&\qquad \;=\; \int_{\theta^+_2}^{1+\theta^+_2} e^{[S(y) -S(1+\theta^+_1)]/\epsilon} \,
dy \;-\; \int_{1+\theta^-_1}^{1+\theta^+_2} e^{[S(y) -S(1+\theta^+_1)]/\epsilon} \,
dy \;.
\end{align*}
Since $z(1+\theta^-_1) = z(\theta^+_2) < 1+\theta^+_2$ and since
$G_a(1+\theta^-_1) = G_a(\theta^-_1)$, $a=1$, $2$, this
expression is equal to
\begin{align*}
& \big \{G_0(\theta^+_2) 
\,+\, \sqrt{\epsilon} \, G_1(\theta^+_2) \,+\, o(\sqrt{\epsilon})\big\}  
\, e^{[S(z(\theta^+_2)) - S(1+\theta^+_1)]/\epsilon} \\
&\quad \;- \; \big \{G_0(\theta^-_1) 
\,+\, \sqrt{\epsilon} \, G_1(\theta^-_1) 
\,+\, o(\sqrt{\epsilon}) \big\} 
\,  e^{[S(z(1+\theta^-_1)) - S(1+\theta^+_1)]/\epsilon}\;.
\end{align*}
As $z(\theta^+_2) = z(1+\theta^-_1) = z(1+\theta^+_1)$, and $\widehat
V(1+\theta^+_1) = \widehat V(\theta^+_1)$, the expressions in the
exponential above are equal to $-\widehat V(\theta^+_1)$, which
completes the proof of the proposition because $G_a(\theta^-_1) =
G_a(\theta^+_1)$, $a=1$, $2$.
\end{proof}

We conclude this section providing an alternative formula for the
capacity.

\begin{asser}
\label{as03}
Fix two disjoint closed intervals $\ms A_1=[\theta^-_1,\theta^+_1]$,
$\ms A_2=[\theta^-_2,\theta^+_2]$ of $\bb T$. Then,
\begin{equation*}
\Cap_\epsilon (\ms A_1,\ms A_2) \;=\; \epsilon\, 
\big\{ h_{\ms A_1,\ms A_2}'(\theta^-_1) \, m_\epsilon(\theta^-_1)
\,-\, h_{\ms A_1,\ms A_2}'(\theta^+_1) \, m_\epsilon(\theta^+_1)\big\}\;.
\end{equation*}
\end{asser}

\begin{proof}
By the expression \eqref{14} for the Dirichlet form, and since the harmonic
function is constant on the sets $\ms A_1$, $\ms A_2$, $D_\epsilon(h_{\ms A_1,\ms A_2})$ is
equal to the sum of two integrals. The first one is carried over the
interval $[\theta^+_1, \theta^-_2]$, while the second one over the interval
$[\theta^+_2, \theta^-_1]$. We estimate the first integral.  By an
integration by parts, 
\begin{align*}
\epsilon \int_{\theta^+_1}^{\theta^-_2} [h_{\ms A_1,\ms A_2}'(x)]^2 \, m_\epsilon(x)\,
dx \; & =\; \epsilon\, h_{\ms A_1,\ms A_2}(x)\,  h_{\ms A_1,\ms A_2}'(x) \, m_\epsilon(x)\,
\Big |_{\theta^+_1}^{\theta^-_2} \\
\; & -\; \epsilon \int_{\theta^+_1}^{\theta^-_2} h_{\ms A_1,\ms A_2}(x) \, \partial_x \big\{
h_{\ms A_1,\ms A_2}'(x) \, m_\epsilon(x) \big\}\, dx \;.
\end{align*}
Since the harmonic function vanishes at $\theta^-_2$ and is equal to
$1$ at $\theta^+_1$, by \eqref{12} and since $L_\epsilon h_{\ms A_1,\ms
  A_2} = 0$ on $(\theta^+_1,\theta^-_2)$, the previous expression is
equal to
\begin{equation*}
-\, \epsilon\,  h_{\ms A_1,\ms A_2}'(\theta^+_1) \, m_\epsilon(\theta^+_1)\;
+\; \epsilon \, R_\epsilon \, \int_{\theta^+_1}^{\theta^-_2} h_{\ms A_1,\ms A_2}(x) \,
 h_{\ms A_1,\ms A_2}'(x) \, dx \;.
\end{equation*}
The integral is equal to $(1/2) \epsilon \, R_\epsilon \,
\{h_{\ms A_1,\ms A_2}(\theta^-_2)^2 
- h_{\ms A_1,\ms A_2}(\theta^+_1)^2\} = -\, (1/2) \epsilon \,
R_\epsilon$. 

For similar reasons, the contribution to $D_\epsilon(h_{\ms A_1,\ms A_2})$
of the integral carried over the interval $[\theta^+_2, \theta^-_1]$
is equal to $\epsilon\, h_{\ms A_1,\ms A_2}'(\theta^-_1) \,
m_\epsilon(\theta^-_1) \,+\, (1/2) \epsilon \, R_\epsilon$. This
completes the proof of the assertion.
\end{proof}

\section{Metastability among the deepest valleys}
\label{sec3}

We examine in this section the metastable behavior of $X_\epsilon(t)$
among the deepest valleys. The goal is to define a finite-state,
continuous-time Markov chain, called the reduced chain, which
describes the evolution of the diffusion $X_\epsilon(t)$ among the
deepest wells in an appropriate time scale.

A similar analysis could be carried out for shallower valleys. This
task is left to the interested reader.  We assume throughout this
section that the drift $b$ satisfies the conditions ({\bf H4}) of
Subsection \ref{sec0}.4. 

\smallskip\noindent{\bf \ref{sec3}.1 The deepest valleys.}  Denote by
$\ms W_j = (\mf w^-_j , \mf w^+_j)$, $1\le j\le \bs n$, all valleys of
depth $H$. These are the valleys $\ms V_i$ introduced in \eqref{54}
such that $\min_{\theta\in \ms V_i} V(\theta)=0$.  Denote by
$\mf m_{j,k}\in \ms W_j$, $1\le k\le \kappa(j)$, the global minima of
$V$ on $\ms W_j$. Let $\ms E_j = (e^-_j, e^+_j)$ be a subset of
$\ms W_j$ which contains all minima $\mf m_{j,k}$ and such that
$V(e^-_j)= V(e^+_j)<H$, $V(\theta)<V(e^-_j)$ for all
$\theta\in \ms E_j$. The sets $\ms E_j$ are called wells. We refer to
Figure \ref{fig4} for an illustration. Assume, without loss of
generality, that the valleys are ordered in the sense that
$0<\mf m_{1,1} < \cdots < \mf m_{\bs n, 1}<1$.  Denote by $\bs \pi(j)$
the weight of the well $\ms E_j$:
\begin{equation}
\label{58}
\bs \pi(j) \;=\; \sum_{k=1}^{\kappa (j)} \sigma(\mf m_{j,k})
\;=\; \sum_{k=1}^{\kappa (j)} \sqrt{\frac{2\pi}{S''(\mf m_{j,k})}}
\;, \;\; 1\le j\le \bs n \;,
\end{equation}
where $\sigma(\mf m_{j,k})$ has been introduced in \eqref{39}.

Let $\mf M_{j,k}$, $1\le k\le \upsilon(j)$, be the global maxima of
$V$ which belong to the interval $(e^+_j, e^-_{j+1})$ and to the
landscape which contains $\ms W_j$. Hence if the valley $\ms W_j$ is
contained in the landscape $[\ell_n, \mf L^+_{n+1}]$, $\{ \mf M_{j,k}:
1\le k\le \upsilon(j)\} = \{ \mf M^+_{l} : \mf M^+_{l} \in (e^+_j,
e^-_{j+1}) \cap [\ell_n, \mf L^+_n] \,,\, V(\mf M^+_{l}) =H\}$. We
refer to Figure \ref{fig4} for an illustration. Denote by $\bs
\sigma_{j,j+1}$ the sum of the weights of these local maxima:
\begin{equation}
\label{59}
\bs \sigma_{j,j+1} \;=\; \sum_{k=1}^{\upsilon (j)} \omega(\mf M_{j,k})
\;=\; \sum_{k=1}^{\upsilon (j)} \sqrt{\frac{2\pi}{-\, S''(\mf M_{j,k})}}
\;, \;\; 1\le j\le \bs n \;.
\end{equation}

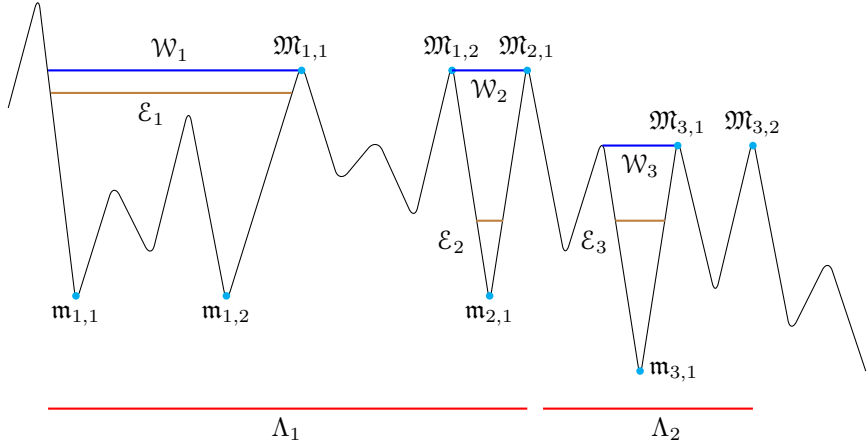
\begin{figure}
  \centering
\begin{tikzpicture}
\draw (-5.4,4.5) [rounded corners = 5pt] -- (-5,6) 
[rounded corners = 5pt] -- (-4.5, 1.9) 
[rounded corners = 5pt] -- (-4,3.5) 
[rounded corners = 5pt] -- (-3.5,2.5) 
[rounded corners = 5pt] -- (-3,4.5) 
[rounded corners = 5pt] -- (-2.5,1.9) 
[rounded corners = 5pt] -- (-1.5,5.1) 
[rounded corners = 5pt] -- (-1,3.5) 
[rounded corners = 5pt] -- (-.5,4.1)
[rounded corners = 5pt] -- (0,3)
[rounded corners = 5pt] -- (.5,5.1)
[rounded corners = 5pt] -- (1,1.9)
[rounded corners = 5pt] -- (1.5,5.1)
[rounded corners = 5pt] -- (2,2.5)
[rounded corners = 5pt] -- (2.5,4.1)
[rounded corners = 5pt] -- (3,.9)
[rounded corners = 5pt] -- (3.5,4.1)
[rounded corners = 5pt] -- (4,2)
[rounded corners = 5pt] -- (4.5,4.1)
[rounded corners = 5pt] -- (5,1.5)
[rounded corners = 5pt] -- (5.5,2.5)
[rounded corners = 5pt] -- (6,1);
\draw[thick, blue] (-4.87,5)--(-1.5,5);
\draw (-3.25,5) node[anchor=south] {$\ms W_1$};
\draw[brown, thick] (-4.83,4.7)--(-1.63,4.7);
\draw (-3.5,4.7) node[anchor=north] {$\ms E_1$};
\fill[cyan] (-4.5,2) circle [radius = .05cm];
\draw (-4.5,2) node[anchor=north] {$\mf m_{1,1}$};
\fill[cyan] (-2.5,2) circle [radius = .05cm];
\draw (-2.5,2) node[anchor=north] {$\mf m_{1,2}$};
\fill[cyan] (-1.5,5) circle [radius = .05cm];
\draw (-1.5,5) node[anchor=south] {$\mf M_{1,1}$};
\fill[cyan] (.5,5) circle [radius = .05cm];
\draw (.5,5) node[anchor=south] {$\mf M_{1,2}$};
\draw[thick, blue] (.5,5)--(1.5,5);
\draw (1,5) node[anchor=north] {$\ms W_2$};
\draw[brown, thick] (.82,3)--(1.18,3);
\draw (.8,3) node[anchor=north east] {$\ms E_2$};
\fill[cyan] (1,2) circle [radius = .05cm];
\draw (1,2) node[anchor=north] {$\mf m_{2,1}$};
\fill[cyan] (1.5,5) circle [radius = .05cm];
\draw (1.5,5) node[anchor=south] {$\mf M_{2,1}$};
\draw[red, thick] (-4.87,0.5)--(1.5,0.5);
\draw (-1.7,0.5) node[anchor=north] {$\Lambda_1$};
\draw[thick, blue] (2.5,4)--(3.5,4);
\draw (3,4) node[anchor=north] {$\ms W_3$};
\draw[brown, thick] (2.67,3)--(3.33,3);
\draw (2.7,3) node[anchor=north east] {$\ms E_3$};
\fill[cyan] (3,1) circle [radius = .05cm];
\draw (3,1) node[anchor=west] {$\mf m_{3,1}$};
\fill[cyan] (3.5,4) circle [radius = .05cm];
\draw (3.5,4) node[anchor=south] {$\mf M_{3,1}$};
\fill[cyan] (4.5,4) circle [radius = .05cm];
\draw (4.5,4) node[anchor=south] {$\mf M_{3,2}$};
\draw[red, thick] (1.71,0.5)--(4.5,0.5);
\draw (3.35,0.5) node[anchor=north] {$\Lambda_2$};
\end{tikzpicture}
\caption{This figure represents the wells and valleys in two
  landscapes. There are $3$ valleys whose depth is maximal, $\ms W_1$,
  $\ms W_2$ and $\ms W_3$. The first one contains two global minima of
  the quasi-potential $V$, while the other two only one. If the
  process starts from the well $\ms W_2$ the next well visited may be
  either $\ms W_1$ or $\ms W_3$, while if it starts from $\ms W_1$ the
  next well visited can only be $\ms W_2$.}
\label{fig4}
\end{figure}

\smallskip\noindent{\bf \ref{sec3}.2 The evolution among the wells $\ms E_j$.}
The asymptotic behavior of the diffusion $X_\epsilon(t)$ among the
wells $\ms E_l$ can be foretold. Rename the valleys $\ms W_1, \dots,
\ms W_{\bs n}$ as $\ms W_{a,k}$, $1\le a\le \bs p$, $1\le k\le n_a$,
in such a way that two valleys $\ms W_{a,k}$, $\ms W_{a',k'}$ belong
to the same landscape if and only if $a=a'$. Denote the minimum $\mf
m_{l,r}$ by $\mf m_{a,k,r}$ if $\ms W_l = \ms W_{a,k}$ and assume that
the valleys are ordered in the sense that $\mf m_{a,k,r} < \mf
m_{a',k',s}$ if $a<a'$ or if $a=a'$ and $k<k'$.

Assume that $X_\epsilon(0)$ belongs to $\ms W_i$. The next visited
valley can only be $\ms W_{i-1}$ or $\ms W_{i+1}$, where we adopt the
convention that $\ms W_{r\bs n +k} = r + \ms W_k$. However, if $\ms
W_{i} = \ms W_{a,1}$ for some $a$, since, by Remark \ref{rm2}, the
diffusion does not visit a landscape to its left, modulo a probability
exponentially close to $1$, the next visited valley is necessarily
$\ms W_{i+1}$. Hence, if $p(i,j)$ represents the jump probabilities of
the reduced chain, we must have that
\begin{equation}
\label{i02}
p(i,i+1) \;+\; p(i,i-1) \;=\; 1\quad\text{and}\quad
p(i,i+1) \;=\; 1 \;\;\text{if}\; \ms
W_{i} = \ms W_{a,1}\;\text{for some }a\;.
\end{equation}

We may compute the jump probabilities using formula \eqref{11} for the
equilibrium potential. Assume that $\ms W_i=\ms W_{a,k}$ for some
$k\ge 2$, and let
\begin{equation*}
p_\epsilon (i,i+1) \;=\; h_{\ms W_{i+1}, \ms W_{i-1}}(\mf m_{i,1})\;,
\end{equation*}
where $h_{\ms W_{i+1}, \ms W_{i-1}}$ is the equilibrium potential
introduced in \eqref{11}.  An elementary computation gives that
\begin{equation*}
p_\epsilon(i,i+1) \;=\; \big[1+o(1)\big]\, \frac{\bs \sigma_{i-1,i}}
{\bs \sigma_{i-1,i} + \bs \sigma_{i,i+1}}\;.
\end{equation*}
Therefore, we have to set
\begin{equation}
\label{55}
p(i,i+1) \;=\; \frac{\bs \sigma_{i-1,i}}
{\bs \sigma_{i-1,i} + \bs \sigma_{i,i+1}}\;,
\quad \text{if $\ms W_i \,=\, \ms W_{a,k}$ for some $k\ge 2$} \;.
\end{equation}
Equations \eqref{i02} and \eqref{55} characterize the jump
probabilities of the reduced chain.  We turn to the holding rates of
the reduced chain. By Proposition \ref{l02} and Lemma \ref{l09},
\begin{equation}
\label{57}
\mu_\epsilon (\ms E_i) \;=\;  [1+o(1)]\,
\bs \mu (i) \quad\text{where}\quad
\bs \mu(i) \;=\; \frac 1Z\, G_1(\mf m_{i,1})\, \bs \pi(i)\;.
\end{equation}
On the other hand, by similar computations to the ones presented in
the previous section and by Proposition \ref{p02}, $\Cap_\epsilon(\ms W_i, \ms
W_{i-1} \cup \ms W_{i+1}) = [1+o(1)]\, e^{-H/\epsilon} \, c(i)$,
where, 
\begin{align*}
& c(i) \;=\; \frac{G_1(\mf m_{i,1})}{Z} \, \frac 1{\bs
  \sigma_{i,i+1}}\quad\text{if $\; p(i,i+1)=1$}\;, \\
&\quad c(i) \;=\; \frac {G_1(\mf m_{i,1})}Z\, \Big( 
\frac 1{\bs \sigma_{i-1,i}} \,+\, \frac 1{ \bs \sigma_{i,i+1}} \Big) 
\quad\text{if $\; p(i,i+1)<1$}\;.
\end{align*}

It follows from the previous estimates and from equation (A.8) in
\cite{bl7}, that on the time-scale $e^{H/\epsilon}$, the diffusion
$X_\epsilon (t)$ is expected to evolve among the valleys $\ms W_j$ as
the $\{1, \dots, \bs n\}$-valued, continuous-time Markov chain with
jump probabilities given by \eqref{55} and holding times $\lambda(i)$
given by $\lambda(i) = c(i)/\bs \mu(i)$.

\smallskip\noindent{\bf \ref{sec3}.3 The reduced chain.} At this point
we have all elements to define the Markov chain which describes the
metastable behavior of $X(t)$. Denote by $R(j,k)$ the jump rates of
the Markov chain whose holding rates are $\lambda$ and whose jump
probabilities are $p$: $R(j,k) = \lambda(j) p(j,k)$, $j\not = k\in
S$. By the previous computations,
\begin{equation}
\label{70}
R(j,j+1) \;=\; \frac 1{\bs \pi(j)\, \bs \sigma_{j,j+1}}
\;,  \quad\text{$R(j+1,j) = 0$ or}\;\;
\frac 1{\bs \pi(j+1)\, \bs \sigma_{j,j+1}}\;,
\end{equation}
and $R(j,k)=0$ if $k\not = j\pm 1$.  More precisely, $R(j+1,j)=0$ if
$\ms W_{j+1} = \ms W_{a,1}$ for some $a$, and $R(j+1,j) = [\bs
\pi(j+1)\, \bs \sigma_{j,j+1}]^{-1}$ otherwise.

In view of the previous computation, denote by $\bs X(t)$ the
continuous-time Markov chain on $S=\{1, \dots, \bs n\}$ whose
generator $\bs L$ is given by
\begin{equation}
\label{66}
(\bs L f)(j) \;=\; \sum_{a=\pm 1} R(j,j+a)\, [F(j+a) -
F(j)]\;.
\end{equation}
Summation is performed modulo $\bs n$ in the previous formula. The
next result is proved at the end of this section.

\begin{lemma}
\label{ls01}
The measure $\bs \mu$, introduced in \eqref{57}, is the stationary
state of the Markov chain $\bs X(t)$.
\end{lemma}

Denote by $D(\bb R_+, S)$ the space of right-continuous functions $x:
\bb R_+ \to S$ with left-limits endowed with the Skorohod topology,
and by $\bs Q_j$, $1\le j\le \bs n$, the probability measure on $D(\bb
R_+,S)$ induced by the Markov process whose generator is $\bs L$ and
which starts from $j$.

\smallskip\noindent{\bf \ref{sec3}.4 The metastable behavior.}
Denote by $\widehat X_\epsilon(t)$ the process $X_\epsilon(t)$
speeded-up by $e^{H/\epsilon}$. This is the diffusion on $\bb T$ whose
generator, denoted by $\widehat L_\epsilon$, is given by
$\widehat L_\epsilon = e^{H/\epsilon} L_\epsilon$.  Let
$C(\bb R_+, \bb T)$ be the space of continuous trajectories
$X: \bb R_+ \to\bb T$ endowed with the topology of uniform convergence
on compact subsets of $\bb R_+$.  Denote by $\bb P^\epsilon_\theta$,
$\theta\in \bb T$, the probability measure on $C(\bb R_+, \bb T)$
induced by the diffusion $\widehat X_\epsilon(t)$ starting from
$\theta$. Expectation with respect to $\bb P^\epsilon_\theta$ is
represented by $\bb E^\epsilon_\theta$.

Let 
\begin{equation}
\label{65}
\ms E \;=\; \bigcup_{j=1}^{\bs n} \ms E_j \;, \quad \Delta \;=\; \bb T
\,\setminus\, \ms E\;, \quad \breve{\ms E}_j \;=\; \bigcup_{k:k\not =
  j} \ms E_k\;. 
\end{equation}
Denote by $T_{\ms E} (t)$, $t\ge 0$, the total time spent by the
diffusion $\widehat X_\epsilon(t)$ on the set $\ms E$ in the time
interval $[0,t]$:
\begin{equation*}
T_{\ms E} (t)  \,:=\, \int_0^t \chi_{\ms E} (\widehat X_\epsilon(s)) \,ds\;.
\end{equation*}
Denote by $\{S_{\ms E} (t) : t\ge 0\}$ the generalized inverse of
$T_{\ms E} (t)$:
\begin{equation*}
S_{\ms E} (t) \,:=\, \sup\{s\ge 0 : T_{\ms E} (s) \le t \}\,.
\end{equation*}
Clearly, for all $r\ge 0$, $t\ge 0$,
\begin{equation}
\label{60}
\{ S_{\ms E} (r) \ge t\} \;=\; \{ T_{\ms E} (t) \le r\}\;.
\end{equation}
It is also clear that for any starting point $\theta\in \bb T$,
$\lim_{t\to \infty} T_{\ms E} (t) = \infty$ almost surely. Therefore,
the random path $\{Y_\epsilon (t) : t\ge 0\}$, given by $Y_\epsilon(t)
= \widehat X_\epsilon (S_{\ms E} (t))$, is well defined for all $t\ge
0$ and takes value in the set $\ms E$.  We call the process
$\{Y_\epsilon(t) : t\ge 0\}$ the trace of $\{ \widehat X_\epsilon(t) :
t\ge 0\}$ on the set $\ms E$.

Denote by $\{\ms F^0_t : t\ge 0\}$ the natural filtration of
$C(\bb R_+, \bb T)$: $\ms F^0_t = \sigma(X_s: 0\le s\le t)$.  Fix
$\theta_0\in \ms E$ and denote by $\{\ms F_t : t\ge 0\}$ the usual
augmentation of $\{\ms F^0_t : t\ge 0\}$ with respect to
$\bb P^\epsilon_{\theta_0}$. We refer to Section III.9 of \cite{rw}
for a precise definition.

\begin{lemma}
\label{as15}
For each $t\ge 0$, $S_{\ms E} (t)$ is a stopping time with respect to
the filtration $\{\ms F_t\}$. Let $\{\ms G_r : r\ge 0\}$ be the filtration
given by $\ms G_r = \ms F_{S_{\ms E}(r)}$, and let $\tau$ be a stopping time
with respect to $\{\ms G_r\}$. Then, $S_{\ms E}(\tau)$ is a stopping
time with respect to $\{\ms F_t\}$.
\end{lemma}

\begin{proof}
Fix $t\ge 0$ and $r\ge 0$. By \eqref{60},
\begin{equation*}
\{S_{\ms E} (t) \le r\} \;=\; \bigcap_{q} \, \{S_{\ms E} (t) < r + q\}
\;=\; \bigcap_{q} \, \{T_{\ms E} (r + q)>t \}\;,
\end{equation*}
where the intersection is carried out over all
$q\in (0,\infty)\cap \bb Q$. By definition of $T_{\ms E}$,
$\{T_{\ms E} (r + q)>t \}$ belongs to $\ms F_{r+q}$. Hence, as the
filtration is right-continuous,
$\{S_{\ms E} (t) \le r\} \in \cap_q \, \ms F_{r+q} = \ms F_{r}$, which
proves the first assertion.

Fix a stopping time $\tau$ with respect to the filtration
$\{\ms G_r\}$. This means that for every $t\ge 0$,
$\{\tau\le t\} \in \ms G_t = \ms F_{S_{\ms E}(t)}$. Hence, for all
$r\ge 0$,
\begin{equation*}
\{\tau\le t\} \,\cap\, \{S_{\ms E}(t) \le r \} \,\in\,  \ms F_{r}\;.
\end{equation*}

We claim that $\{S_{\ms E}(\tau) <t \} \,\in\,  \ms F_{t}$. Indeed, by
\eqref{60}, this event is equal to $\{T_{\ms E} (t)> \tau \}$, which
can be written as
\begin{align*}
& \bigcup_q\, \{\tau \le q \} \, \cap \, \{T_{\ms E} (t) > q \}
\;=\; \bigcup_q\, \{\tau \le q \} \, \cap \, \{S_{\ms E} (q)<t \} \\
&\qquad \;=\; \bigcup_q\, \bigcup_{n\ge 1}
\{\tau \le q \} \, \cap \, \{S_{\ms E} (q) \le t - (1/n)\}\;,
\end{align*}
where the union is carried over all $q\in \bb Q$. By the penultimate
displayed equation, each term belongs to $\ms F_{t-(1/n)} \subset \ms
F_t$, which proves the claim. 

We may conclude. Since 
\begin{equation*}
\{S_{\ms E}(\tau) \le t \} \;=\; \bigcap_{q} \, \{S_{\ms E} (\tau) < t + q\}\;,
\end{equation*}
where the intersection is carried out over all
$q\in (0,\infty)\cap \bb Q$, and since the filtration $\{\ms F_t\}$ is
right continuous, by the previous claim, $\{S_{\ms E}(\tau) \le t \}
\in \ms F_t$.
\end{proof}

Since $S_{\ms E} (t)$ is a stopping time with respect to the
filtration $\{\ms F_t\}$,
\begin{equation*}
Y_\epsilon(t) \;=\; \widehat X_\epsilon (S_{\ms E} (t))\;.
\end{equation*}
is an $\ms E$-valued, Markov process with respect to the filtration
$\ms G_t = \ms F_{S(t)}$.  Let $\Psi: \ms E \to \bb G_{\bs n}=\{1,
\dots, \bs n\}$ be the projection given by
\begin{equation*}
\Psi(\theta) \;=\; \sum_{j=1}^{\bs n} j \, \chi_{\ms E_j} (\theta) \;,
\end{equation*}
and denote by $\bs x_\epsilon(t)$ the projected process which takes
value in $S= \{1, \dots, \bs n\}$ and is defined by
\begin{equation*}
\bs x_\epsilon(t) \;=\; \Psi(Y_\epsilon(t))\;.
\end{equation*}
Denote by $\bb Q^\epsilon_\theta$, $\theta\in \ms E$, the probability
measure on $D(\bb R_+, \ms E)$ induced by the process $Y_\epsilon(t)$
starting from $\theta$, and by $\mb Q^\epsilon_\theta$ the probability
measure on $D(\bb R_+, S)$ induced by the function $\Psi$: $\mb
Q^\epsilon_\theta = \bb Q^\epsilon_\theta \circ \Psi^{-1}$.  Note that
$\mb Q^\epsilon_\theta$ corresponds to the distribution of $\bs
x_\epsilon(t)$ starting from $\Psi(\theta)$.

\begin{theorem}
\label{th01}
Fix $1\le j\le \bs n$ and $\theta_0\in \ms E_j$. The sequence of
measures $\mb Q^\epsilon_{\theta_0}$ converges, as $\epsilon \to 0$,
to the probability measure $\mb Q_j$ introduced below
\eqref{66}.
\end{theorem}

\begin{remark}
\label{rm8}
This is the first example, to our knowledge, that the reduced chain is
an non-reversible, irreducible dynamics.
\end{remark}

\smallskip\noindent{\bf \ref{sec3}.5 Tightness.}  The proof of Theorem \ref{th01}
is divided in two steps. We prove in this subsection that the sequence
$\mb Q^\epsilon_{\theta_0}$ is tight and that all its limit points
fulfill certain conditions. In the next subsection, we prove
uniqueness of limit points.

\begin{lemma}
\label{l03}
For every $1\le j\le \bs n$, $\theta_0\in \ms E_j$, the sequence of
measures $\mb Q^\epsilon_{\theta_0}$ is tight. Moreover, every
limit point $\mb Q^*$ of the sequence $\mb Q^\epsilon_{\theta_0}$ is
such that
\begin{equation*}
\mb Q^* \{ x : x(0) = j\} \;=\; 1 \quad \text{and} \quad 
\mb Q^* \{ x : x(t) \not = x(t-)\}\;=\; 0
\end{equation*}
for every $t>0$.
\end{lemma}

\begin{proof}
Fix $\theta_0\in \ms E$.  According to Aldous' criterion, we have to
show that for every $\delta>0$, $R>0$,
\begin{equation*}
\lim_{a_0\to 0} \limsup_{\epsilon \to 0} \, 
\sup \bb Q^\epsilon_{\theta_0} \big[ \, |\Psi(Y(\tau+a)) - 
\Psi(Y(\tau))| >\delta\, \big] \;=\; 0\;,
\end{equation*}
where the supremum is carried over all stopping times $\tau$ bounded
by $R$ and all $0\le a<a_0$.

By definition of the measure $\bb Q^\epsilon_{\theta_0}$ and since
$|\Psi(Y_\epsilon (\tau+a)) - \Psi( Y_\epsilon (\tau))|>\delta$
entails that
$\Psi(Y_\epsilon (\tau+a)) \not = \Psi(Y_\epsilon (\tau))$, the
probability appearing in the previous displayed equation is bounded by
\begin{equation*}
\bb P^\epsilon_{\theta_0} \big[ \Psi(X (S_{\ms E} (\tau+a))) \not = \Psi(X
(S_{\ms E} (\tau)))  \, \big]\;.
\end{equation*}
Fix $b=2a_0$ so that $b-a\ge a_0$. Decompose this probability
according to the event $\{S_{\ms E} (\tau+a) - S_{\ms E} (\tau) > b\}$
and its complement.

Suppose that $S_{\ms E} (\tau+a) - S_{\ms E} (\tau) > b$. In this
case, $ S_{\ms E} (\tau) + b < S_{\ms E} (\tau+a)$, so that
$T_{\ms E}(S_{\ms E} (\tau) + b) \le T_{\ms E}(S_{\ms E} (\tau+a))
= \tau+a$. Hence, as $T_{\ms E}(S_{\ms E} (\tau)) = \tau$,
$T_{\ms E}(S_{\ms E} (\tau) + b) - T_{\ms E}(S_{\ms E} (\tau)) \le a$,
that is,
\begin{equation*}
\int_{S_{\ms E} (\tau)}^{S_{\ms E} (\tau) + b} \chi_{\ms E}(X(s))\,
ds \;\le\; a\;. \quad\text{Equivalently}\;, \quad
\int_{S_{\ms E} (\tau)}^{S_{\ms E} (\tau) + b} \chi_{\Delta}(X(s))\,
ds \;\ge\; b - a\;.
\end{equation*}
By Lemma \ref{as15}, $S_{\ms E} (\tau)$ is a stopping time for the
filtration $\{\ms F_t\}$. Hence, by the strong Markov property and
since $X(S_{\ms E} (t))\in\ms E$ for all $t\ge 0$,
\begin{equation*}
\bb P^\epsilon_{\theta_0} \big[  S_{\ms E} (\tau+a) - S_{\ms E} (\tau)
> b\, \big] \;\le\;
\sup_{\theta\in \ms E} \bb P^\epsilon_{\theta} \Big[  
\int_{0}^{b} \chi_{\Delta}(X(s))\, ds \;\ge\; b - a \, \Big]
\;.
\end{equation*}
By Chebychev inequality and by our choice of $b$, this expression is
less than or equal to
\begin{equation*}
\frac 1{b-a} \, \sup_{\theta\in \ms E} \bb E^\epsilon_{\theta} \Big[  
\int_{0}^{b} \chi_{\Delta}(X(s))\, ds \, \Big] \;\le\;
\frac 1{a_0} \, \sup_{\theta\in \ms E} \bb E^\epsilon_{\theta} \Big[  
\int_{0}^{2a_0} \chi_{\Delta}(X(s))\, ds \, \Big]\;.
\end{equation*}
By Lemma \ref{l06}, this expression vanishes as $\epsilon\to 0$ for
every $a_0>0$.

We turn to the case $\{S_{\ms E} (\tau+a) - S_{\ms E} (\tau) \le b\}$.
On this set we have that
$\{\Psi(X (S_{\ms E} (\tau+a))) \not = \Psi(X (S_{\ms E} (\tau)))\}$
is contained in
$\{\Psi(X (S_{\ms E} (\tau)+c)) \not = \Psi(X (S_{\ms E} (\tau)))$ for
some $0\le c\le b\}$.  By Lemma \ref{as15}, since $S_{\ms E} (\tau)$
is a stopping time for the filtration $\{\ms F_t\}$ and since
$X(S_{\ms E} (t))$ belongs to $\ms E$ for all $t$,
\begin{align*}
& \bb P^\epsilon_{\theta_0} \Big[ \Psi(X (S_{\ms E} (\tau+a))) \not = \Psi(X
(S_{\ms E} (\tau)))  \,,\,
S_{\ms E} (\tau+a) - S_{\ms E} (\tau) \le b \Big] \\
&\quad \le\;
\sup_{\theta\in\ms E} \bb P^\epsilon_{\theta} \Big[ \Psi(X (c)) \not
  = \Psi(\theta) \text{ for
  some } 0\le c\le b \Big] \;.  
\end{align*}
If $\theta\in\ms E_j$, this later event corresponds to the event
$\{H(\breve{\ms E}_j) \le b\}$, where $\breve{\ms E}_j = \cup_{k\not
  =j} \ms E_k$. The supremum is thus bounded by
\begin{equation*}
\max_{1\le j\le \bs n} \sup_{\theta\in\ms E_j} \bb P^\epsilon_{\theta}
\big[\, H(\breve{\ms E}_j) \le b \,\big] \;=\;
\max_{1\le j\le \bs n} \sup_{\theta\in\ms E_j} \bb P^\epsilon_{\theta}
\big[\, H(\breve{\ms E}_j) \le 2a_0 \,\big]  \;.
\end{equation*}
By Corollary \ref{c01}, this expression vanishes as $\epsilon\to 0$
and then $a_0\to 0$. This completes the proof of the tightness.

The same argument shows that for every $t>0$, 
\begin{equation*}
\lim_{a_0\to 0} \limsup_{\epsilon \to 0} 
\mb Q^\epsilon_{\theta_0} \big[ \, x(t-a) \not = x(t) \text{ for some
} 0\le a\le a_0 \big] \;=\; 0\;.
\end{equation*}
Hence, if $\mb Q^*$ is a limit point of the sequence $\mb
Q^\epsilon_{\theta_0}$,
\begin{equation*}
\lim_{a_0\to 0} \mb Q^* \big[ \, x(t-a) \not = x(t) \text{ for some
} 0\le a\le a_0 \big] \;=\; 0\;.
\end{equation*}
This completes the proof of the second assertion of the lemma since
$\{ x : x(t) \not = x(t-)\}\subset \{x : x(t-a) \not = x(t) \text{ for
  some } 0\le a\le a_0\}$ for all $a_0>0$. The claim that $\mb Q^* \{
x : x(0) = j\} = 1$ is clear.
\end{proof}

\smallskip\noindent{\bf \ref{sec3}.6 Uniqueness of limit points.}  The
proof of the uniqueness of limit points of the sequence
$\mb Q^\epsilon_\theta$ relies on a PDE approach to metastability
\cite{et, st}.

\begin{lemma}
\label{l08}
Fix $1\le j\le \bs n$ and $\theta_0\in \ms E_j$. Let $\mb Q^*$ be a
limit point of the sequence $\mb Q^\epsilon_{\theta_0}$. Then, under
$\mb Q^*$, for every $\bs F: S \to \bb R$,
\begin{equation*}
\bs F(x(t)) \;-\; \int_0^t (\bs L \bs F)(x(s))\, ds
\end{equation*}
is a martingale.
\end{lemma}

\begin{proof}
Fix $1\le j\le \bs n$, $\theta_0\in \ms E_j$ and a function $\bs F: S
\to \bb R$.  Let $f_\epsilon: \bb T \to \bb R$ be the function given
by Proposition \ref{as14}. By this result,
\begin{equation*}
M_\epsilon(t) \;=\; f_\epsilon (\widehat X_\epsilon(t)) \;-\; \int_0^t 
(\widehat L f_\epsilon\bs)(\widehat X_\epsilon(s))\, ds \;=\;
f_\epsilon (\widehat X_\epsilon(t)) \;-\; \int_0^t 
\overline{g}_\epsilon (\widehat X_\epsilon(s))\, ds
\end{equation*}
is a martingale with respect to the filtration $\ms F_t$ and the
measure $\bb P^\epsilon_{\theta_0}$. Since $\{S_{\ms E}(t): t\ge 0\}$ are
stopping times with respect to $\ms F_t$,
\begin{equation*}
\widehat M_\epsilon(t) \;=\; M_\epsilon(S_{\ms E}(t)) \;=\;  
f_\epsilon (Y_\epsilon(t)) \;-\; \int_0^{S_{\ms E}(t)} 
\overline{g}_\epsilon (\widehat X_\epsilon(s))\, ds
\end{equation*}
is a martingale with respect to $\ms G_t$.  Since $
\overline{g}_\epsilon$ vanishes on $\ms E^c$, by a change of
variables,
\begin{equation*}
\int_0^{S_{\ms E}(t)} \overline{g}_\epsilon (\widehat X_\epsilon(s))\, ds \;=\;
\int_0^{S_{\ms E}(t)} \chi_{\ms E} (\widehat X_\epsilon(s)) \,
\overline{g}_\epsilon (\widehat X_\epsilon(s))\, ds \;=\;
\int_0^{t} \overline{g}_\epsilon (\widehat X_\epsilon(S_{\ms E}(s)))\, ds\;.
\end{equation*}
Hence,
\begin{equation*}
\widehat M_\epsilon(t) \;=\; 
f_\epsilon (Y_\epsilon(t)) \;-\; 
\int_0^{t} \overline{g}_\epsilon (Y_\epsilon (s))\, ds
\end{equation*}
is a $\{\ms G_t\}$-martingale under the measure $\bb
Q^\epsilon_{\theta_0}$.

Since $\lim_{\epsilon\to 0} r(\epsilon)=0$, $\overline{g}_\epsilon -
g$ vanishes uniformly in $\ms E$ as $\epsilon \to 0$. By Proposition
\ref{as14}, the same holds for $\widehat f_\epsilon - f$. Hence, since
$Y_\epsilon(s) \in \ms E$ for all $s\ge 0$, we may replace in the
previous equation $g_\epsilon$, $f_\epsilon$ by $g$, $f$,
respectively, at a cost which vanishes as $\epsilon\to 0$. Therefore,
\begin{equation*}
\widehat M_\epsilon(t) \;=\; 
f (Y_\epsilon(t)) \;-\; \int_0^{t} g (Y_\epsilon (s))\, ds \;+\;
 o(1)\, 
\end{equation*}
is a $\{\ms G_t\}$-martingale under the measure $\bb
Q^\epsilon_{\theta_0}$.

Since $f$ and $g$, $f (Y_\epsilon(t)) = \bs F(\bs
x_\epsilon(t))$, $g (Y_\epsilon(t)) = \bs G (\bs x_\epsilon(t))$.  By
Lemma \ref{l03}, all limit points of the sequence $\mb
Q^\epsilon_{\theta_0}$ are concentrated on trajectories which are
continuous at any fixed time with probability $1$. We may, therefore,
pass to the limit and conclude that $\bs F(x(t)) \;-\; \int_0^t (\bs L
\bs F)(x(s))\, ds$ is a martingale under $\mb Q^*$.
\end{proof}

\begin{proof}[Proof of Theorem \ref{th01}]
The assertion is a consequence of Lemma \ref{l03}, Lemma \ref{l08} and
the fact that there is only one measure $\mb Q$ on $D(\bb R_+, S)$
such that $\mb Q[x(0)=j]=1$ and 
\begin{equation*}
\bs F(x(t)) \;-\; \int_0^t (\bs L \bs F)(x(s))\, ds
\end{equation*}
is a martingale for all $\bs F : S\to\bb R$.
\end{proof}

\smallskip\noindent{\bf \ref{sec3}.7. Proof of Lemma \ref{ls01}.} We
have seen in Subsection \ref{sec3}.3 that the jump rates depend on the
position of the valley in the landscape. If the valley is the
left-most valley, it jumps only to the right. We need therefore a
notation to indicate if a point $j\in S$ is the index of a left-most
valley or not.

Recall that the wells $\ms W_j$ which belong to the same landscape are
represented as $\ms W_{a,1}, \dots , \ms W_{a,n_a}$. We may thus
associate each $j\in S$ to a pair $(a,\ell)$, where $a \in \{1, \dots,
\bs p\}$ represents the landscape and $\ell\in \{1, \dots, n_a\}$ the
position in the landscape. Hence, $S$ can also be written as
\begin{equation*}
S\;=\; \big\{ (1,1), \dots (1,n_1), \dots, (\bs p, 1) , 
\dots, (\bs p, n_{\bs p}) \big\} \;.
\end{equation*}
 
Consider the subset $S_a = \{(a,1) , \dots, (a,n_a)\}$ of $S$.  Recall
from \eqref{59} the notation $\bs \sigma[j,\,j+1]=\bs
\sigma_{j,j+1}$. For $1\le j < n_a$, the Markov chain $\bs X(t)$
defined in Section \ref{sec3}.3 jumps $(a,j)$ to $(a,j+1)$ at rate
$\{\bs \pi(a,j) \, \bs \sigma[(a,j)\,,\, (a,j+1)]\}^{-1}$ and from
$(a,j+1)$ to $(a,j)$ at rates $\{\bs \pi(a,j+1)\, \bs
\sigma[(a,j)\,,\, (a,j+1)]\}^{-1}$.  Additionally, it jumps from
$(a,n_a)$ to $(a+1,1)$. If we disregard this last jump, on the set
$S_a$, the Markov chain behaves as a reversible Markov chain whose
equilibrium state is $\bs \pi$. The additional jump from $(a,n_a)$ to
$(a+1,1)$ changes the stationary state by the multiplicative factor
$G_1(\mf m_{a,j,1})$. This is the content of the assertion below.

\begin{proof}[Proof of Lemma \ref{ls01}]
Consider a function $F\colon S\to \bb R$, and recall that a point
$j\in S$ is also represented as $(a,k)$. With this notation, $E_{\bs
  \mu}[\bs LF]$ becomes
\begin{align*}
& \sum_{a=1}^{\bs p} \bs\mu(a,n_a)\, \frac 1{\bs \pi(a,n_a) \,
  \bs \sigma[(a,n_a)\,,\, (a+1,1)]} \, \big[F (a+1,1) - F (a,n_a)\big] \\
& \quad +\; \sum_{a=1}^{\bs p} \sum_{j=1}^{n_a-1} \bs\mu(a,j)\, \frac 1{\bs \pi(a,j) \,
  \bs \sigma[(a,j)\,,\, (a,j+1)]} \, \big[F (a,j+1) - F (a,j)\big] \\
& \qquad +\; \sum_{a=1}^{\bs p} \sum_{j=1}^{n_a-1} \bs\mu(a,j+1)\, \frac 1{\bs \pi(a,j+1) \,
  \bs \sigma[(a,j)\,,\, (a,j+1)]} \, \big[F (a,j) - F (a,j+1)\big]\;.
\end{align*}
Since the first summation is performed modulo $\bs p$ and since, by
\eqref{57}, $\bs\mu(a,j)/\bs \pi(a,j) = (1/Z) G_1(\mf m_{a,j,1})$,
a change of variables yields that the first sum can be rewritten as
\begin{align*}
\frac 1Z\, \sum_{a=1}^{\bs p} \frac{G_1(\mf m_{a-1,n_{a-1},1})} 
{\bs \sigma[(a-1,n_{a-1})\,,\, (a,1)]} \, F (a,1) 
\;-\; \frac 1Z\, \sum_{a=1}^{\bs p} \frac {G_1 (\mf m_{a,n_a,1})} 
{\bs \sigma[(a,n_a)\,,\, (a+1,1)]} \, F (a,n_a)\;.
\end{align*}
By definition of $G_1$ and $\bs \sigma_{j,j+1}$, the previous ratios are
equal to $1$ and the difference becomes
\begin{equation}
\label{71}
\frac 1Z\, \sum_{a=1}^{\bs p} \, F (a,1) 
\;-\; \frac 1Z\, \sum_{a=1}^{\bs p} F (a,n_a)\;.
\end{equation}

Use the identity $\bs\mu(a,j)/\bs \pi(a,j) = (1/Z) \, G_1(\mf
m_{a,j,1})$ to rewrite the last two terms of the first displayed
formula of this proof as
\begin{align*}
& \frac 1Z\, \sum_{a=1}^{\bs p} 
\sum_{j=2}^{n_a}   F (a,j)\, 
\Big\{ \frac {G_1(\mf m_{a,j-1,1})}{\bs \sigma[(a,j-1)\,,\, (a,j)]} 
- \frac {G_1(\mf m_{a,j,1})}{\bs \sigma[(a,j-1)\,,\, (a,j)]} \Big\} \\
& \quad \;+\; 
\frac 1Z\, \sum_{a=1}^{\bs p} 
\sum_{j=1}^{n_a-1}   F (a,j)
\Big\{ \frac {G_1(\mf m_{a,j+1,1})}{\bs \sigma[(a,j)\,,\, (a,j+1)]} 
\;-\;  \frac {G_1(\mf m_{a,j,1})}{\bs \sigma[(a,j)\,,\, (a,j+1)]} \Big\}\;. 
\end{align*}
By definition of $G_1$ and $\bs \sigma$, $G_1(\mf m_{a,j-1,1}) -
G_1(\mf m_{a,j,1}) = \bs \sigma[(a,j-1)\,,\, (a,j)]$. This sum is thus
equal to
\begin{align*}
\frac 1Z\, \sum_{a=1}^{\bs p} 
\sum_{j=2}^{n_a}   F (a,j)\, 
\;-\; \frac 1Z\, \sum_{a=1}^{\bs p} 
\sum_{j=1}^{n_a-1}   F (a,j) 
\;=\; \frac 1Z\, \sum_{a=1}^{\bs p} 
\{F (a,n_a) \,-\, F (a,1)\} \;. 
\end{align*}
This term cancels \eqref{71}, which completes the proof of the
assertion. 
\end{proof}

The same proof yields the next claim, which is needed later. 

\begin{lemma}
\label{la02}
Fix a function $F: S\to \bb R$. For every $1\le a\le p$ and every
$1\le \ell\le n_a$,
\begin{align*}
F(a,\ell) \,-\, F(1,1) \; &=\;
\sum_{b=1}^{a-1} \sum_{k=1}^{n_b} (\bs LF)(b,k)\, \bs \pi(b,k)\, G_1 (\mf m_{b,k,1}) \\
\;& +\; \sum_{k=1}^{\ell-1} (\bs LF)(a,k)\, \bs \pi(a,k)\, \big[G_1
(\mf m_{a,k,1}) - G_1(\mf m_{a,\ell,1})] \;.
\end{align*}
\end{lemma}

\section{Hitting times estimates via enlarged processes}
\label{sec7}

We prove in this section an upper bound for the probability of the
transition time between wells to be small.  This estimate plays a
central role in the proof of the tightness of a sequence of metastable
processes. The argument presented below is absolutely general and does
not rely on the one-dimensionality of the process.

The argument is based on an enlargement of the process $X_\epsilon(t)$.
Fix $\gamma>0$, let $\bb T_2 = \bb T \times \{-1,1\}$, and consider
the process
$\bs X^\gamma_\epsilon(t) = (X^\gamma_\epsilon(t), \sigma(t))$ on
$\bb T_2$ whose generator $L^\gamma_\epsilon$ is given by
\begin{equation*}
(L^\gamma_\epsilon f)(\theta,\sigma) \;=\; 
(\widehat L_\epsilon f)(\theta,\sigma) \;+\; \gamma \, 
[f(\theta,-\, \sigma) - f(\theta,\sigma)]\;.
\end{equation*}
In the first term on the right hand side, the derivatives act only on
the first coordinate. The process $\bs X^\gamma_\epsilon(t)$ is named
the enlarged process. The first coordinate evolves as the original
process, while the second one, independently from the first, jumps from
$\pm 1$ to $\mp 1$ at rate $\gamma$.

Denote by $\bb P^{\gamma,\epsilon}_{(\theta, \sigma)}$ the probability
measure on $D(\bb R_+, \bb T_2)$ induced by the Markov process
$\bs X^\gamma_\epsilon$ starting from $(\theta, \sigma)$. It is clear
that the measure $\mu^\gamma_\epsilon$, given by
\begin{equation*}
\int_{\bb T_2} f(\theta, \sigma)\, \mu^\gamma_\epsilon (d\theta,
d\sigma) \;=\; \frac 12\, \int_{\bb T} f(\theta, 1)\, 
\mu_\epsilon (d\theta) \;+\;
\frac 12\, \int_{\bb T} f(\theta, -1)\, 
\mu_\epsilon (d\theta)\;,
\end{equation*}
is the unique stationary state of the process $\bs X^\gamma_\epsilon$.

Fix an open interval $I$ of $\bb T$, and let
$(I^c,1) = \{(\theta, \sigma) \in \bb T_2 : \theta\in I^c \,,\,
\sigma=1\}$,
$(\bb T,-1) = \{(\theta, \sigma) \in \bb T_2 : \sigma=-1\}$.  Denote
by $h_I:\bb T \to \bb R_+$ the equilibrium potential between $(I^c,1)$ and
$(\bb T,-1)$:
\begin{equation}
\label{63}
h_I(\theta) \;=\; \bb P^{\gamma,\epsilon}_{(\theta, 1)} \big[
H_{(I^c,1)} \le H_{(\bb T,-1)} \big]\;,
\end{equation}
and by $\Cap_{\gamma,\epsilon} [ (I^c,1), (\bb T,-1)]$ the capacity
between the sets $(I^c,1)$, $(\bb T,-1)$, which is given by the energy
of $h_I$:
\begin{equation}
\label{62}
\Cap_{\gamma,\epsilon} [ (I^c,1), (\bb T,-1)] \;=\;
\frac 1{2} \, \epsilon\, e^{H/\epsilon}\, \int_I (\partial_\theta
h_I(\theta))^2 \, \mu_\epsilon(d\theta) 
\;+\; \frac\gamma{2}\, \int_I h_I(\theta)^2 \, \mu_\epsilon(d\theta) \;.
\end{equation}

\begin{proposition}
\label{p01}
Let $I$ be an open interval of $\bb T$, $\theta\in I$. Then, for every
$A>0$ , $\mf m\in I$ and $\eta>0$ such that
$(\mf m-\eta, \mf m+\eta) \subset I$,
\begin{align*}
\bb P^\epsilon_\theta \big[ H_{I^c} \le A \big] \; & \le\; 
\sup_{\mf m-\eta \le\theta'\le \mf m+\eta} \bb P^\epsilon_{\theta} \big[ H_{I^c} <
H_{\theta'} \big] \\
& +\;  \frac {2\, e\, A}{\mu_\epsilon(\mf m-\eta, \mf
  m+\eta)}\, \Cap_{\gamma,\epsilon} [ (I^c,1), (\bb T,-1)] \;,
\end{align*}
where $\gamma= A^{-1}$.
\end{proposition}

\begin{remark}
\label{rmi1}
We will select later $I$ as a valley and $\mf m$ as a minimum in $I$.
\end{remark}

\begin{remark}
\label{rm4}
Let $I$ be an open interval of $\bb T$.  The same arguments show that
for every $A>0$ , $J\subset I$,
\begin{align*}
\frac 1{\mu_\epsilon (J)} \int_J \bb P^\epsilon_\theta \big[ H_{I^c} \le A
  \big] \, \mu_\epsilon (d\theta)
\;\le\;   \frac {2\, e\, A}{\mu_\epsilon(J)}\, 
\Cap_{\gamma,\epsilon} [ (I^c,1), (\bb T,-1)] \;,
\end{align*}
where $\gamma= A^{-1}$. But the proof does not use the
one-dimensionality of the process, that is, the fact that the process
visit points. We leave the proof of this remark to the reader.
\end{remark}

The proof of Proposition \ref{p01} relies on an idea taken from
\cite{bl9}, and it is divided in several assertions.

\begin{asser}
\label{l04}
Let $I$ be an open interval of $\bb T$, $\theta\in I$.  Fix $A>0$, and
let $\mf e_A$ be a mean-$A$, exponential random variable independent
of the process $\widehat X_\epsilon$. Then,
\begin{equation*}
\bb P^\epsilon_\theta \big[ H_{I^c} \le A \big] \;\le\; e\;
\bb P^\epsilon_\theta \big[ H_{I^c} \le \mf e_A \big] \;. 
\end{equation*}
\end{asser}

\begin{proof}
By independence, if $\gamma=A^{-1}$,
\begin{equation*}
\bb P^\epsilon_\theta \big[ H_{I^c} \le \mf e_A \big] \;\ge\;
\int_A^\infty \bb P^\epsilon_\theta \big[ H_{I^c} \le t \big] 
\, \gamma \, e^{-\gamma t}\, dt \;\ge\;
\bb P^\epsilon_\theta \big[ H_{I^c} \le A \big]  \int_A^\infty 
\, \gamma \, e^{-\gamma t}\, dt \;,
\end{equation*}
as claimed.
\end{proof}

To estimate $\bb P^\epsilon_\theta [ H_{I^c} \le \mf e_A ]$ in terms
capacities, we interpret the exponential time $\mf e_A$ as the time
the process $\bs X^\gamma_\epsilon (t)$ starting from $(\theta, 1)$
jumps to $(\bb T, -1)$ provided $\gamma = A^{-1}$. Indeed, since the
second coordinate jumps at rate $\gamma$, independently from the first
one, for any open interval $I$ of $\bb T$ and any $\theta\in I$,
\begin{equation}
\label{61}
\bb P^\epsilon_\theta \big[ H_{I^c} \le \mf e_A \big] \;=\;
\bb P^{\gamma,\epsilon}_{(\theta, 1)} \big[ H_{(I^c,1)} \le H_{(\bb T,-1)} \big]
\;=\; h_I(\theta)\;.
\end{equation}

\begin{asser}
\label{l05}
Let $I$ be an open interval of $\bb T$, $\theta\in I$. Then,
\begin{equation*}
\gamma\, \int_I h_I(\theta)\, \mu_\epsilon(d\theta) \;=\;
2\, \Cap_{\gamma,\epsilon} [ (I^c,1), (\bb T,-1)]\;.
\end{equation*}
\end{asser}

\begin{proof}
The function $H(\theta,\sigma) = h_I(\theta) \mb 1\{\sigma =1\}$ is
harmonic on $(I,1)$, so that $L^\gamma_\epsilon H=0$ on this
set. Multiplying this identity by $1-h_I$, integrating over $(I,1)$
with respect to  $\mu^\gamma_\epsilon$, and integrating by parts
yields that 
\begin{align*}
& 0\;=\; \epsilon \, e^{H/\epsilon}\, \int_I (h_I')^2 m_\epsilon\, d\theta
\,-\, \epsilon \, e^{H/\epsilon} \, \int_I (1-h_I)\,   h_I'\,  
  m_\epsilon'\, d\theta \\
&\quad \,+\,  e^{H/\epsilon}\, \int_I (1-h_I)\ h_I'\, b\,  m_\epsilon\, d\theta
\,-\, \gamma\, \int_I (1-h_I)\, h_I\,  m_\epsilon\, d\theta\;.
\end{align*}
By \eqref{12},
$\epsilon\,  m_\epsilon' - b\, m_\epsilon$ is equal to a
constant, denoted below by $-\epsilon R_\epsilon$. Hence, if
$I=(u,v)$, the sum of the second and third terms of the
previous equation is equal to
\begin{equation*}
\epsilon R_\epsilon\, e^{H/\epsilon} \, \int_I (1-h_I)\,  h_I' \,
d\theta \;=\; -\frac 12 \, \epsilon R_\epsilon\, e^{H/\epsilon} \, \big\{
[1-h_I(v)]^2 - [1-h_I(u)]^2 \big\}\;=\;0\;,
\end{equation*}
because $h_I(u)=h_I(v)=1$. This proves the assertion in view of
formula \eqref{62} for the capacity.
\end{proof}

In the next assertion we take advantage of working in a
one-dimensional space. More precisely, although the next statement is
correct in higher dimension, it is empty since the first term on the
right-hand side is equal to $1$.

\begin{asser}
\label{l07}
Let $I$ be an open interval of $\bb T$.  For every $\theta$,
$\theta'\in I$, $A>0$,
\begin{equation*}
\bb P^\epsilon_\theta \big[ H_{I^c} \le A \big] \;\le\;
\bb P^\epsilon_\theta \big[ H_{I^c} < H_{\theta'} \big] \;+\; 
\bb P^\epsilon_{\theta'} \big[ H_{I^c} \le A \big] \;. 
\end{equation*}
\end{asser}

\begin{proof}
Intersect the set $\{H_{I^c} \le A\}$ with the event $\{H_{I^c} <
H_{\theta'}\}$ and its complement. The first set appears on the right
hand side. On the other hand, on $\{H_{\theta'} < H_{I^c} \}$,
$H_{I^c} = H_{I^c} \circ \vartheta (H_{\theta'}) + H_{\theta'}$, where
$\vartheta (t)$ represents the translation of a trajectory by $t$. In
particular, $H_{I^c} \le A$ implies that $H_{I^c} \circ \vartheta
(H_{\theta'}) \le A$. Hence, by the strong Markov property,
\begin{align*}
& \bb P^\epsilon_{\theta} \big[ H_{\theta'} < H_{I^c} \,,\, H_{I^c} \le A \big]  
\;\le\; \bb P^\epsilon_{\theta} \big[ H_{\theta'} < H_{I^c} \,,\, 
H_{I^c} \circ \vartheta (H_{\theta'})\le A \big]  \\
&\quad \;=\; \bb E^\epsilon_{\theta} \Big[ \mb 1\{H_{\theta'} < H_{I^c}\} \, 
\bb P^\epsilon_{\theta'} \big[ H_{I^c} \le A \big] \, \Big] \;\le\;
\bb P^\epsilon_{\theta'} \big[ H_{I^c} \le A \big] \;.
\end{align*}
This proves the claim.
\end{proof}

We are now in a position to prove Proposition \ref{p01}.

\begin{proof}[Proof of Proposition \ref{p01}.]
By Assertion \ref{l07},
\begin{align*}
& \bb P^\epsilon_\theta \big[ H_{I^c} \le A \big] \;\le\; 
\sup_{\mf m-\eta \le\theta'\le \mf m+\eta} \bb P^\epsilon_{\theta} \big[ H_{I^c} <
H_{\theta'} \big] \\
&\qquad \;+\; \frac 1{\mu_\epsilon(\mf m-\eta, \mf m+\eta)}
\int_{\mf m-\eta}^{\mf m+\eta}
\bb P^\epsilon_{\theta'} \big[ H_{I^c} \le A \big] 
\, m_\epsilon(\theta')\, d\theta'\;. 
\end{align*}
By Assertions \ref{l04} and \ref{l05} and \eqref{61}, the integral
appearing in the second term is less than or equal to
\begin{align*}
e\; \int_{\mf m-\eta}^{\mf m+\eta}
\bb P^\epsilon_{\theta'} \big[ H_{I^c} \le \mf e_A \big] 
\, m_\epsilon(\theta')\, d\theta'
\; & \le\; e\; \int_{I} \bb P^\epsilon_{\theta'} \big[ H_{I^c} \le \mf e_A \big] 
\, m_\epsilon(\theta')\, d\theta' \\
& \le\; \frac {2e}{\gamma} \,
\Cap_{\gamma,\epsilon} [ (I^c,1), (\bb T,-1)] \;. 
\end{align*}
which completes the proof of the proposition.
\end{proof}

We apply Proposition \ref{p01} to the case in which the interval $I$
is a valley $\ms W_j$, $1\le j\le \bs n$, introduced in Section
\ref{sec3}.1. Denote by $d(\theta, J)$ the distance from $\theta$ to a
subset $J$ of $\bb T$: $d(\theta, J) = \inf\{d(\theta, \theta') :
\theta'\in J\}$, where $d$ represents the distance in the
torus. Recall from Section \ref{sec3}.1 the definition of a well $\ms
E_j$.

\begin{corollary}
\label{c01}
Fix $1\le j\le \bs n$. Then, 
\begin{equation*}
\lim_{a \to 0} \limsup_{\epsilon \to 0} 
\sup_{\theta\in \ms E_j} \bb P^\epsilon_\theta 
\big[ H_{\ms W_j^c} \le a \big] \; =\;0\;.
\end{equation*}
\end{corollary}

\begin{proof}
Let $\mf m = \mf m_{j,1}$ and fix $\eta>0$ such that $(\mf m - \eta,
\mf m+\eta)\subset \ms E_j$.  We need to estimate the two terms which
appear on the right hand side of Proposition \ref{p01}. On the one
hand, it follows from the explicit formulae for the equilibrium
potential derived in Section \ref{sec1} that
\begin{equation*}
\limsup_{\epsilon \to 0}  \sup_{\theta\in\ms E_j} \,
\sup_{\mf m-\eta \le\theta'\le \mf m+\eta} \bb P^\epsilon_\theta 
\big[ H_{\ms W_j^c}  < H_{\theta'} \big] \; =\;0\;.
\end{equation*}

On the other hand, since $V(\mf m)=0$, there exists a constant
$c(\eta)>0$, independent of $\epsilon$, such that
$\mu_\epsilon (\mf m - \eta, \mf m + \eta) \ge c(\eta)$.
It remains, therefore, to show that
\begin{equation}
\label{64}
\lim_{a \to 0} \limsup_{\epsilon \to 0} \,
a \; \Cap_{\gamma,\epsilon} [ (\ms W_j^c,1), (\bb T,-1)] \;,
\end{equation}
where $\gamma = a^{-1}$.

Since the process is interrupted as it reaches the boundary of the
valley $\ms W_j$, it evolves as a reversible process, and all
computations can be performed with respect to this later one.

On the set of functions $f: I \to \bb R$ which are equal to $1$ at the
boundary of $I$, the energy which appears on the right hand side of
\eqref{62} is minimized by the equilibrium potential $h_I$ introduced
in \eqref{63}. Hence, in order to prove \eqref{64}, it is enough to
exhibit a function $f_\epsilon: \ms W_j \to \bb R$ which is equal to
$1$ at the boundary of $\ms W_j$ and such that
\begin{equation*}
\lim_{a \to 0} \limsup_{\epsilon \to 0} \Big\{ a \, e^{H/\epsilon}\,  
\epsilon \, \int_{\ms W_j} (\partial_\theta f_\epsilon (\theta))^2 \, m_\epsilon(\theta) 
\, d\theta \;+\; \int_{\ms W_j} f_\epsilon (\theta)^2 \, m_\epsilon(\theta) 
\, d\theta \Big\} \;=\;0\;.
\end{equation*}

Let $f_\epsilon: \ms W_j \to \bb R$ be the continuous function defined
by
\begin{equation*}
f_\epsilon (\theta) \;=\; \frac{\int_{\mf m}^\theta e^{S(y)/\epsilon}
  \, dy }{\int_{\mf m}^{\mf w^+_j} e^{S(y)/\epsilon} \, dy} \,
\chi_{[\mf m , \mf w^+_j]} (\theta) \;+\; \frac{\int_\theta^{\mf m} e^{S(y)/\epsilon}
  \, dy }{\int_{\mf w^-_j}^{\mf m} e^{S(y)/\epsilon} \, dy} \,
\chi_{[\mf w^-_j, \mf m ]} (\theta) \;.
\end{equation*}
Note that this function is not differentiable at $\mf m = \mf m_{j,1}$.
It is easy to check that this test function fulfills the condition
introduced in the penultimate displayed equation, which completes the
proof of the corollary.
\end{proof}

We turn to an estimate for the time spent outside the wells. Recall
from Section \ref{sec3}.4 the definition of the speeded-up processes
$\widehat X_\epsilon(\cdot)$.

\begin{lemma}
\label{l06}
For every $1\le j\le \bs n$, $t>0$, 
\begin{equation*}
\lim_{\epsilon \to 0}
\sup_{\theta\in \ms E_j} 
\bb E^\epsilon_{\theta} \Big[ \int_{0}^{t} 
\chi_{\Delta}(\widehat X_\epsilon(s))\, ds
\Big] \;=\; 0 \;. 
\end{equation*}
\end{lemma}

\begin{proof}
Fix $\eta>0$ and let $\ms E_j^{(\eta)} = \{\theta\in\ms E_j :
d(\theta, \Delta)>\eta\}$. The time integral appearing in the
statement of the lemma is bounded above by
\begin{equation*}
H_{\ms E_j^{(\eta)}} \;+\; \int_{H_{\ms E_j^{(\eta)}}}^{t+H_{\ms
    E_j^{(\eta)}}} \chi_{\Delta}(\widehat X_\epsilon(s))\, ds\;.
\end{equation*}
As observed in Section \ref{sec1}, in one-dimension diffusions visit
points and one can compute the capacity between singletons and sets.
It follows from the proof of Proposition 3.3 in \cite{lms} that
\begin{equation*}
\bb E^\epsilon_{\theta} \big[ \, H_{\ms E_j^{(\eta)}} (\widehat X_\epsilon(\cdot))\,\big] 
\;=\; \frac {e^{-H/\epsilon}}{\Cap_\epsilon (\{\theta \} , \ms
  E_j^{(\eta)})} \int_{\bb T} h^*_{\{\theta \} , \ms
  E_j^{(\eta)}} (\theta') \, m_\epsilon (\theta') \, d\theta' \;,
\end{equation*}
where the factor $e^{-H/\epsilon}$ appeared because the process
$\widehat X_\epsilon(s)$ has been speeded-up by $e^{H/\epsilon}$. In
this formula, $h^*_{\{\theta \} , \ms E_j^{(\eta)}}$ represents the
equilibrium potential between $\{\theta \}$ and $\ms E_j^{(\eta)}$ for
the adjoint process. Therefore,
\begin{equation*}
\bb E^\epsilon_{\theta} \big[ \, H_{\ms E_j^{(\eta)}} (\widehat X_\epsilon(\cdot))\,\big] 
\;\le\; \frac {e^{-H/\epsilon}}{\Cap_\epsilon (\{\theta \} , \ms E_j^{(\eta)})}\;.
\end{equation*}
As in Section \ref{sec1}, it is possible to derive an explicit formula
for this capacity and to show that this expression vanishes as
$\epsilon\to 0$, uniformly in $\theta\in \ms E_j$.

By the strong Markov property, it remains to show that for every
$1\le j\le \bs n$, $t>0$, $\eta>0$,
\begin{equation*}
\lim_{\epsilon \to 0}
\sup_{\theta\in \ms E_j^{(\eta)}} 
\bb E^\epsilon_{\theta} \Big[ \int_{0}^{t} 
\chi_{\Delta}( \widehat X_\epsilon(s))\, ds
\Big] \;=\; 0 \;.
\end{equation*}

Let $\eta'>0$ be such that
$\mf m_{j,1}\in \ms E^{(2\eta')}_j$.  It follows
from the explicit formulae for the equilibrium potentials computed in
Section \ref{sec1} that
\begin{equation*}
\lim_{\epsilon \to 0} \sup_{\theta\in \ms E^{(\eta)}_j} \,
\sup_{\theta'\in \ms E^{(\eta')}_j} \, 
\bb P^\epsilon_{\theta} \big[ H_{\Delta} < H_{\theta'} \big] \;=\; 0 \;. 
\end{equation*}
Since the time integral appearing in the statement of the lemma is
bounded, it follows from the previous estimate that we may insert
inside the expectation the indicator of the set
$\{H_{\theta'} < H_{\Delta}\}$. Since we may start the time
integral from $H_{\Delta}$, on the set
$\{H_{\theta'} < H_{\Delta} \}$
\begin{equation*}
\int_{0}^{t}  \chi_{\Delta}( \widehat X_\epsilon(s))\, ds \;=\;
\int_{H_{\theta'}}^{t}  \chi_{\Delta}(\widehat X_\epsilon(s))\, ds \;\le\;
\int_{0}^{t}  \chi_{\Delta}( \widehat X_\epsilon(s))\, ds \,\circ\,
\vartheta(H_{\theta'})\;. 
\end{equation*}
Hence, by the strong Markov property,
\begin{equation*}
\bb E^\epsilon_{\theta} \Big[ \mb 1\{H_{\theta'} < H_{\Delta} \} \,
\int_{0}^{t}  \chi_{\Delta}(\widehat X_\epsilon(s))\, ds \Big] \;\le\;
\bb E^\epsilon_{\theta'} \Big[ 
\int_{0}^{t}  \chi_{\Delta}(\widehat X_\epsilon(s))\, ds \Big]\;.
\end{equation*}
Note that the starting point changed from $\theta$ to $\theta'$. 

In view of the previous bounds, to prove the lemma it is enough to
show that
\begin{equation*}
\lim_{\epsilon \to 0} \frac 1{\mu_\epsilon ([\mf m_{j,1}-\eta', \mf m_{j,1}+\eta'])}\,
\int_{\mf m_{j,1}-\eta'}^{\mf m_{j,1}+\eta'} 
\bb E^\epsilon_{\theta'} \Big[ 
\int_{0}^{t}  \chi_{\Delta}(\widehat X_\epsilon(s))\, ds \Big] 
\, \mu_\epsilon (d\theta') \;=\; 0 \;. 
\end{equation*}
Since $V(\mf m_{j,1})=0$, there exists a constant $c(\eta')>0$, independent
of $\epsilon$, such that
$\mu_\epsilon ([\mf m_{j,1} - \eta', \mf m_{j,1} + \eta']) \ge
c(\eta')$. On the other hand, the integral is bounded by
\begin{equation*}
\int_{\bb T}  \bb E^\epsilon_{\theta'} \Big[ 
\int_{0}^{t}  \chi_{\Delta}(\widehat X_\epsilon(s))\, ds \Big] 
\, \mu_\epsilon (d\theta') \;=\; t\, \mu_\epsilon (\Delta)\;,
\end{equation*}
which vanishes as $\epsilon\to 0$.
\end{proof}

\section{The Poisson equation}
\label{sec5}

We examine in this section properties of the solution of the equation
$\widehat L_\epsilon f = g$ for a function $g:\bb T\to \bb R$ which
has mean zero with respect to $\mu_\epsilon$.  We assume in this
section the conditions ({\bf H4}) of Subsection \ref{sec0}.4. Recall
the notation introduced in Section \ref{sec3}, and that we denote by
$\mf w^\pm_i$ the endpoints of the well $\ms W_i$. Throughout this
section, we assume, without loss of generality, that $\ms W_1$ is the
left-most valley of a landscape: $\ms W_1 = \ms W_{1,1}$ in the
notation introduced in the paragraph below \eqref{59}. Therefore,
there exists $\eta>0$ such that
\begin{equation}
\label{72}
S(x) \;\ge\; S(\mf m_{1,1}) \;+\; \eta \quad\text{for all}\quad
-\infty < x\le \mf w^-_1\;.
\end{equation}

Fix a function $\bs F \colon S \to \bb R$, and let $\bs G = \bs L \bs
F$.  Denote by $g\colon \bb T\to\bb R$ the function given by
\begin{equation*}
g \;=\; \sum_{1\le i \le \bs n} \bs G(i) \, \chi_{\ms E_i}\;.  
\end{equation*}

\begin{asser}
\label{as17}
 We have that
$\lim_{\epsilon \to 0} E_{\mu_{\epsilon}}[g] = 0$.
\end{asser}

\begin{proof}
By definition of the function $g$,
\begin{equation*}
E_{\mu_{\epsilon}}[g] \;=\; \sum_{i=1}^{\bs n} \bs G(i) \,
\mu_{\epsilon}(\ms E_i)\;. 
\end{equation*}
Fix $1\le i\le \bs n$.  By definition of $\mu_{\epsilon}$, by
Proposition \ref{p02} and since, by Remark \ref{rm5}, $G_1$ is constant on
valleys,
\begin{equation*}
\mu_{\epsilon}(\ms E_i) \;=\; [1+o(1)] \, \frac 1{Z} \; \frac 1{\sqrt{\epsilon}}
\; G_1(\mf m_{i,1}) \, \int \chi_{\ms E_i} (x)\, e^{ -V(x)/
  \epsilon}\, dx\;.
\end{equation*}
By Remark \ref{rm6}, $V$ and $S$ differ by an additive constant on
valleys. Hence, since $V(\mf m_{i,k}) = 0$, the previous expression is
equal to
\begin{equation*}
[1+o(1)] \, \frac 1{Z} \; \; G_1(\mf m_{i,1}) \, 
\sum_{k=1}^{\kappa(i)} \sigma (\mf m_{i,k})\; = \; 
[1+o(1)] \, \bs \mu (i)\;,
\end{equation*}
where the last identity follows from the definition of $\bs \mu$ given
in \eqref{57}. 

In conclusion,
\begin{equation*}
E_{\mu_{\epsilon}}[g] \;=\; [1+o(1)] \,
\sum_{i=1}^{\bs n} \bs G(i) \, \bs \mu (i)
\;=\; [1+o(1)] \,
\sum_{i=1}^{\bs n} (\bs L \bs F)(i) \, \bs \mu (i)\;. 
\end{equation*}
To complete the proof, it remains to recall the statement of Lemma \ref{ls01}
\end{proof}

Let  $\overline{g}_\epsilon \colon \bb T\to\bb R$ be given by 
\begin{equation*}
\overline{g}_\epsilon \;=\; g \;-\;  r(\epsilon) \, \chi_{\ms E_1} \;,
\end{equation*}
where $r(\epsilon) = E_{\mu_\epsilon}[g]/\mu_\epsilon(\ms E_1)$ and
$E_{\mu_\epsilon}[g]$ represents the expectation of $g$ with respect
to $\mu_\epsilon$. Clearly,
$E_{\mu_\epsilon}[\overline{g}_\epsilon]=0$, and, by Assertion
\ref{as17} and \eqref{57}, $\lim_{\epsilon} r(\epsilon)=0$. The
following proposition is the main result of this section.

\begin{proposition}
\label{as14}
Let $f_\epsilon: [\mf w^-_1,1+\mf w^-_1]\to\bb R$ be
the function given by
\begin{equation}
\label{73}
\begin{aligned}
f_\epsilon (x) \; & =\; \bs F(1) \;+\; a(\epsilon) 
\int_{\mf w^-_1}^x e^{S(y)/\epsilon}\, dy \\
\; &+\; \frac 1{\epsilon}\, e^{-H/\epsilon}\, 
\int_{\mf w^-_1}^x e^{S(y)/\epsilon}\, \int_{\mf w^-_1}^y 
\overline{g}_\epsilon (z) \, e^{-S(z)/\epsilon}\, dz\, dy\;,
\end{aligned}
\end{equation}
where
\begin{equation*}
a(\epsilon) \;=\; \frac 1{e^{B/\epsilon} -1} \, \frac 1{\epsilon}\,
e^{-H/\epsilon}\, \int_{\mf w^-_1}^{1+\mf w^-_1} 
\overline{g}_\epsilon (y) \, e^{-S(y)/\epsilon}\, dy\;.
\end{equation*}
Then, $f_\epsilon$ is $1$-periodic, and solves the elliptic problem
$\widehat L_\epsilon f_\epsilon = \overline{g}_\epsilon$ in $\bb T$. Moreover, there
exists a finite constant $C_0$ such that
\begin{equation*}
\sup_{0<\epsilon<1}\, \sup_{\theta\in\bb T} |f_\epsilon(\theta)| \;\le\; C_0\;,
\quad\text{and} \quad
\lim_{\epsilon\to 0} 
\sup_{\theta\in \ms E} \big|\, f_\epsilon(\theta) 
- f(\theta)\, \big| \;=\; 0\;,
\end{equation*}
where $f: \bb T \to \bb R$ is given by $f = \sum_{1\le i \le \bs n} \bs
F(i) \, \chi_{\ms E_i}$.
\end{proposition}

The proof of this proposition is divided in several steps. In the next
lemma, we show that the function $f_\epsilon$ is $1$-periodic and
solves the Poisson equation.

\begin{lemma}
\label{l10}
Let $g: \bb T \to \bb R$ be a bounded function which has mean zero
with respect to $\mu_\epsilon$, and let $f_\epsilon : [0,1] \to \bb R$
be given by
\begin{equation*}
f_\epsilon (x) \;=\; A \;+\; a(\epsilon) \int_0^x e^{S(y)/\epsilon}\, dy
\;+\; \frac 1{\epsilon}\, \int_0^x e^{S(y)/\epsilon}\, \int_0^y g (z)
\, e^{-S(z)/\epsilon}\, dz\, dy\;,
\end{equation*}
where $A\in \bb R$ and
\begin{equation*}
a(\epsilon) \;=\; \frac 1{e^{B/\epsilon} -1} \frac 1{\epsilon}\,
\int_0^1 g (y) \, e^{-S(y)/\epsilon}\, dy\;.
\end{equation*}
Then, $f_\epsilon$ solves the elliptic problem $L_\epsilon f = g$ in
$\bb T$.
\end{lemma}

\begin{proof}
We have to show that $(L_\epsilon f_\epsilon)(x) = g (x)$ for all
$x\in (0,1)$ and that $f'_\epsilon (1) = f'_\epsilon (0)$, $f_\epsilon
(1) = f_\epsilon (0)$. The first two properties are straightforward.
The third one is proved in Assertion \ref{as16} below.
\end{proof}

\begin{asser}
\label{as16}
We claim that $f_\epsilon (1) = f_\epsilon (0)$. 
\end{asser}

\begin{proof}
In view of its definition, $\epsilon  [f_\epsilon (1) - f_\epsilon
(0)]$ is equal to
\begin{equation*}
\epsilon \, a(\epsilon) \int_0^1 e^{S(y)/\epsilon}\, dy
\;+\; \int_0^1 e^{S(y)/\epsilon}\, \int_0^y g (z)\,
e^{-S(z)/\epsilon}\, dz\, dy\;.
\end{equation*}
Change the order of integration in the second term, and recall the
definition of $a(\epsilon)$. Since $B=S(0) - S(1)$, this expression is
equal to $[e^{B/\epsilon} -1]^{-1} e^{-S(1)/\epsilon}$ times
\begin{align*}
& e^{S(1)/\epsilon} \int_0^1 g (z) \, e^{-S(z)/\epsilon}\, dz 
\, \int_0^1 e^{S(y)/\epsilon}\, dy\\
&\quad + \; \big[\, e^{S(0)/\epsilon}- e^{S(1)/\epsilon}\, \big]
\, \int_0^1 g (z) 
\, e^{-S(z)/\epsilon}\,  \int_z^{1} e^{S(y)/\epsilon} \, dy\, dz\;.
\end{align*}
This difference is equal to
\begin{equation}
\label{69}
\begin{aligned}
& e^{S(1)/\epsilon} \int_0^1 g (z) \, e^{-S(z)/\epsilon}
\, \int_0^z e^{S(y)/\epsilon}\, dy \, dz \\
&\quad + \; e^{S(0)/\epsilon} \, \int_0^1 g (z) 
\, e^{-S(z)/\epsilon}\,  \int_z^1 e^{S(y)/\epsilon} \, dy\, dz\;.
\end{aligned}
\end{equation}
Rewrite the second integral as
\begin{equation*}
\int_0^1 g (z) \, e^{-S(z)/\epsilon}\,  \int_z^{1+z}
e^{S(y)/\epsilon} \, dy\, dz \;-\;
\int_0^1 g (z) \, e^{-S(z)/\epsilon}\,  \int_1^{1+z}
e^{S(y)/\epsilon} \, dy\, dz\;.
\end{equation*}
Note that in the first integral the density $\pi_\epsilon$
appears. Since $g$ has mean zero with respect to
$\mu_\epsilon$, the first term vanishes. In the second integral,
change variables $y=y'+1$ and recall that $S(y'+1) = S(y') -B$ to
obtain that the second integral of \eqref{69} is equal to
\begin{equation*}
-\; e^{[S(0) - B] /\epsilon} \int_0^1 g (z) \, e^{-S(z)/\epsilon}
\, \int_0^z e^{S(y)/\epsilon}\, dy \, dz\;.
\end{equation*}
Since $S(0) - B = S(1)$, the terms in \eqref{69} cancel, which
completes the proof of the assertion.
\end{proof}

\begin{proof}[Proof of Proposition \ref{as14}]
We proved in Lemma \ref{l10} that $f_\epsilon$ is $1$-periodic and
solves the elliptic equation $\widehat L f_\epsilon = 
\overline{g}_\epsilon$ in $\bb T$. It remains to show that $f_\epsilon$ is
uniformly bounded and converges uniformly to $f$ on the set $\ms E$.
We examine separately the second and third terms on the right-hand
side of \eqref{73}.

We claim that the second term vanishes as $\epsilon\to 0$, uniformly
in $x\in [\mf w^-_1, 1+\mf w^-_1]$. On the one hand, since $S(x) \le
S(\mf w^-_1)$ for all $x\ge \mf w^-_1$ [because $\mf w^-_1$ is the
left endpoint of a valley],
\begin{equation}
\label{is01}
\frac 1{\sqrt{\epsilon}}\, \int_{\mf w^-_1}^{1+\mf w^-_1} 
e^{[S(y)-S(\mf w^-_1)]/\epsilon}\, dy \;\le\; C_0
\end{equation}
for some finite constant $C_0$ independent of $\epsilon$. On the other
hand, since $|\overline{g}_\epsilon (x)|\le C_0$, $e^{B/\epsilon} -1
\ge (1-e^{-B}) e^{B/\epsilon}$ for sufficiently small$ \epsilon>0$,
and $S(1+\mf w^-_1) = S(\mf w^-_1) -B$,
\begin{equation}
\label{is02}
\begin{aligned}
& \frac 1{e^{B/\epsilon} -1} \, e^{[S(\mf w^-_1)-H]/\epsilon}\, 
\frac 1{\sqrt{\epsilon}}\, \Big|\, \int_{\mf w^-_1}^{1+\mf w^-_1} 
\overline{g}_\epsilon(y) \, e^{-S(y)/\epsilon}\, dy \,\Big| \; \\
&\qquad \le\; C_0
\, e^{[S(1+\mf w^-_1)-H]/\epsilon}\, 
\frac 1{\sqrt{\epsilon}}\, \int_{\mf w^-_1}^{1+\mf w^-_1} 
\, e^{-S(y)/\epsilon}\, dy\;. 
\end{aligned}
\end{equation}
Since $\mf w^-_1$ is the left endpoint of a valley whose depth is $H$,
$S(1+\mf w^-_1)-H = S(1+\mf m_{1,1})$. By \eqref{72}, there exists
$\eta>0$ such that $S(y) \ge S(1+\mf m_{1,1}) + \eta$ for all $\mf
w^-_1 \le y\le 1+\mf w^-_1$. Multiplying \eqref{is01} and \eqref{is02}
yields that the second term in the formula of $f_\epsilon$ vanishes
as $\epsilon\to 0$, uniformly for $x\in [\mf w^-_1 , 1+\mf w^-_1]$.

We turn to the third term on the right hand side of \eqref{73}.
Exchange the order of the integrals to write it as
\begin{equation}
\label{74}
\frac 1\epsilon \, e^{-H/\epsilon}
\int_{\mf w^-_1}^x   \overline{g}_\epsilon (z)  \int_{z}^x 
e^{[S(y)-S(z)]/\epsilon} \, dy\, dz \;.
\end{equation}
Since $\overline{g}_\epsilon$ is uniformly bounded, the absolute value
of this expression is bounded by
\begin{equation*}
\frac {C_0}{\epsilon} \, e^{-H/\epsilon}
\int_{\mf w^-_1}^{1+\mf w^-_1}  \int_{z}^{1+z} 
e^{[S(y)-S(z)]/\epsilon} \, dy\, dz \;= \frac {C_0}{\epsilon} \,
e^{-H/\epsilon} \, c(\epsilon)\;,
\end{equation*}
where $c(\epsilon)$ is the normalizing constant introduced below
\eqref{34}.  By \eqref{75}, this expression is uniformly bounded in
$\epsilon$. This proves one assertion of the proposition. It also
proves that we may replace $\overline{g}_\epsilon$ in \eqref{74} by $g$ because, by
Assertion \ref{as17}, $r(\epsilon)$ converges to $0$ as $\epsilon\to
0$.

It remains to prove the uniform convergence in $\ms E$ of the sequence
$f_\epsilon$ with $\overline{g}_\epsilon$ replaced by $g$.  As the
integral in \eqref{74} is carried over pairs $(y,z)$ such that
$z\le y$ the maximum value of the difference $S(y)-S(z)$ is $H$ and it
is attained only when $y$, $z$ belong to the same landscape and
$V(y)=H$, $V(z)=0$, that is, when $y$ is an endpoint of a valley, and
$z$ is a global minima of a valley in the same landscape. Hence, the
dominant terms of the integral are the ones in which $y$ belongs to a
neighborhood of an endpoint of a valley and $z$ to a neighborhood of a
global minima of a valley.

Recall from Subsection \ref{sec3}.2 that the valleys $\ms W_j$ which
belong to the same landscape are represented as $\ms W_{a,1}, \dots ,
\ms W_{a,n_a}$, $1\le a\le \bs p$. To prove uniform convergence in
$\ms E$, fix a point $x\in\ms E_j = \ms E_{a,\ell}$. By the
observation of the previous paragraph, the contribution to the
integral of the points $z$ which do not belong to a neighborhood of
point $\mf m_{b,k,i}$ [that is a global minimum of $V$ in the valley
$\ms W_{b,k}$] is negligible.

Fix $b<a$, $1\le k\le n_b$. The contribution to integral when $z$ belongs to the
neighborhoods of the local minima of $\ms W_{b,k}$ is given by $\bs\pi
(b,k)\, G_1(\mf m_{b,k,1})$, where $\bs \pi$ has been introduced in
\eqref{58}. For $b=a$, $1\le k\le \ell \le n_a$, the contribution to
integral of the neighborhoods of the local minima of $\ms W_{a,k}$ is
equal to $\bs\pi (a,k)\, [G_1(\mf m_{a,k,1}) - G_1(\mf
m_{a,\ell,1})]$. Hence, the integral \eqref{74} with $\overline{g}_\epsilon$
replaced by $g$ is equal to
\begin{align*}
& \sum_{b=1}^{a-1} \sum_{k=1}^{n_b} (\bs L \bs
F)(b,k)\, \bs\pi (b,k)\, G_1(\mf m_{b,k,1}) \\
&\quad \;+\; \sum_{k=1}^{\ell-1} (\bs L \bs F)(a,k)\, \bs\pi (a,k)\, 
\big[ G_1(\mf m_{a,k,1}) - G_1(\mf m_{a,\ell,1}) \big] \;+\; R(\epsilon)\;,
\end{align*}
where $R(\epsilon)$ is a remainder which converges to $0$ as
$\epsilon\to 0$, uniformly for $x\in [\mf w^-_1, 1+\mf w^-_1]$.  By
Lemma \ref{la02}, the previous sum is equal to $\bs F(a,\ell) -
\bs F(1)$, which proves that $f_\epsilon$ converges to $f$ uniformly
in $\ms E$.
\end{proof}

\end{document}